\documentclass[reqno, 11pt, a4paper]{amsart} 
\usepackage[english]{babel}
\usepackage{amsfonts, amsmath, amsthm, amssymb,amscd,indentfirst}
\usepackage{mathtools}
\usepackage{xcolor}
\usepackage{times}
\usepackage{enumerate}
\usepackage{enumitem}
\usepackage{esint}
\usepackage{stackrel}
\usepackage{anysize} 
 \usepackage[mathscr]{euscript}
\marginsize{2cm}{2cm}{1.5cm}{1.5cm}

\newtheorem{theorem}{Theorem}[section]
\newtheorem{proposition}[theorem]{Proposition}
\newtheorem{lemma}[theorem]{Lemma}
\newtheorem{definition}[theorem]{Definition}
\newtheorem{corollary}[theorem]{Corollary}

 \numberwithin{equation}{section}

 \newcommand{\dx}{\textnormal{d}x}

 \newcommand{\dt}{\textnormal{d} t}

\newcommand{\conspoin}{\mathtt{K}}
\newcommand{\cpoin}{\textnormal{n}}
\newcommand{\cones}{\mathcal{C}}
\newcommand{\conrem}{\hat{\textnormal{s}}}
\newcommand{\consman}{\hat{c}}
\newcommand{\vinf}{\mathfrak{T}}
\newcommand{\rast}{R_\ast}
 \newcommand{\bfc}{\textbf{c}}

\newcommand{\consfi}{\Lambda _1}
\newcommand{\consfii}{\Lambda _2}

 \newcommand{\tr}{\mathtt{R}}
 \newcommand{\kob}{\mathscr{K}_{\mathscr{O} _1 ,  \mathscr{O} _2} (\Omega)}
 \newcommand{\obsi}{\mathscr{O} _1}
 \newcommand{\obsii}{\mathscr{O} _2}
 \newcommand{\uu}{u}
 \newcommand{\ue}{u_{\ell}}
 \newcommand{\vuu}{\vartheta}
 \newcommand{\vue}{\vartheta _{\ell}}
\newcommand{\wo}{W^{1, p(x)}(\Omega ;  \mathtt{V} )}
\newcommand{\wob}{W^{1, p(x)} _0 (\Omega ;  \mathtt{V} )}
\newcommand{\lwo}{W^{1, p(x)}_{\operatorname{loc}}(\Omega ;  \mathtt{V} )}
\newcommand{\dhd}{\textnormal{d}h}
 \newcommand{\spt}{\operatorname{supp}}
\newcommand{\dist}{\operatorname{dist}}
\newcommand{\wghtv}{\mathtt{V}}
\newcommand{\wghtz}{\textnormal{I}}
\newcommand{\wghtzi}{\textnormal{II}}

\newcommand{\loc}{\operatorname{loc}}

\newcommand{\mec}{\delta_{\mathbb{R}^n}}

\newcommand{\pvr}{x}
\newcommand{\pvry}{y}
\newcommand{\pvrz}{z}
\newcommand{\evrx}{\textnormal{x}}
\newcommand{\evry}{\textnormal{y}}
\newcommand{\evrz}{\textnormal{z}}

\newcommand{\singset}{\mathcal{S}}
\newcommand{\grad}{\boldsymbol{\nabla}}
\newcommand{\fnf}{F}
\newcommand{\fnfs}{F ^\ast}
\newcommand{\deriv}{\mathscr{D}}
\newcommand{\dwvi}{\textnormal{d} \mathtt{V}}
\newcommand{\wghtvi}{\mathtt{V}}
\newcommand{\funcion}{f}
\newcommand{\multg}{\bullet}

\title{Removable set for Hölder continuous solutions of \(\mathscr{A}\)-harmonic functions on Finsler manifolds}

\author{Juan Pablo Alcon Apaza}

\address{Departamento de Matem\'atica, Universidade Federal de São Carlos, 13565-905, São Carlos--SP, Brazil}

\email{juanpabloalconapaza@gmail.com}

\begin{document}
\maketitle

\begin{abstract}
We establish that a closed set \( \singset \) is removable for \( \alpha \)-Hölder continuous \( \mathscr{A} \)-harmonic functions in a reversible Finsler manifold \( (\Omega, \fnf, \wghtvi) \) of dimension \( n \geq 2 \), provided that (under certain conditions on \( (\Omega, \fnf, \wghtvi) \) and the variable exponent \( p \)) for each compact subset \( K \) of \( \singset \), the \( \cpoin_1 - p^+_K + \alpha (p^+_K - 1) \)-Hausdorff measure of \( K \) is zero. Here, \( p_K^+ = \sup_K p \) and \( \cpoin_1 \) is chosen so that \( \wghtvi(B(\pvr, r)) \leq \conspoin r^{\cpoin_1} \) for every ball.

The estimates used to remove the singularities will focus on a family \( \{ \ue \}_{\ell \in \mathscr{J}} \subset \lwo \) that converges to \( u \) in a certain sense. As a second main result of this article, we will also obtain an estimate (when \( \lim_{d(\pvr , 0_{\Omega}) \rightarrow \infty} p = 1 \)) for
$$
\mu_{\ell}(B(\pvr , r)) := \sup \left\{ \int_{B(\pvr , r)} \mathscr{A}(\cdot , \grad \ue) \multg \deriv \zeta \; \dwvi \:|\: 0 \leq \zeta \leq 1 \text { and } \zeta \in C_0^{\infty}( B(\pvr , r) ) \right\},
$$
which is related to the measure \( \mu = \operatorname{div}(\mathscr{A} (\cdot , \grad u )) \).
\end{abstract}

\let\thefootnote\relax\footnote{2020 \textit{Mathematics Subject Classification}. 35J60; 35J70; 35J92; 58J05}
\let\thefootnote\relax\footnote{\textit{Keywords and phrases}. Variable exponent space; removable set; Hausdorff measure; Finsler manifold}


\section{Introduction}

In this paper, we study solutions to quasilinear elliptic equations of the form:
\begin{equation} \label{13}
-\operatorname{div} \mathscr{A} (\cdot , \grad u)=0 \quad \text { in } \quad \Omega.
\end{equation}
Throughout this work, \((\Omega, \fnf, \wghtvi)\) is a reversible Finsler manifold satisfying the properties \eqref{78} - \eqref{120} below, and the variable exponent \(p: \Omega \rightarrow (1,\infty)\) is continuous. Additionally, the operator \(\mathscr{A}: T\Omega \rightarrow T\Omega\) satisfies the following conditions for all functions \(u: \Omega \rightarrow \mathbb{R}\) that are differentiable in the distributional sense:
\begin{enumerate}[label=($a_{\arabic*}$)]
\item \label{16} $\left[ \mathscr{A}(\cdot , \grad u  )  \multg  \deriv u \right] (\pvr) \geq \consfi  | \grad u |_{\fnf} ^{p} (\pvr) $ for  a.e.  $\pvr \in \Omega$.

\item \label{21} $  |\mathscr{A}( \cdot , \grad u ) |_{\fnf} (\pvr )\leq \consfii  | \grad u | ^{p-1} _{\fnf}  (\pvr) $ for a.e.  $\pvr \in \Omega$.

\item \label{17} $\left\{ \left[  \mathscr{A} (\cdot , \grad u )-\mathscr{A}   (\cdot , \grad v  ) \right] \multg  \left(\deriv u  - \deriv v \right) \right\} (\pvr)>0$   for a.e. $\pvr \in \Omega$ whenever $\deriv u (\pvr ) \neq \deriv v (\pvr)$.
\end{enumerate}
Here, we write
$$
\fnf (\pvr , Y  ) = |Y|_{\fnf} (\pvr) \quad \text { and }  \quad \fnfs (\pvr , \omega  ) = |\omega|_{\fnfs} (\pvr),
$$
for $Y\in T_{\pvr} \Omega$ and $\omega \in T^\ast  _{\pvr}\Omega$. Furthermore, for a function \( f: \Omega \rightarrow \mathbb{R} \) that is differentiable in the distributional sense, we use the notation
$$
Y\multg \deriv f (\pvr)= \deriv f _{\pvr}(Y) \quad \text { for } \quad Y\in T_{\pvr} \Omega,
$$
where \( \deriv f_{\pvr} \in T_{\pvr}^\ast \Omega \) (or equivalently \( \deriv f(\pvr) \)) denotes the derivative of \( f \) at \( \pvr \).

For a relatively closed subset \(E \subset \Omega ^\prime \subset \mathbb{R}^n\) with \(s\)-dimensional Hausdorff measure zero (for \(s > n - p\)) and \(u\) \(p\)-harmonic in \(\Omega ^\prime \backslash E\), previous studies determined the values of \(s\) for which the extension of \(u\) is \(p\)-harmonic in all of \(\Omega ^\prime\). Classical results are given in \cite{carleson1967selected, kilpelainenzhong2001removable, koskelamartio1993removability}, and problems involving lower-order terms are discussed in \cite{hirata2011removable, ono2013removable}. In the nonstandard growth framework, removability problems were studied in variable exponent spaces by \cite{fu2015removable, latvala2011fundamental, alcon2023removable, apaza2023removable}, in Orlicz spaces by \cite{challallyaghfouri2011removable, chlebicka2021remmusielak}, and in double-phase spaces by \cite{chlebicka2020removable}. Building on these studies, we use the intrinsic capacities and Hausdorff measures introduced by \cite{baruah2018capacities} and \cite{defilippis2020manifold}, respectively.

Mäkäläinen in \cite{makalainen2008removable} studied $p$-harmonic functions in complete metric spaces $\tilde{\Omega}$, assuming the space has a doubling measure and supports a weak $(1, p)$-Poincaré inequality. To control the integrability of the derivative in the metric space setting, a substitute for Sobolev space is needed. The Newtonian space, introduced by Shanmugalingam in \cite{shanmugalingam2000newtonian}, is used. The following equation is studied:  
$$
\int_{\tilde \Omega}|D u|^{p-2} D u \cdot D \varphi \, \textnormal d \mu = 0,
$$
where $1 < p < \infty$ and $D$ denotes the derivative. The paper shows that sets of weighted $(-p + \alpha(p-1))$-Hausdorff measure zero are removable for $\alpha$-Hölder continuous Cheeger $p$-harmonic functions.


Kilpeläinen $\&$ Zhong in \cite{kilpelainenzhong2001removable} studied continuous solutions $u \in W_{\loc}^{1, p}(\Omega ^\prime )$ of the equation
$$
-\operatorname{div} A(x, \nabla u) = 0,
$$
where $\Omega ^\prime \subset \mathbb{R}^n$ is an open set and $1 < p < \infty$. These solutions are called $A$-harmonic in $\Omega ^\prime $. An example of such operators is the $p$-Laplacian
$$
-\Delta_p u = -\operatorname{div} \left( |\nabla u|^{p-2} \nabla u \right).
$$
They show that sets with $n - p + \alpha(p - 1)$ Hausdorff measure zero, where $0 < \alpha \leq 1$, are removable for $\alpha$-Hölder continuous solutions to quasilinear elliptic equations.  For the quasilinear case, Heinonen $\&$ Kilpeläinen \cite{heinonen1988superharmonic} proved similar results for $\alpha = 1$, and Trudinger $\&$ Wang \cite{trudinger2002weak} proved it under the assumption that $u$ has an $A$-superharmonic extension to $\Omega ^\prime $, which can be omitted for small $\alpha$. Koskela $\&$ Martio \cite{koskelamartio1993removability} proved a weaker version using Minkowski content instead of Hausdorff measure. Buckley and Koskela \cite{buckley1999fusion} also established special cases.

In relation to Hölder continuity of solutions, Lyaghfouri in \cite{lyaghfouri2010holdersuphar} showed that a solution of the equation $-\Delta_{p(x)} u = \hat \mu$ is Hölder continuous with exponent $\alpha$ if and only if the nonnegative Radon measure $\hat \mu$ satisfies the growth condition
$$
\hat \mu (B_r(x)) \leq C r^{n - p(x) + \alpha(p(x) - 1)},
$$
for any ball $B_r(x) \subset \Omega^\prime \subset \mathbb{R}^n$, with $r$ small enough. This extends a result of Lewy and Stampacchia for the Laplace operator and a result of Kilpeläinen and Zhong for the $p$-Laplace operator with constant $p$.

In \cite{lewystampac1970smoothness}, Lewy and Stampacchia studied the equation
\begin{equation}\label{126}
-\Delta u = \hat \mu.
\end{equation}
They proved the equivalence between the growth condition
$$
\hat\mu(B_r(x)) \leq C r^{n-2+\alpha},
$$
on the measure $\hat \mu$, and the Hölder continuity of the solutions with exponent $\alpha$, when $n > 2$. For $n = 2$, they proved that each solution of \eqref{126} is Hölder continuous with exponent $\beta$ for all $\beta \in (0, \alpha)$.

In \cite{kilpelainen1994holderquas}, Kilpeläinen considered the $p$-Laplace operator, where $p$ is a constant. He proved that if a solution to the equation
\begin{equation}\label{124}
-\Delta_p u = \hat \mu
\end{equation}
is Hölder continuous with exponent $\alpha$, then the measure $\hat \mu$ satisfies the growth condition
\begin{equation}\label{125}
\hat \mu (B_r(x)) \leq C r^{n-p+\alpha(p-1)}.
\end{equation}
Conversely he proved that if $\hat \mu$ satisfies \eqref{125}, then each solution of \eqref{124} is Hölder continuous with exponent $\beta$ for all $\beta \in(0, \alpha)$.

In \cite{kilpelainenzhong2001removable}, Kilpeläinen and Zhong improved the result in \cite{kilpelainen1994holdermeasu} by showing that if $\hat \mu$ satisfies \eqref{125}, then each solution of \eqref{124} is Hölder continuous with exponent $\alpha$. Rakotoson \cite{rakotoson1990equivalence} and Rakotoson $\&$ Ziemer \cite{rakotoson1990localstru} also obtained results in this direction.


Throughout this paper, a Finsler manifold will be denoted by $(\Omega, \fnf, \wghtvi)$, where $\wghtvi$ is a measure and $\fnf : T\Omega \rightarrow [0, \infty)$ is a function (called a Finsler structure) with the following properties:
\begin{enumerate}[label=($F_{\arabic*}$)]
\item $\fnf$ is $C^{\infty}( T\Omega \backslash \{0\} )$.

\item $\fnf (\pvr , t \xi)=t \fnf ( \pvr , \xi)$ for all $(\pvr, \xi) \in T \Omega$ and  $t\geq 0$.

\item \label{115} For each chart $\psi : \tilde U \subset \mathbb{R} ^n \rightarrow \psi (\tilde{U}) \subset \Omega$, the matrix
$$
g_{i j}(\pvr , y):=\left. \frac{\partial^2}{\partial \xi _i \partial \xi _j}\left(\frac{1}{2} \fnf ^2 (\pvr , d\psi _{\pvr} (\xi)) \right)\right|_{\xi = y}
$$
is uniformly elliptic in the sense that there exist positive constants \(\kappa_1\) and \(\kappa_2\) such that
\begin{equation*}
\kappa _1 \sum _{k=1}^n (\eta ^k) ^2 \leq  g_{i j}(\pvr , \xi) \eta ^ i \eta ^j \leq \kappa _2 \sum _{k=1}^n (\eta ^k)^2
\end{equation*}
holds for all \(\pvr \in \psi ( \tilde{U} ) \), \(\xi \in T_{\pvr} \Omega \backslash \{0\}\), and \(\eta = \eta^i \partial_i \in T_{\pvr} \Omega\), where \(\{\partial_i\}_{i=1}^n\) is the basis associated with the chart \(\psi\).
\end{enumerate}

If $\fnf(\pvr , Y)= \fnf (\pvr , - Y) $ for all  $(\pvr , Y)\in T\Omega$, we say that the Finsler manifold $(\Omega , \fnf)$ is {\it reversible}.


Throughout this paper, we fix a measure \( \wghtvi \) on \( \Omega \). For each point \( \pvr \in \Omega \), we assume there exists a neighborhood \( U \) with a local coordinate system \( \psi : \tilde{U} \to U \), and a measurable function \( \tilde \wghtvi : \tilde{U} \to \mathbb{R}^+ \) such that for some constant \( C > 0 \), we have 
\begin{equation}\label{118}
C \leq \tilde \wghtvi  \leq C^{-1}  \quad \text { in } \quad \tilde U .
\end{equation}
Additionally, for every measurable set \( E \subset U \), the following holds:
\begin{equation} \label{119}
\wghtvi(E) = \int_{\psi^{-1}(E)} \tilde \wghtvi \,  \dx.
\end{equation}

We work with the following conditions on \( \wghtvi \) and \( p \).
\begin{enumerate}[label=(\roman*)]
\item There is constants $\conspoin , \conspoin _0, \conspoin _1 , \cpoin _1 >0$  such that
\begin{gather} 
\wghtvi  (B(\pvr , r)) \leq \conspoin r^{\cpoin _1} , \label{78}\\
  \left[ \wghtvi  (B(\pvr  , r))\right]^{p^{-} _r - p^+ _r} \leq \conspoin _0 ,\label{74}
\end{gather}
and
\begin{equation}\label{33}
\wghtvi ( B(\pvr , 2r)) \leq \conspoin _{1} \wghtvi (B(\pvr , r)),
\end{equation}
for all $r>0$ and $\pvr \in \Omega$, where $p^- _r = \inf _{B(\pvr  , r)} p$ and $p^+ _r = \sup _{B(\pvr  , r)} p$. 

\noindent {\it  Remark:} Under certain conditions, the doubling condition \eqref{33} follows from the Bishop-Gromov-type volume comparison theorem. See \cite[Section 2.8]{mester2022functionalinequa} or \cite[Lemma 2.2]{kristaly2022nonlinear}.

\item  There are constants $1_{\wghtz} >1$ and $\conspoin  _{\wghtz}, \conspoin  _{\wghtzi} >0$ such that for all ball $B=B(\pvr , r)$:
\begin{equation} \label{26}
\left(  \fint _{B} |v|^{1_{\wghtz}} \, \dwvi  \right)^{\frac{1}{1_{\wghtz}}} \leq \conspoin  _{\wghtz}  r  \fint _{B} |\grad  v| _{\fnf} \, \dwvi  ,
\end{equation}
whenever   $v \in C_0^{\infty}(B)$.

Furthermore:
\begin{equation}\label{64}
\int_B |v-v_B| \, \dwvi  \leq \conspoin  _{\wghtzi} r \int_B|\grad  v|_{\fnf }\, \dwvi ,
\end{equation}
whenever  $v \in C^{\infty}(B)$ is bounded. 

Here we are using the notation:
$$
\fint _B f \, \dwvi := \frac{1}{\wghtvi (B)} \int_B f\, \dwvi \quad \text { and } \quad f_B :=\frac{1}{\wghtvi (B)} \int_B f\, \dwvi  .
$$

\item   For every bounded open set \( U \), there is a constant such that
\begin{equation}\label{120}
\|u\|_{p,U,\wghtvi} \leq C\|\fnf( \cdot , \grad u) \|_{p ,  U, \wghtvi} \quad \forall  u\in C^{\infty} _{0} (U).
\end{equation}

\end{enumerate}

Our first main result is as follows:
\begin{theorem}  \label{104}
Let \( (\Omega, \fnf, \wghtvi) \) be a reversible Finsler manifold, where \( \wghtvi \) and \( p \) satisfy conditions \eqref{118}–\eqref{120}. Let $\singset \subset \Omega$ be a closed set. Assume that $u \in C(\Omega)$ is a solution of \eqref{13} in $\Omega \backslash \singset$, and there exists $\alpha \in (0,1]$ such that for any $\pvr \in \singset$ and $\pvry \in \Omega$:
$$
|u(\pvr)-u(\pvry)| \leq Cd^\alpha (\pvr , \pvry).
$$

If for each compact subset $K$ of $\singset$, the $\cpoin _1 -p^+_{K} + \alpha (p^+_{K} -1 )$-Hausdorff measure of $K$ is zero, then $u$ is a solution of \eqref{13} in $\Omega$.
\end{theorem}

Next, we consider a family \(\{ \vue \}_{\ell \in \mathscr{J}} \subset \lwo \) and an element \(\vuu \in \lwo \). For the following result, we assume the properties:
\begin{enumerate}[label=($P_{\arabic*}$)]
\item \label{94} For all $\ell \in \mathscr{J}$ we have
\begin{equation}\label{10}
\int_{\Omega}  \mathscr{A}(\cdot, \grad  \vuu) \multg   \deriv (\vue - \vuu )  \, \dwvi   =0 
\end{equation}
and, for some positive constants $\vinf _0  $ and $\vinf$,
\begin{equation} \label{28}
 |\vue - \vuu|\leq \vinf _0  \quad \text {  and  } \quad |\mathscr{A}(\cdot, \grad \vue) -  \mathscr{A}(\cdot, \grad \vuu) |_{\fnf}\leq \vinf \leq \frac{\consfi }{4}   \quad \text { in }  \Omega.
\end{equation}

\item \label{93} There exist  functions \( \Upsilon _i: \mathbb{R}^+ \rightarrow [0,\infty)\), $i=1,2$, such that, for all \( \ell \in \mathscr{J} \) and for all \( \zeta \in C_0^\infty(B(\pvr, r)) \) with \( B(\pvr, r) \subset \subset \Omega \) and \( 0 \leq \zeta \leq 1 \), the following holds:
\begin{gather*}
\int _{B(\pvr , r)} \left[ \mathscr{A}( \cdot , \grad  \vuu ) - \mathscr{A}( \cdot , \grad  \vue ) \right] \multg  \deriv \zeta  \, \dwvi  \leq \Upsilon _1 ( \pvr , r)  \quad  \text { and }\\
-\int _{B(\pvr , r)} \left[ \mathscr{A}( \cdot , \grad  \vuu ) - \mathscr{A}( \cdot , \grad  \vue )  \right] \multg  \deriv \zeta  \, \dwvi  \leq \Upsilon _2 ( \pvr , r) \quad  \text { if } \quad  B(\pvr , r)\subset \{ \vuu > \obsi \}
\end{gather*}
\end{enumerate}

\begin{enumerate}[label=($P_{\arabic*}$)]
\setcounter{enumi}{2}
\item \label{103} Let \( 0_{\Omega} \in \Omega \) be fixed. Assume that \( p \in C(\Omega) \) satisfies \( 1 < p < p^+ := \sup_{\Omega} p < \infty \) in \( \Omega \), and that $\lim_{d(\pvr, 0_{\Omega}) \to \infty} p(\pvr) = 1$, and, for some constant $\tilde{\conspoin} _1>0$,
\begin{equation} \label{70}
\frac{\sup _{B(\pvr _0 , r)} p -1}{\inf _{B(\pvr _0 , r)} p  -1} \leq \tilde{\conspoin} _1 \quad \forall B(\pvr _0 , r ) \subset \Omega.
\end{equation}

\end{enumerate}

For every open set $V\subset \subset \Omega$, we define
$$
\mu _{\ell}(V):=\sup \left\{\int_V  \mathscr{A}(\cdot , \grad  \vue)  \multg  \deriv \zeta  \, \dwvi  \:|\: 0 \leq \zeta \leq 1 \text { and } \zeta \in C_0^{\infty}(V)\right\} .
$$

Our second main result is the following
\begin{theorem} \label{81} 
Let \( (\Omega, \fnf, \wghtvi) \) be a reversible Finsler manifold, where \( \wghtvi \) and \( p \) satisfy conditions \eqref{118}–\eqref{64}. Let \( K \) be a compact subset of \( \Omega \), and define \( K_\delta := \{\pvr \in \Omega \:|\: d(\pvr, K) < \delta\} \), where \( \overline{K_\delta} \subset \Omega \). Suppose that \( \vartheta \in \mathscr{K}_{\obsi}(\Omega) \) is a solution to the obstacle problem \eqref{4}, with the obstacle \( \obsi \in C(\Omega) \) satisfying
\[
|\obsi(\pvr) - \obsi(\pvry)| \leq \conrem \, d^\alpha(\pvr, \pvry), \quad \text { for all } \pvr \in K, \pvry \in \Omega,
\]
where \( 0 < \alpha \leq 1 \) and \( \conrem > 0 \).

Additionally, let \( \{ \vartheta_\ell \} \subset \mathscr{K}_{\obsi}(\Omega) \) be a family such that \( \{\vuu = \obsi\} \subset \{\vue = \obsi\} \) for all \( \ell \), and suppose that this family satisfies properties \ref{94} and \ref{93}, with \( p \in C(\Omega) \) satisfying property \ref{103}.

There exists $ R_1= R_1 ( \consfi  , \consfii  , p^+ , \tilde \conspoin _1 , \conspoin _1 , \conspoin _2  ,  \conspoin _3 , \conspoin _4 , \conspoin _{\wghtzi} , \conspoin _0 , \vinf _0 ,  n , s , \cpoin _1 , \inf _{B(\pvr , 5r)} \obsi , \sup _{B(\pvr , 5r)} \obsi)>0$, where $s>\sup _{K_\delta} p - \inf _{K_\delta} p$, such that $p-1\leq 1$ in $\Omega \backslash B(0_{\Omega} , R_1)$.  Furthermore, for every $\pvr \in K \cap( \Omega \backslash B(0_{\Omega} , R_1) )$ and $0<r< \frac{1}{41}\min \{\rast , \delta\}$, the following  holds:
\begin{align*}
&\mu_\ell(B(\pvr, r)) \\
& \quad \ \leq C _1 C_2^{\frac{\tilde{\conspoin} _1}{p^-_{ 41r} -1}}   (p^-_{41r}-1)^{-\tilde{\conspoin} _1} r^{\cpoin _1 -p(\pvr) + \alpha (p(\pvr) -1 )  
} + \Upsilon _1 ( \pvr , 5r) + \Upsilon _2 (\pvr , r), 
\end{align*}
if $B(\pvr , 41 r )\subset \subset  \psi (\tilde U)$, where $\psi : \tilde U \rightarrow \psi (\tilde U )\subset \Omega$ is a chart,  {\footnotesize $p^-_{41r} := \inf _{B(\pvr , 41r)}p$,    $C_i:= C_i( \consfi  , \consfii  , p^+ , \tilde \conspoin _1 , \conspoin _1 , \conspoin _2 , \conspoin _3 ,$ $ \conspoin _4 , \conspoin _{\wghtz} , \conspoin _{\wghtzi} , $ $\conspoin _0 , \conspoin , n , s , \gamma , 1_{\wghtz}, \cpoin _1 , \cpoin _2 , \alpha , \conrem , $ $\inf _{B(\pvr , 5r)} \obsi , \sup _{B(\pvr , 5r)} \obsi , \consman _1 ,  \vinf _0 , \| \vuu  \|_{s , B(\pvr  , 41r) , \wghtvi } , $ $ \| \vuu  \|_{s \gamma ^\prime , B(\pvr  , 41r) , \wghtvi } )>0$,  for $i=1,2$, and $1<\gamma<1_{\wghtz}$.  Also, \(\consman_1\) is a constant (depending on \(\psi\) and \(\fnf\)) given in \eqref{114}, and $\rast$ is defined in Lemma \ref{128} \ref{76}.}
\end{theorem}


In this work, we employ the techniques from \cite{fu2015removable, lyaghfouri2012remharmonic, ono2013removable}, which address the case where the norm and the domain are in $\mathbb{R}^n$. This paper is organized as follows. In Section \ref{121}, we provide the necessary definitions of variable exponent Lebesgue and Sobolev spaces on Finsler manifolds. In this section, we also define $\mathscr{A}$-harmonic functions, supersolutions, and solutions to the obstacle problem. Additionally, we present some basic results that will be used throughout the paper. In Section \ref{122}, we provide estimates for solutions of the obstacle problem and for supersolutions. Finally, in Section \ref{123}, we study removable sets for $\mathscr{A}$-harmonic functions and prove Theorems \ref{104} and \ref{81}.

\section{Preliminary Results and Definitions}\label{121}

\subsection{Finsler Manifolds}

We recall some facts and notation about Finsler manifolds. Most properties can be found in \cite{farkas2015singular, mester2022functionalinequa,  ohta2009heatfinsler, xiong2020uniquenessnonnegative}.

\subsubsection{The Legendre Transform}
The polar transform (or co-metric) \( \fnfs : T^\ast \Omega \rightarrow [0, \infty) \) is defined as the dual metric of \( \fnf \), given by the expression:
\[
\fnfs (\pvr , \omega ) = \sup_{v \in T_{\pvr} \Omega \backslash  \{0\}} \frac{\omega (v)}{\fnf (x, v)},
\]
where \( T^\ast \Omega = \cup_{\pvr \in \Omega} T_{\pvr} ^\ast \Omega \) is the cotangent bundle of \( \Omega \), and \( T_{\pvr}^\ast \Omega \) is the dual space of \( T_\pvr \Omega \).

If $\pvr \in \psi (\tilde U)\subset \Omega$, where $\psi$ is a chart. We have that for all $\omega \in T^\ast _{\pvr} \Omega$ there exist a unique $(\omega _1 , \ldots , \omega _n)\in \mathbb{R}^n$ such that 
$$
\omega (d \psi _{\pvr} ( v )) = \langle (\omega _1 , \ldots , \omega _n) , v \rangle   \quad \forall v\in \mathbb{R}^n,
$$
where $\langle \cdot , \cdot \rangle$ is the Euclidean metric.  In these local coordinates, for simplicity, we write \( \fnfs (\pvr , \omega) = \fnfs (\pvr , (\omega_1 , $ $\ldots , \omega_n)) \).

Since \( \fnf ^{\ast 2}(\pvr , \cdot) \) is twice differentiable on \( T_{\pvr}^\ast \Omega\backslash \{0\} \), one can define the Hessian (dual) matrix $[ g^\ast _{ij}(\pvr , \omega )]$, where
\[
 g^\ast _{ij}(\pvr , \omega ):=  \frac{1}{2} \frac{\partial^2}{\partial \omega _i \partial \omega _j } \fnf ^{\ast 2}\big(\pvr , (\omega_1 , \ldots , \omega_n)    \big)   .
\]

Using the strong convexity assumption on the Finsler structure \(\fnf\), the Legendre transform \( J^\ast : T^*\Omega  \rightarrow T\Omega \) is defined as follows: for each fixed \( \pvr \in \Omega  \), \( J^\ast \) associates to each \( \omega  \in T_{\pvr }^\ast \Omega  \) the unique maximizer \( v \in T_{\pvr} \Omega \) of the mapping:
\[
v \mapsto \omega (v) - \frac{1}{2} \fnf ^2(\pvr, v).
\]

Note that if \( J^\ast ( \pvr , \omega ) = (\pvr , v) \), then the following relations hold:
\[
\fnf (\pvr , v) = \fnfs (\pvr  , \omega ) \quad \text { and } \quad \omega (v) = \fnfs (\pvr  , \omega ) \fnf (\pvr , v).
\]

In local coordinates,  we have:
\[
J^\ast (\pvr , \omega ) = \sum_{i=1}^n \frac{\partial}{\partial \omega _i} \left( \frac{1}{2} \fnf ^{\ast 2}(\pvr , \omega ) \right) \partial _i,
\]
where \(\{\partial_i\}_{i=1}^n\) is the basis associated with the chart \(\psi\).

\subsubsection{Gradient and Distance Functions}

For a weakly differentiable function \( u: \Omega \rightarrow \mathbb{R} \), the gradient vector at \( \pvr \) is defined by
$$
\grad u(\pvr):=J^*(\pvr, \deriv u(\pvr)),
$$
for every regular point \( \pvr \in \Omega \), where the derivative \( \deriv u(\pvr) \in T_{\pvr}^\ast \Omega \) is well-defined. By applying the properties of the Legendre transform, it follows that
\begin{equation} \label{113}
 \fnfs  (\pvr , \deriv v (\pvr))=\fnf  (\pvr , \grad v (\pvr))
\end{equation}

In a local coordinate system, if \( \pvr \in \psi(\tilde{U}) \), we have 
$$
\deriv u(\pvr)=\sum_{i=1}^n \left( \partial _i (u)  \partial _i ^\ast \right)(\pvr) \quad \text { and } \quad \grad u(\pvr)=\sum_{i, j=1}^n \left(  g_{i j}^\ast (\cdot  , \deriv u) \partial _i(u) \partial _j \right)(\pvr) ,
$$
where $\partial ^\ast _i (\pvr) \in T^\ast _{\pvr} \Omega$ is defined by $\partial ^\ast _i (\pvr)(\partial _j (\pvr)):= \delta _{ij}$.

For a differentiable vector field $Y : \Omega \rightarrow T\Omega$ on $\Omega$, we define its divergence $\operatorname{div} Y: \Omega \to \mathbb{R}$ through the identity
$$
\int_{\Omega} u \operatorname{div}  Y \, \dwvi =-\int_{\Omega} \deriv u (Y) \, \dwvi.
$$

We observe that the nonlinearity of the Legendre transform extends to the gradient vector, namely \(\grad (u+v) \neq \grad u + \grad v\) in general. For the same reason, at points \(\pvr\) where \(\grad u(\pvr) = 0\), the gradient vector field \(\grad u\) is, in general, not differentiable (even if \(u\) is smooth), but only continuous.




We define the distance function $d : \Omega \times \Omega \rightarrow [0,\infty)$ by
$$
d(\pvr , \pvry)=\inf _\gamma \int_a^b \fnf( \gamma (t) , \gamma^\prime (t)) \, \dt,
$$
where the infimum is taken over all piecewise differentiable curves \(\gamma : [a,b] \to \Omega\) with \(\gamma(a) = \pvr\) and \(\gamma(b) = \pvry\). This definition is equivalent to
$$
d(\pvr, \pvry):=\sup \left\{u(\pvry)-u(\pvr) \:|\: u \in C^1(\Omega), \fnf ( \pvrz, \grad u(\pvrz)) \leq 1 \text { for all } \pvrz \in \Omega\right\} .
$$

 For a fixed \(\pvry \in \Omega\), the distance function \(\pvr \mapsto d(\pvry, \pvr)\) satisfies \(\fnf(\pvr, \grad d(\pvry, \pvr)) = 1\) for almost every \(\pvr \in \Omega\). Moreover, the distance function \(d\) possesses the following properties of a metric:
\begin{enumerate}
\item[$(a)$] \(d(\pvr, \pvry) \geq 0\) for all \(\pvr, \pvry \in \Omega\), and \(d(\pvr, \pvry) = 0\) if and only if \(\pvr = \pvry\).

\item[$(b)$] \(d(\pvr, \pvrz) \leq d(\pvr, \pvry) + d(\pvry, \pvrz)\) for all \(\pvr, \pvry, \pvrz \in \Omega\).
\end{enumerate}

However, in general, the distance function is not symmetric. In fact, \(d(\pvr, \pvry) = d(\pvry , \pvr)\) for all \(\pvr, \pvry \in \Omega\) if and only if \((\Omega , \fnf)\) is a reversible Finsler manifold.


\subsection{Variable exponent Sobolev spaces}

Next, we recall some key facts and notation regarding variable exponent Lebesgue and Sobolev spaces. A detailed discussion of the properties of these spaces can be found in \cite{gaczkowski2013sobolev, gaczkowski2016sobolev} for Riemannian manifolds, \cite{futamura2006sobolev, harjulehto2006sobolev, harjulehto2006variablemetric} for metric spaces, and \cite{diening2011lebesgue, kovuavcik1991spaces} for the Euclidean setting.

Let $p : \Omega \rightarrow (1,\infty)$ be a continuous function, and let $\hat \Omega  \subset \Omega$ be a open subset. The variable exponent, or generalized Lebesgue space $L^{p(\pvr)}(\hat \Omega ; \wghtvi)$, is the space of all measurable functions \( u : \hat \Omega \rightarrow \mathbb{R} \) for which the functional
$$
\rho_{p}(u):=\int_{\hat \Omega}|u|^{p} \, \dwvi
$$
is finite. This is a special case of an Orlicz-Musielak space, see \cite{musielak2006orlicz}.

 The functional \( \rho_{p} \) is convex and is sometimes called a convex modular \cite{kovuavcik1991spaces}. The space \( L^{p(\pvr)}(\hat \Omega; \wghtvi) \) is a Banach space with respect to the Luxemburg-Minkowski-type norm
$$
\|u\|_{p , \hat \Omega , \wghtvi}:=\inf  \left\{t>0  \:|\: \rho_{p} \left(\frac{u}{t} \right) \leq 1  \right\} .
$$
The space $L^{p(\pvr)}(\Omega ; \wghtvi)$ is reflexive if $1<p^-_{\hat \Omega} \leq p^+_{\hat \Omega} <\infty$, see \cite[Proposition 2.3]{harjulehto2006variablemetric}, where $p^-_{\hat \Omega} := \inf _{\hat \Omega} p$ and $p^+_{\hat \Omega} := \sup _{\hat \Omega} p$.

In what follows, we will use the generalized Hölder inequality:
\begin{equation}\label{116}
\int_{\hat \Omega }  |uv| \, \dwvi \leq 2 \|u\|_{p,\hat \Omega , \wghtvi} \|v\|_{p^{\prime} ,\hat \Omega , \wghtvi},
\end{equation}
as stated in \cite[Theorem 2.1]{kovuavcik1991spaces}.

To compare the functionals $\|\cdot \|_{p , \hat \Omega ,\wghtvi}$ and $\rho_{p}$, we have the following relations:
\begin{equation}\label{117}
\min  \{\rho_{p}(u)^{1 / p^{-}}, \rho_{p}(u)^{1 / p^{+}} \} \leq\|u\|_{p,\hat \Omega , \wghtvi} \leq \max  \{\rho_{p} (u) ^{1 / p^{-}}, \rho_{p}(u)^{1 / p^{+}} \}.
\end{equation}


The variable exponent Sobolev space $W^{1, p(\pvr)}(\hat \Omega ; \wghtvi)$ is defined by
$$
W ^ {1, p(\pvr)}(\hat \Omega ; \wghtvi)  :=\left\{u \in L^{p(\pvr)} (\hat \Omega ; \wghtvi) \:|\: \int_{\hat \Omega} \fnf (\cdot , \grad u)^p \, \dwvi<\infty \right\} .
$$
This space is a Banach space with the norm
$$
\|u\|_{1, p , \hat \Omega , \wghtvi}:= \|u\|_{p , \hat \Omega , \wghtvi} + \|\fnf ( \cdot ,  \grad u) \|_{ p , \hat \Omega , \wghtvi}.
$$
We also define  $W^{1 , p(\pvr)} _0(\hat \Omega ; \wghtvi)$ as the closure of $C^{\infty} _{0} (\hat \Omega)$ in  $W^{1 , p(\pvr)} (\hat \Omega ; \wghtvi)$ with respect to the norm  $\|\cdot\|_{1, p , \hat \Omega , \wghtvi}$.



Note that the classes \( L_{\loc}^{p(\pvr)}(\Omega ; \wghtvi) \) and \( W_{\loc}^{1, p(\pvr)}(\Omega; \wghtvi) \) depend only on the manifold structure of \( \Omega \) (not on the Finsler structure \( \fnf \) and the measure \( \wghtvi \)); see the last paragraph on page 1394 of \cite{ono1997solutionsquasi}. More precisely, let \( \hat \Omega \subset \subset \Omega \) be an open set, and consider two measures \( \wghtv_1 \) and \( \wghtv_2 \) satisfying \eqref{118} and \eqref{119}. These measures are comparable in \( \hat \Omega \). In other words, there exists a constant \( C > 0 \) such that
$$
C \wghtv _1 (E) \leq \wghtv _2 (E) \leq C^{-1} \wghtv _1 (E),
$$
for every measurable subset \( E \subset \hat \Omega \). This implies that \( \int_{\hat \Omega} |u|^p  d\wghtv_1 < \infty \) if and only if \( \int_{\hat \Omega} |u|^p  d\wghtv_2 < \infty \). Thus, \( L_{\loc}^p(\Omega; \wghtvi) \) does not depend on the measure \( \wghtvi \).

For \( W_{\loc}^{1, p(\pvr)}(\Omega, \wghtvi) \), we further observe that for any two Finsler structures \( \fnf_1 \) and \( \fnf_2 \), by 1-homogeneity and property \ref{115}, we have
$$
c \fnf_1(\cdot, \grad u) \leq \fnf_2(\cdot, \grad u) \leq c^{-1} \fnf_1(\cdot, \grad u) \quad \text { in }  \hat \Omega,
$$
for some constant \( c > 0 \). Therefore, \( W_{\loc}^{1, p(\pvr)}(\Omega; \wghtvi) \) does not depend on either the measure \( \wghtvi \) or the Finsler structure \( \fnf \).

We will finish this subsection with the following definitions:
  \begin{definition} We say that a continuous function $u\in W_{\loc}^{1, p(\pvr)}( \Omega \backslash \singset ; \wghtvi )$ is an \( \mathscr{A} \)-harmonic function in $ \Omega \backslash \singset$ if,  for every test function $\varphi \in W^{1, p(\pvr)}(\Omega \backslash \singset ;   \wghtvi  )$ with compact support, the following holds:
$$
\int_{\Omega}   \mathscr{A}(\cdot , \grad u)  \multg  \deriv \varphi \, \dwvi   =0 .
$$
\end{definition}

\begin{definition} 
We say that $u \in \lwo$ is a \textnormal{supersolution} of \eqref{13} in $\Omega$ if, for any nonnegative test function $\varphi \in W^{1, p(\pvr)}(\Omega ; \wghtvi )$ with compact support, the following holds:
\begin{equation} \label{14}
\int_{\Omega} \mathscr{A}( \cdot , \grad u) \multg  \deriv \varphi  \, \dwvi   \geq 0 .
\end{equation}

Additionally, \( u \) is a \textnormal{subsolution} in \( \Omega \) if \( -u \) is a supersolution in \( \Omega \), and \( u \) is a \textnormal{solution} in \( \Omega \) if it is both a supersolution and a subsolution in \( \Omega \).
\end{definition}

\begin{definition}
Let \( \singset \) be a closed subset of \( \Omega \). We say that \( \singset \) is a removable set for Hölder continuous \( \mathscr{A} \)-harmonic functions if, for any Hölder continuous function \( u : \Omega \rightarrow \mathbb{R} \), the following holds: If \( u \) is \( \mathscr{A} \)-harmonic in \( \Omega \backslash \singset \), then \( u \) is \( \mathscr{A} \)-harmonic in \( \Omega \).
 \end{definition}


\subsection{Metric Spaces}


We assume  that $X=(X, d, \mu)$ is a metric space endowed with a metric $d$ and a positive complete Borel measure $\mu$ such that
$$
\mu(B)<\infty \quad \text { for all balls } B \subset X.
$$

The measure $\mu$ is said to be {\it doubling} if there exists a constant $C_\mu \geq 1$, called the doubling constant of $\mu$, such that for all balls $B$,
$$
0<\mu(2 B) \leq C_\mu \mu(B)<\infty .
$$

\begin{definition} Let
$$
\|f\|_{\mathrm{BMO}(W ; X)}:=\sup _{B \subset W} \fint _B\left|f-f_B\right| \, \textnormal{d}\mu,
$$
where the supremum is taken over all balls \( B \subset \Omega \) (and we implicitly require that \( f _B \) is finite for all balls \( B \subset \Omega \)). We define the class of functions of \textnormal{bounded mean oscillation} as
$$
\operatorname{BMO}(W ; X):=\left\{f \:|\:\|f\|_{\mathrm{BMO}(W ; X)}<\infty\right\} .
$$
\end{definition}

\begin{proposition} \label{73}(see \cite[Theorem 3.20]{bjorn2011nonlinear}.) Suppose $\mu$ is doubling. Let $B \subset X$ be a ball and $f \in \operatorname{BMO}(5 B ; X)$. Let further $A:=\log (2) / 4 C_\mu ^{15}$. Then for every $0<\epsilon<A$ there is a constant $c$, only depending on $\epsilon$ and $C_\mu$, such that
$$
\fint _B e^{\epsilon |f-f_B| /\|f\|_{\mathrm{BMO}(5 B ; X)}} \, \textnormal{d} \mu \leq c .
$$

\end{proposition}

\subsection{Integrals in Terms of Distribution Functions}

Let $\{X, \mathcal{A}, \mu\}$ be a measure space and $E \in \mathcal{A}$. 

Let $f: E \rightarrow [-\infty , \infty] $ be measurable and nonnegative. The distribution function of $f$ relative to $E$ is defined as
$$
t \in \mathbb{R}^{+} \mapsto \mu(\{f>t\}) .
$$

This is a nonincreasing function of $t$ and if $f$ is finite a.e. in $E$, then
$$
\lim _{t \rightarrow \infty} \mu([f>t])=0 \quad \text { unless } \quad \mu([f>t]) \equiv \infty .
$$

If $f$ is integrable, such a limit can be given a quantitative form. Indeed
$$
t \mu([f>t])=\int_E t \chi_{[f>t]} \, \textnormal d \mu \leq \int_E f \, \textnormal d \mu<\infty .
$$
\begin{proposition} \label{7} (see \cite[Proposition 15.1]{dibenedetto2016real}.) Let $\{X, \mathcal{A}, \mu\}$ be complete and $\sigma$-finite and let $f: E \rightarrow [-\infty , \infty] $ be measurable and nonnegative. Let also $\mathtt{v}$ be a complete and $\sigma$-finite measure on $\mathbb{R}^{+}$ such that $\mathtt{v}([0, t))=\mathtt{v} ([0, t])$ for all $t>0$. Then
$$
\int_E \mathtt{v} ( \{ 0, f\} ) \,  \textnormal d \mu=\int_0^{\infty} \mu(\{ f>t \}) \, \textnormal d \mathtt{v} .
$$

In particular if $\mathtt{v} ([0, t])=t^q$ for some $q>0$, then
$$
\int_E f^q \, \textnormal d \mu=q \int_0^{\infty} t^{q-1} \mu(\{ f>t \} ) \, \textnormal d t , 
$$
where $\textnormal d t$ is the Lebesgue measure on $\mathbb{R}^{+}$.
\end{proposition}


\subsection{Auxiliary Inequalities and Estimates}

From \eqref{33}, there exist constants $\conspoin _2 >0$  and $\cpoin _2>1$ such that we have
\begin{equation}\label{34}
\frac{\wghtvi  (B(\pvr _1 , r))}{\wghtvi  (B(\pvr _0 , R))} \geq \conspoin _2\left( \frac{r}{R}\right)^{\cpoin _2},
\end{equation}
whenever $0<r<R$ and $\pvr _1 \in B(\pvr _0 , R)$, see \cite[pp. 103-104]{heinonen2001lecturesmetric}

\begin{lemma}
Assume that \eqref{26} is satisfied. Let $v\in W^{1,q} _0 (B(\pvr , r) ; \wghtvi)\cap L^{\infty}(B(\pvr , r)) $ with $q\geq 1$, then
\begin{equation}  \label{27}
\left(  \fint _{B(\pvr , r)} |v|^{q 1_{\wghtz}} \, \dwvi  \right)^{\frac{1}{q1_{\wghtz}}} \leq q \conspoin  _{\wghtz}  r\left(  \fint _{B(\pvr , r)} |\grad  v|_{\fnf}^{q} \, \dwvi  \right)^{\frac{1}{q}}.
\end{equation}
\end{lemma}
\begin{proof}
From \eqref{26} and  Hölder's inequality, 
\begin{align*}
&\left(  \fint _{B(\pvr , r)} |v|^{q1_{\wghtz}} \, \dwvi  \right)^{\frac{1}{1_{\wghtz }}} \leq \conspoin  _{\wghtz}  r\fint _{B(\pvr , r)} |\grad  (|v|^q)| _{\fnf} \, \dwvi \\ 
& \quad  \ \leq q \conspoin  _{\wghtz}  r \left(\fint _{B(\pvr , r)} |\grad  v|_{\fnf}^q \, \dwvi \right)^{\frac{1}{q}} \left(\fint _{B(\pvr , r)} |v|^q \, \dwvi  \right)^{\frac{q-1}{q}}\\
& \quad  \ \leq q \conspoin  _{\wghtz}  r \left(\fint _{B(\pvr , r)} |\grad  v|_{\fnf} ^q \, \dwvi \right)^{\frac{1}{q}} \left(\fint _{B(\pvr , r)} |v|^{q 1_{\wghtz}}  \, \dwvi  \right)^{\frac{q-1}{q1_{\wghtz}}} .
\end{align*}
Using this inequality we obtain \eqref{27}.   \end{proof}

The next two lemmas are based on \cite{harjulehto2006unbounded}.
\begin{lemma} \label{128} Assume that \eqref{78} and \eqref{74} holds, we have
\begin{enumerate}[label=(\roman*)]
\item \label{76}   If $0<r\leq \rast := [1/(\conspoin + 1)]^{1/\cpoin _1} <1$, then
\begin{equation} \label{66}
p^+ _{r} - p^- _{r} \leq \conspoin _3:= \frac {\log \conspoin _0 }{ \log \frac{\conspoin +1}{\conspoin}},
\end{equation}
where $p^- _{r} = \inf _{B(\pvr ,r) } p$ and $p^+ _{r} = \sup _{B(\pvr ,r)} p$.

\item \label{77} If $0<r\leq \rast$, then
\begin{equation} \label{38}
r^{-p(\pvr )} \leq \conspoin _4 ( \conspoin , \conspoin _0 , \cpoin _1)r^{-p_{W  }^{-}} ,
\end{equation}
where $\pvr \in W \subset B(\pvr ,r)$ and $p^- _{W} = \inf _{W } p$.

\end{enumerate}
\end{lemma}
\begin{proof}
\ref{76}  This is a consequence of \eqref{78} and \eqref{74}. \\

\ref{77}  We have $0\leq r \leq \rast < 1$ and $\pvr\in W\subset B(\pvr , r)$. From \eqref{74},
$$
r^{-p(\pvr)} \leq r^{-p^+_W} \leq r^{-p^+_r} \leq  r^{-p^+_r + p^-_r - p^-_W}\leq  \left(\conspoin ^{p^+_r - p^-_r} \conspoin _0 \right)^{\frac{1}{\cpoin _1}} r^{- p^-_W}.
$$
By employing \eqref{66}, we obtain \eqref{38}.
\end{proof}

\begin{lemma} Assume that  \eqref{74} holds. Let $f\in L^s (B(\pvr , r) ; \wghtvi )$, with $0<r\leq \rast$ and $s>p_{r  }^{+}-p_{r  }^{-}$. Then,
\begin{equation} \label{75}
\fint _{B(\pvr , r   )} |f|^{p_{r  }^{+}-p_{r  }^{-}} \, \dwvi  \leq \conspoin _0^{\frac{1}{s}}\|f\|_{s,B(\pvr , r   ) , \wghtvi }^{p_{r  }^{+}-p_{r  }^{-}}
\end{equation}
and
\begin{equation} \label{80}
\fint _{B(\pvr , r   )} |f|^{p-p_{r  }^{-}} \, \dwvi  \leq \conspoin _0^{\frac{1}{s}}\|f\|_{s,B(\pvr , r   ) , \wghtvi }^{p_{r  }^{+}-p_{r  }^{-}} +1.
\end{equation}
\end{lemma}
\begin{proof}
Let $q:= p^+ _{r} - p^- _{r}$. H\"{o}lder's inequality gives
\begin{align*}
&\fint _{B(\pvr  , r)} |f|^{p^+ _{r} - p^- _{r}} \, \dwvi  \leq  \left[\wghtvi  (B(\pvr  , r)) \right]^{-\frac{q}{s}} \left( \int _{B(\pvr  , r)} |f|^s \, \dwvi \right)^{\frac{q}{s}} .
\end{align*}
Using \eqref{74}, we get \eqref{75}.

Inequality \eqref{80} follows from \eqref{75} by applying  \( |a|^{p(\pvry) - p^-_r} \leq |a|^{p^+_r - p^-_r} + 1 \) for all \( a \in \mathbb{R} \) and \( \pvry \in B(\pvr, r) \).\end{proof}

Based on \cite[Lemma 3.38]{heinonen1993nonlinear}, we obtain the following result.
\begin{lemma} \label{47}
Assume that \eqref{33} is verified. Suppose  $0<q<\gamma <s $ and $ \beta \geq 0$. If a nonnegative function $v \in L^{s} (B(\pvr _0 , r  ) ; \wghtvi )$ satisfies 
\begin{equation} \label{44}
 \left(\fint _{B(\pvr _0, \lambda r^\prime)} v^s \, \dwvi \right)^{\frac{1}{s}} \leq \bar{c}  (1-\lambda)^{-\beta} \left( \fint _{B(\pvr _0 , r^{\prime} )} v^{\gamma} \, \dwvi  \right)^{\frac{1}{\gamma}}
\end{equation}
for each  $0<r^{\prime} \leq r$ and $0<\lambda<1$, then
\begin{equation}\label{46}
 \left(\fint _{B( \pvr _0 , \lambda r) } v^s \, \dwvi  \right)^{\frac{1}{s}} \leq  C (1-\lambda)^{-\frac{\beta}{\theta}} \left(\fint _{B(\pvr _0 , r)} v^q \, \dwvi  \right)^{\frac{1}{q}}.
\end{equation}
for all $0<\lambda<1$. Here, {\footnotesize $C:=2 (\conspoin _2 2^{-\cpoin _2})^{-\frac{1}{q}} \theta \left[(1-\theta)\bar c 2^{\frac{\beta}{\theta} +1}  \right] ^{\frac{1-\theta}{\theta}}$, and $\theta \in(0,1)$ is such that
$$
\frac{1}{\gamma}=\frac{\theta}{q}+\frac{1-\theta}{s}.
$$}
\end{lemma}
\begin{proof} Let
$$
N:=\sup _{\frac{1}{2}<\lambda<1}(1-\lambda)^{\hat{\beta}} \left(\fint _{B(\pvr _0, \lambda r)} v^{\gamma} \, \dwvi  \right)^{\frac{1}{\gamma}},
$$
where $\hat{\beta}=\beta(1-\theta) / \theta$.

 Writing $\lambda^{\prime} := 1/2(1+\lambda)$ for each $\lambda \in(0,1)$, we have, by \eqref{44},
\begin{equation} \label{45}
(1-\lambda)^{\frac{\hat{\beta}}{1-\theta}} \left(\fint _{ B(\pvr , \lambda r )} v^s \, \dwvi  \right)^{\frac{1}{s}} \leq \bar{c} 2^{\frac{\beta}{\theta}}(1-\lambda^{\prime} )^{\hat{\beta}} \left( \fint _{ B (\pvr _0 , \lambda^{\prime} r )} v^{\gamma} \, \dwvi  \right)^{\frac{1}{\gamma}} \leq \bar{c} 2^{\frac{\beta}{\theta}} N.
\end{equation}

Fix $\delta>0$ and choose $\lambda _{\delta} \in [1/2 , 1 )$ such that
$$
N \leq (1-\lambda _{\delta})^{\hat \beta} \left(\fint _{ B(\pvr _0 , \lambda _{\delta} r )} v^\gamma \, \dwvi  \right)^{\frac{1}{\gamma}}+\delta .
$$

Next, we apply Young's inequality:
$$
|a b| \leq  \epsilon^{-\frac{1-\theta}{\theta}} \theta |a|^{\frac{1}{\theta}}+\epsilon (1-\theta) |b|^{\frac{1}{1-\theta}},
$$
where $\epsilon>0$.

Since
$$
1=\theta \frac{\gamma}{q}+(1-\theta) \frac{\gamma}{s},
$$
we have
\begin{equation*}
\begin{aligned}
& N \leq (1-\lambda_0 )^{\hat{\beta}} \left(\fint _{B(\pvr _0 , \lambda _{\delta} r)} v^\gamma \, \dwvi  \right)^{\frac{1}{\gamma}}+\delta \\
& \quad \ = (1-\lambda _{\delta} )^{\hat{\beta}} \left(\fint _{ B( \pvr _0 , \lambda _{\delta} r)} v^{\gamma \theta} v^{\gamma (1-\theta)} \, \dwvi  \right)^{\frac{1}{\gamma}}+\delta \\
& \quad \ \leq (1-\lambda _{\delta} )^{\hat{\beta}} \left(\fint _{B(\pvr _0 , \lambda _{\delta} r)} v^q \, \dwvi  \right)^{\frac{\theta}{q}} \left(\fint _{B( \pvr _0 , \lambda _{\delta} r)} v^s \, \dwvi  \right)^{\frac{1-\theta}{s}}+\delta \\
& \quad \ \leq \theta \epsilon ^{-\frac{1-\theta}{\theta}} \left( \fint _{B(\pvr _0 , \lambda _{\delta} r)} v^q \, \dwvi   \right)^{\frac{1}{q}}+  \epsilon (1-\theta) (1-\lambda _{\delta})^{\frac{\hat{\beta} }{1-\theta}} \left(\fint _{ B(\pvr _0 , \lambda _{\delta} r )} v^s \, \dwvi  \right)^{\frac{1}{s}}+\delta \\
& \quad \ \leq (\conspoin _2 \lambda _{\delta}^{\cpoin _2})^{-\frac{1}{q}} \theta \epsilon ^{-\frac{1-\theta}{\theta}} \left(\fint _{B(\pvr _0 , r)} v^q \, \dwvi  \right)^{\frac{1}{q}}+\epsilon   (1-\theta)\bar c 2^{\frac{\beta}{\theta}} N +\delta,
\end{aligned}
\end{equation*}
where the last two inequalities follow from the  property \eqref{34} of $\wghtvi $ and \eqref{45}. 

Choosing $\epsilon= [(1-\theta)\bar c 2^{\frac{\beta}{\theta} +1}  ]^{-1}$ and letting $\delta\rightarrow 0$, we get
$$
N \leq 2 (\conspoin _2 2^{-\cpoin _2})^{-\frac{1}{q}} \theta \left[(1-\theta)\bar c 2^{\frac{\beta}{\theta} +1}  \right] ^{\frac{1-\theta}{\theta}} \left(\fint _{B(\pvr _0 , r)} v^q \, \dwvi  \right)^{\frac{1}{q}}.
$$

Hence, from \eqref{45}, we have
$$
(1-\lambda)^{\frac{\beta}{\theta}} \left(\fint _{B( \pvr _0 , \lambda r) } v^s \, \dwvi  \right)^{\frac{1}{s}} \leq  2 (\conspoin _2 2^{-\cpoin _2})^{-\frac{1}{q}} \theta \left[(1-\theta)\bar c 2^{\frac{\beta}{\theta} +1}  \right] ^{\frac{1-\theta}{\theta}} \left(\fint _{B(\pvr _0 , r)} v^q \, \dwvi  \right)^{\frac{1}{q}}.
$$

This proves \eqref{46}.\end{proof}

As a consequence of Lemma \ref{47}, we have the
\begin{corollary} \label{127}
Assume that \eqref{33} is verified. Suppose  $0<q<\gamma <s $ and $ \beta \geq 0$. If a nonnegative function $v \in L^{\infty} (B(\pvr _0 , r  )  )$ satisfies 
\begin{equation*} 
 \| v\|_{L^\infty(B(\pvr _0, \lambda r^\prime)) } \leq \bar{c}  (1-\lambda)^{-\beta} \left( \fint _{B(\pvr _0 , r^{\prime} )} v^{\gamma} \, \dwvi  \right)^{\frac{1}{\gamma}}
\end{equation*}
for each  $0<r^{\prime} \leq r$ and $0<\lambda<1$, then
\begin{equation*}
 \| v\|_{L^\infty(B(\pvr _0, \lambda r)) } \leq  C (1-\lambda)^{-\frac{\beta}{\theta}} \left(\fint _{B(\pvr _0 , r)} v^q \, \dwvi  \right)^{\frac{1}{q}}.
\end{equation*}
for all $0<\lambda<1$. Here, {\footnotesize $C:=2 (\conspoin _2 2^{-\cpoin _2})^{-\frac{1}{q}} \theta \left[(1-\theta)\bar c 2^{\frac{\beta}{\theta} +1}  \right] ^{\frac{1-\theta}{\theta}}$, and $\theta \in(0,1)$ is such that
$$
\frac{1}{\gamma}=\frac{\theta}{q}.
$$}
\end{corollary}
\subsection{The existence of solutions for obstacle problems}

Let $\obsi : \Omega \rightarrow [-\infty,+\infty]$ be a function, and let $\obsii \in \wo$. Define
$$
\kob :=\left\{v \in \wo \:|\: v \geq \obsi \text { a.e. in } \Omega, v-\obsii \in  \wob \right\} .
$$
If $\obsi =\obsii$, we write $\mathscr{K}_{\obsi  , \obsi }(\Omega)=\mathscr{K}_{\obsi} (\Omega)$.

The obstacle problem is to find a function $u\in \kob$ such that
\begin{equation} \label{4}
\int_{\Omega}  \mathscr{A}( \cdot , \grad  u) \multg  \deriv  (v-u)  \, \dwvi    \geq 0
\end{equation}
or 
\begin{equation} \label{111}
\int_{\Omega}  \left[-\mathscr{A}( \cdot , - \grad  u)\right] \multg  \deriv  (v-u)  \, \dwvi    \geq 0,
\end{equation}
for all $v \in \kob$. The function $u \in \wo$ is called a solution to the obstacle problem with obstacle $\obsi$ and boundary values $\obsii$.

Similar to \cite[Theorem 3.4]{ono1997solutionsquasi} and \cite[Theorem 3.2]{fu2015removable} (for the proof, see these works), we obtain:
\begin{proposition} \label{105} Let $\Omega _0 \subset \subset \Omega$ be a open set. Suppose that $\mathscr{K}_{\obsi  , \obsii} (\Omega _0) \neq \emptyset$ and conditions \ref{16} - \ref{17} are satisfied. Then there exists a unique solution $u \in \mathscr{K}_{\obsi  , \obsii} (\Omega _0) $ to the obstacle problem \eqref{4}.
\end{proposition}

The proofs of the following theorems are similar to those on pp. 61-62 in \cite{heinonen1993nonlinear} and are not repeated here.
\begin{lemma} Let $u \in \kob$ be a solution of the obstacle problem \eqref{4}. If $v \in \wo$ is a supersolution in $\Omega$ such that $\min \{u, v\} \in \kob$, then $v \geq u$ a.e. in $\Omega$.
\end{lemma}

\begin{lemma} If $u$ and $v$ are two supersolutions of \eqref{13}, then $\min \{u, v\}$ is also a supersolution.
\end{lemma}

\begin{lemma} \label{99} Let $u \in \kob$ be a solution of the obstacle problem \eqref{4}. Then
$$
u \leq \underset{\pvr \in \Omega}{\operatorname{ess} \sup } \max \{\obsi (\pvr ), \obsii (\pvr)\} \quad \text { a.e. in } \Omega.
$$
\end{lemma}

Similary to \cite[Theorem 4.11]{harjulehto2006unbounded} (see also  \cite[Theorem 3.67]{heinonen1993nonlinear} for the fixed exponent case) we have the 
\begin{proposition} \label{95} Let $\obsi: \Omega \rightarrow[-\infty, \infty)$ be continuous. Then the solution $u\in \kob$ of the obstacle problem \eqref{4} is continuous. Moreover, $u$ is a solution in the open set $\{\pvr \in \Omega \:|\: u(\pvr)>\obsi(\pvr)\}$.
\end{proposition}

\section{Basic estimates} \label{122}

The main results of this section are Propositions \ref{67} and \ref{84}, and  \ref{82}. As a consequence of these results, we have the following:
\begin{proposition}
Let $\obsi: \Omega \to [-\infty, \infty)$ be continuous. Let $\{\ue\} \subset \lwo\cap \kob$ be a family of functions, and let $\uu \in \lwo \cap L^{\infty}_{\loc}(\Omega)$ be an element. Assume that property \ref{94} is satisfied and that $\uu \in \kob$ is a solution to the obstacle problem \eqref{4}. Then, for all $\ell$, $\ue$ is continuous.
\end{proposition}
The proof of this proposition follows the same argument as in \cite[Theorem 4.3]{fu2015removable}, \cite[Theorem 3.67]{heinonen1993nonlinear}, or \cite[Theorem 4.11]{harjulehto2007obstacle}, and is therefore omitted.

For a chart $\psi : \tilde U \rightarrow \psi (\tilde{U}) \subset \Omega$, we define
\begin{gather*}
\|d( \psi ^{-1})  \|:= \sup \{  |d( \psi ^{-1}) _{\pvr} v| _{\mec} \:|\: (\pvr,v) \in T \psi (\tilde U), |v|_{\fnf} (\pvr)= 1\}, \\
 \|d \psi  \|:= \sup \{  |d \psi _{\evrx} \evrz | _{\fnf} (\psi (\evrx ))\:|\:  \evrx \in \tilde U, |\evrz|_{\mec} = 1\},
\end{gather*}
and  $\consman _1 >0$ is a constant such that
\begin{equation} \label{114}
 \consman _1 \geq \|d( \psi ^{-1})  \| \|d \psi  \| .
\end{equation}

\begin{lemma} \label{3}
Let $0<\sigma <\rho$ such that $B(\pvr _0 , \rho)\subset \subset \psi (\tilde U)$. There exists a smooth function $\eta  : B(\pvr _0 , \rho )  \rightarrow [0,1]$ such that  $\eta  = 1$ in $ B(\pvr _0 , \sigma) $,  $\spt \eta  \subset B(\pvr _0 , \rho ) $, and $|\deriv  \eta  | _{\fnfs} \leq \consman _0  \consman _1 (\rho - \sigma)^{-1}$ in $B(\pvr _0 , \rho)$, where $\consman _0 := \consman _0 (n)>0$.
\end{lemma}

\begin{proof}
Set $V:= \psi ^{-1} (B(\pvr _0 , \sigma ) )$ and $W:=\psi ^{-1} ( B (\pvr _0 ,  \rho)  )$. There exists a smooth function $\eta  _0 : V \subset \mathbb R ^n \rightarrow [0,1]$ such that $\eta  _0 =1$ in $V$, $\spt \eta  _0 \subset W$, and $|\nabla _{\mec}  \eta  _0 | _{\mec}\leq C _1(n) (\dist (V , \mathbb{R}^n \backslash W) )^{-1}$.

Let $\evrx  \in \partial V$ and $\evry\in \partial W$ such that $\dist (V , \mathbb{R}^n \backslash W)=|\evrx - \evry| _{\mec}$. For $\alpha (t) :=\psi ((1-t)\evrx + t \evry)$, we have
\begin{equation} \label{48}
\rho - \sigma \leq \int _0 ^1 \fnf ( \alpha , \alpha ^{\prime}) \, \dt \leq  \|d \psi  \| |\evrx - \evry| _{\mec}.
\end{equation}

Then $\eta  = \eta  _0 \circ  \psi ^{-1}   : B(\pvr _0 , \rho) \rightarrow [0,1]$ satisfies $\eta  =1$ in $B(\pvr _0 , \sigma)$, $\spt \eta  \subset B(\pvr _0 ,  \rho)$, and
\begin{align*}
&|\deriv \eta |_{\fnfs } (\pvr)= \sup _{v\in T_{\pvr} \Omega \backslash \{0\}}  \frac{ \deriv \eta _{\pvr}(v)}{|v|_{\fnf}(\pvr)} \leq \|d (\psi ^{-1}) \|  |\nabla _{\mec} \eta  _0 |_{\mec } (\psi ^{-1} (\pvr))\\
 & \quad \ \leq C_1 \|d (\psi ^{-1}) \|  |\evrx - \evry|^{-1} _{\mec} \quad \text { in } \quad B(\pvr _0 ,  \rho)  .
\end{align*}

Using \eqref{48}, we prove the lemma.\end{proof}

\subsection{Upper bound for solutions of the obstacle problem and for supersolutions}

In this subsection, the objective is to prove Proposition \ref{67}. Before that, we show the auxiliary result:
\begin{lemma}
 Let \(N_0 \geq |N|\), \(q \geq 0\), and \(B(\pvr_0, \tr)  \subset \Omega\) with \(0 < \tr \leq 1\). Let \(\eta \in C_0^\infty(B(\pvr_0, \tr))\) with \(0 \leq \eta \leq 1\). Consider a family $\{\ue\} \subset \lwo  $ and an element $\uu \in \lwo \cap L^{\infty} _{\loc} (\Omega)$. Then, we have:
\begin{enumerate}[label=(\roman*)]
\item \label{1} Assume that property \ref{94} is satisfied and $\uu, \ue \in \kob$ for all $\ell$. If $\uu $ is a solution of the obstacle problem \eqref{4} with the obstacle $\obsi \leq N$ in $B(\pvr_0, \tr)$, then
\begin{equation} \label{25}
\begin{aligned}
& \int_{\Omega} \left[ (\ue -N)^{+}+r \right]  ^q|  \grad  (\ue -N)^{+}| _{\fnf} ^{p_{\tr } ^{-}} \eta  ^{p_{\tr } ^{+}} \, \dwvi  \\
&\quad  \  \leq   \cones _1 \int_{\Omega}  [(\ue -N)^{+} + r]  ^{p+q} |\deriv  \eta  |_{\fnfs} ^{p } \, \dwvi     +  \cones _2 \int_{\Omega }    \left[ (\ue -N)^{+}+r \right]  ^q \eta  ^{p_{\tr } ^{+}} \, \dwvi ,
\end{aligned}
\end{equation}
where  
{\footnotesize \begin{align*}
\cones _1 := \frac{4 \consfii  p^{+} _{\tr } }{ \consfi  p^{-} _{\tr }} \max \left\{ \left[2 \consfii  (p^{+} _{\tr } -1 ) / \consfi  \right]^{p^+ _{\tr } -1} , \left[2 \consfii  (p^{+} _{\tr } -1 ) / \consfi  \right]^{p^- _{\tr } -1}  \right\} + \frac{4\vinf  p^{+} _{\tr } }{ \consfi  p^{-} _{\tr }}
\end{align*}
and 
$$
\cones _2:= \frac{4\vinf [  (p^+ _{\tr } )^2-1 ] }{\consfi  p^+ _{\tr }} + 1 .
$$}

\item \label{2} Assume that \eqref{28} (of property \ref{94}) holds. If $\uu $ is a supersolution of \eqref{13} in $\Omega$, then
\begin{equation} \label{22}
\begin{aligned}
& \int_{\Omega} \left[ |(\ue -N)^{-} |+r \right]  ^q|  \grad  (\ue -N)^{-}| _{\fnf}^{p_{\tr } ^{-}} \eta  ^{p_{\tr } ^{+}} \, \dwvi  \\
&\quad  \  \leq \cones _1  \int_{\Omega}  \left[ |(\ue -N)^{-}| + r \right]  ^{p +q} |\deriv  \eta  |_{\fnfs}^{p} \, \dwvi    + \cones _2 \int_{\Omega}  \left[ |(\ue -N)^{-}|+r \right]  ^q \eta  ^{p_{\tr } ^{+}} \, \dwvi   .
\end{aligned}
\end{equation}
 
\end{enumerate}
In \ref{1} and \ref{2},  $p_{\tr } ^{-}:=\inf _{ B( \pvr _0, \tr )   } p$, $p_{\tr } ^{+}:=\sup _{B( \pvr _0, \tr )   } p$, $(\ue -N)^{+}:=\max \{\ue -N, 0\}$, and $(\ue -N)^{-}=\min \{\ue -N, 0\}$.
\end{lemma}

\begin{proof}

\ref{1} {\it Steep 1.} Let $h \geq - r$ and define $\zeta =-[(\ue  -N-h)^{+} -r]\eta  ^{p_{\tr }^{+}}$. Then $\ue +\zeta =\ue -[(\ue - N -h)^{+} -r] \eta  ^{p_{\tr }^{+}} \geq \obsi$ in $B(\pvr _0, \tr)$, and $\ue +\zeta -\obsii \in \wob$. 

We have
\begin{gather*}
\deriv  \zeta =-\eta  ^{p_{\tr }^{+}} \deriv (\ue -N-h)^{+}-p_{\tr } ^{+} \eta  ^{p_{\tr } ^{+}-1} \left[(\ue - N - h)^{+} -r\right]\deriv \eta  .
\end{gather*}

Taking $v:=\ue +\zeta $:
\begin{equation*}
\begin{aligned}
&\int_{\Omega}   \mathscr{A}(\cdot , \grad  \ue )  \multg  \deriv (v - \ue)  \, \dwvi    \\
&\quad \ = \int_{\Omega}  \mathscr{A}( \cdot , \grad  \uu ) \multg  \deriv  (v-\uu )  \\
&\quad \quad \ +  \left[ \mathscr{A}( \cdot , \grad  \ue )  - \mathscr{A}( \cdot , \grad  \uu ) \right] \multg  \deriv  (v-\uu )  +
\mathscr{A}( \cdot , \grad  \ue )  \multg  \deriv  (\uu-\ue )  \, \dwvi \\
&\quad  \ =   \int_{\Omega}   \mathscr{A}( \cdot , \grad  \uu ) \multg  \deriv  (v-\uu )  \, \dwvi   +\funcion (\ell , N, h),
\end{aligned}
\end{equation*}
where 
\begin{align*}
&\funcion (\ell , N , h):=   \int _{\Omega} \left[ \mathscr{A} ( \cdot , \grad  \ue ) - \mathscr{A} ( \cdot , \grad  \uu) \right] \multg   \left\{ \deriv (\ue - \uu) - \eta ^{p^+ _{\tr }} \deriv  (\ue - N -h )^+ \right. \\
& \quad \ \left. - p^+ _{\tr } \eta ^{p^+ _{\tr } - 1} \left[ (\ue - N - h)^+ - r \right]\deriv \eta\right\}\, \dwvi  + \int _{\Omega}   \mathscr{A} (\cdot ,\grad  \ue ) \multg  \deriv (\uu -\ue)  \, \dwvi ,
\end{align*}

Since $\uu \in \kob$ is a solution of the obstacle problem \eqref{4}, and $\ue \in \kob$:
\begin{equation*}
\int_{\Omega}    \mathscr{A}(\cdot , \grad  \ue )   \multg \left\{    -\eta  ^{p_{\tr } ^{+}} \deriv (\ue -N-h)^{+}-p_{\tr } ^{+} \eta  ^{p_{\tr } ^{+}}\left[ (\ue - N -h)^{+} -r\right]\deriv  \eta     \right\} \, \dwvi   \geq \funcion (\ell , N, h) ,
\end{equation*}

Let $\Omega _{\ell, N,h}=\{\pvr \in \Omega \:|\: \ue (\pvr)>N+h\}$. By \ref{16} and \ref{21},
\begin{align}
& \consfi  \int_{\Omega_{\ell ,  N,h}}|\grad  \ue |_{\fnf}^{p} \eta  ^{p_{\tr } ^{+}} \, \dwvi \nonumber \\
& \quad \ \leq  p^{+} _{\tr } \consfii  \int_{\Omega _{\ell , N,h}}|\grad  \ue |_{\fnf }^{p-1}|\deriv \eta  |_{\fnfs } \left|  (\ue -N-h)^+ -r \right|  \eta  ^{p_{\tr } ^{+}-1} \, \dwvi  - \funcion (\ell , N , h) \label{65}\\
& \quad \ \leq  p^{+} _{\tr } \consfii  \int_{\Omega _{\ell , N,h}}  \frac{p^+ _{\tr } -1}{p^+ _{\tr }} \epsilon |\grad  \ue |_{\fnf }^{p} \eta ^{p^+ _{\tr }}  + \frac{\max \{ \epsilon ^{-p^+ _{\tr } +1 } , \epsilon ^{-p^- _{\tr } +1 } \} }{ p^- _{\tr } }  \left|   (\ue -N-h)^+ -r \right|^{p}   \eta  ^{p_{\tr } ^{+}-p} |\deriv  \eta  |_{\fnfs } \, \dwvi  \nonumber\\
&\quad \quad \ - \funcion (\ell , N , h) . \label{5}
\end{align}

Taking $\epsilon =\frac{\consfi }{2 \consfii  (p^{+} _{\tr } -1 )}$ in \eqref{5}, we have
\begin{equation} \label{6}
\begin{aligned}
& \frac{\consfi }{2} \int_{\Omega _{\ell , N , h} }|\grad  \ue |_{\fnf}^{p} \eta  ^{p_{\tr } ^{+}} \, \dwvi  \leq \frac{ \consfii  p^{+} _{\tr } }{ p^{-} _{\tr }} \max \left\{ \left[2 \consfii  (p^{+} _{\tr } -1 ) / \consfi  \right]^{p^+ _{\tr } -1} , \left[2 \consfii  (p^{+} _{\tr } -1 ) / \consfi  \right]^{p^- _{\tr } -1}  \right\} \\
&\quad \quad \ \cdot  \int_{\Omega_{\ell , N , h}}  \left| (\ue -N-h)^{+} - r\right|  ^{p} \eta  ^{p_{\tr } ^{+}-p}|\deriv  \eta  |_{\fnfs }^{p} \, \dwvi      - \funcion (\ell , N , h) .
\end{aligned}
\end{equation}

$ $


{\it Steep 2.} Now we will estimate the last term of \eqref{6}.

From \eqref{10} and \eqref{113},
\begin{align}
& - \funcion (\ell , N , h) \nonumber\\
& \quad  =     \int _{\Omega _{\ell , N , h}} \left[ \mathscr{A} ( \cdot , \grad  \ue ) - \mathscr{A} ( \cdot , \grad  \uu) \right] \multg  \left\{  \eta ^{p^+ _{\tr }} \deriv  (\ue - N -h )^+  + p^+ _{\tr } \eta ^{p^+ _{\tr } - 1} \left[(\ue - N - h)^+ - r \right]\deriv  \eta\right\}\, \dwvi  \nonumber \\
& \quad  \leq    \int _{\Omega _{\ell , N ,h}} \left| \mathscr{A} ( \cdot , \grad  \ue ) - \mathscr{A} ( \cdot , \grad  \uu) \right|_{\fnf}  \eta ^{p^+ _{\tr }} | \grad  \ue   |_{\fnf} \nonumber\\
& \quad \quad   + p^+ _{\tr } \left| \mathscr{A} ( \cdot , \grad  \ue ) - \mathscr{A} ( \cdot , \grad  \uu) \right| _{\fnf} \eta ^{p^+ _{\tr } - 1} \left| (\ue - N + h)^+ - r\right| | \deriv \eta |_{\fnfs }\, \dwvi  . \label{29}
\end{align}

From \eqref{28}, 
\begin{equation*}
\left| \mathscr{A} ( \cdot , \grad  \ue ) - \mathscr{A} ( \cdot , \grad  \uu) \right|_{\fnf} \leq \vinf < \frac{\consfi  p_{\tr} ^-}{4} \quad \text { in } \quad \Omega.
\end{equation*}

We also have, in $B(\pvr _0 , \tr)$,
\begin{equation*} 
|\grad  \ue |_{\fnf}\leq \frac{1}{p^- _{\tr}} |\grad  \ue|_{\fnf}^p + \frac{p^+ _{\tr} -1}{p^+ _{\tr}}
\end{equation*}
and 
\begin{align*}
&\eta ^{p^+_{\tr} -1}  \left|(\ue - N -h )^+ -r\right| |\deriv  \eta|_{\fnfs }\leq \frac{1}{p} \left| (\ue - N - h)^+-r \right|^p |\deriv \eta|_{\fnfs }^p + \frac{p-1}{p} \eta ^{\frac{p^+ _{\tr} -1}{p-1} p}\\
& \quad \ \leq \frac{1}{p^- _{\tr}} \left| (\ue - N -h)^+ -r \right|^p |\deriv  \eta|_{\fnfs }^p + \frac{p^+_{\tr}-1}{p^+ _{\tr}} \eta ^{p^+ _{\tr}}.
\end{align*}

Consequently, from \eqref{29},
\begin{equation} \label{30}
\begin{aligned}
&  - \funcion (\ell , N , h) \\
& \quad \ \leq  \frac{\consfi }{4} \int _{\Omega _{\ell , N , h}}|\grad  \ue|_{\fnf}^p \eta ^{p^+_{\tr}}\, \dwvi  + \frac{\vinf  (p^+ _{\tr} -1)}{p^+ _{\tr}} \int _{\Omega _{\ell , N , h}} \eta ^{p^+_{\tr}} \, \dwvi  \\
& \quad \quad \ + \frac{\vinf p^+ _{\tr}}{p^- _{\tr}} \int _{\Omega _{\ell , N , h}}  \left| (\ue - N - h)^+ -r \right|^p |\deriv  \eta| _{\fnfs }^p \, \dwvi  + \vinf  (p^+_{\tr}-1 ) \int _{\Omega _{\ell , N , h}}  \eta ^{p^+ _{\tr}} \, \dwvi  .
\end{aligned}
\end{equation}

$ $

{\it Steep 3.} Using \eqref{6} and \eqref{30},
\begin{equation} \label{31}
 \frac{\consfi }{4} \int_{\Omega _{\ell , N , h} }|\grad  \ue |_{\fnf}^{p} \eta  ^{p_{\tr } ^{+}} \, \dwvi  \leq \cones  \int_{\Omega_{\ell , N , h}}  \left| (\ue -N-h)^{+} - r\right|  ^{p} |\deriv  \eta  |_{\fnfs }^{p} \, \dwvi     + \frac{\vinf  [(p^+ _{\tr})^2 -1]}{p^+ _{\tr}} \int _{\Omega _{\ell , N , h}} \eta ^{p^+_{\tr}} \, \dwvi ,
\end{equation}
where 
$$
\cones = \frac{ \consfii  p^{+} _{\tr } }{ p^{-} _{\tr }} \max \left\{ \left[2 \consfii  (p^{+} _{\tr } -1 ) / \consfi  \right]^{p^+ _{\tr } -1} , \left[2 \consfii  (p^{+} _{\tr } -1 ) / \consfi  \right]^{p^- _{\tr } -1}  \right\} + \frac{\vinf  p^{+} _{\tr } }{ p^{-} _{\tr }}.
$$

By   Proposition \ref{7}   (see also  \cite[equation (3.31)]{heinonen1993nonlinear}), we have 
\begin{equation*}
\begin{aligned}
&\int _{\Omega} \left[(\ue - N) ^+ +r \right]^q |\grad  (\ue - N) ^+|_{\fnf}^p \eta ^{p^+ _{\tr }} \, \dwvi  \\
& \quad  \ = q\int ^\infty _0 t^{q-1} \int _{ \{ (\ue - N)^+ +r >t \} } |\grad  (\ue - N)^+|_{\fnf} ^p \eta ^{p^+ _{\tr }} \, \dwvi  \dt\\
&\quad \ = q\int ^\infty _{-r} (h+r)^{q-1} \int _{ \{ \ue - N  > h \} } |\grad  (\ue - N)^+|_{\fnf} ^p \eta ^{p^+ _{\tr }} \, \dwvi  \dhd .
\end{aligned}
\end{equation*}

From \eqref{31},
\begin{equation} \label{8}
\begin{aligned}
& \frac{\consfi }{4}\int_{\Omega} \left[  (\ue -N)^{+}+r\right]  ^q|  \grad (\ue -N)^{+} |  _{\fnf} ^{p} \eta  ^{p_{\tr } ^{+}} \, \dwvi  \\
& \quad \ \leq q \frac{\consfi }{4} \int_{-r}^{\infty}(h+r)^{q-1} \int_{\Omega_{\ell,N,h}}|\grad  \ue |_{\fnf} ^{p} \eta  ^{p_{\tr } ^{+}} \, \dwvi  \dhd \\
&\quad  \  \leq q \cones \int_{-r}^{\infty}(h+r)^{q-1}  \int_{\Omega_{\ell , N , h}}  \left|(\ue -N-h)^{+}-r\right|  ^{p} |\deriv  \eta  |_{\fnfs }^{p} \, \dwvi  \dhd  \\
&\quad  \quad \   +  \frac{q\vinf  [(p^+ _{\tr})^2 -1]}{p^+ _{\tr}} \int_{-r}^{\infty}(h+r)^{q-1} \int _{\Omega _{\ell , N , h}} \eta ^{p^+_{\tr}} \, \dwvi \dhd
\end{aligned}
\end{equation}

Consequently, taking into account that $|(\ue - N - h)^+ - r|\leq (\ue - N)^+ +r$ in $\Omega _{\ell , N , h}$:
\begin{equation} \label{9}
\begin{aligned}
& \frac{\consfi }{4}\int_{\Omega} \left[  (\ue -N)^{+}+r\right]  ^q|  \grad (\ue -N)^{+} | _{\fnf} ^{p} \eta  ^{p_{\tr } ^{+}} \, \dwvi  \\
&\quad  \  \leq q\cones \int_{0}^{\infty}t^{q-1} \int_{\Omega_{\ell , N , t-r}}  [(\ue -N)^{+} + r]  ^{p} |\deriv  \eta  |_{\fnfs }^{p} \, \dwvi  \dt \\
&\quad \quad \  +  \frac{q \vinf [  (p^+ _{\tr } )^2-1 ] }{p^+ _{\tr }} \int_{0}^{\infty}t ^{q-1}    \int_{\Omega _{\ell , N, t-r}}     \eta ^{p^+ _{\tr }} \, \dwvi  \dt .
\end{aligned}
\end{equation}

Employing $|a|^{p^-_{\tr } } \leq |a|^p + 1$  $\forall a\in \mathbb{R}$ and  Proposition \ref{7}. From \eqref{9}, we have
\begin{equation}\label{11}
\begin{aligned}
& \frac{\consfi }{4}\int_{\Omega} \left[ (\ue -N)^{+}+r \right]  ^q|  \grad  (\ue -N)^{+}| _{\fnf} ^{p_{\tr } ^{-}} \eta  ^{p_{\tr } ^{+}} \, \dwvi  \\
& \quad \ \leq  \frac{\consfi }{4} \int_{\Omega} \left[ (\ue -N)^{+}+r \right]  ^q|  \grad  (\ue -N)^{+}| _{\fnf} ^{p} \eta  ^{p_{\tr } ^{+}} \, \dwvi  +\frac{\consfi }{4}\int_{\Omega}  \left[ (\ue -N)^{+}+r \right]  ^q \eta  ^{p_{\tr } ^{+}} \, \dwvi  \\
&\quad  \  \leq   \cones  \int_{\Omega}  \left[(\ue -N)^{+} + r \right]  ^{p +q} |\deriv  \eta  |_{\fnfs } ^{p} \, \dwvi   \\
&\quad \quad \  +  \left\{\frac{\vinf [  (p^+ _{\tr } )^2-1 ] }{p^+ _{\tr }} + \frac{\consfi }{4}    \right\}\int_{\Omega }    \left[ (\ue -N)^{+}+r \right]  ^q \eta  ^{p_{\tr } ^{+}} \, \dwvi 
\end{aligned}
\end{equation}

Which  conclude the proof of \ref{1}.\\

\ref{2}   Let $h \geq -r$ and define $\zeta =-(\ue - N + h)^{-} \eta  ^{p_{\tr } ^{+}}$, then  $\zeta  \geq 0$ and $\zeta  \in \wob$. So $\zeta $ is a test function for \eqref{14}, then
$$
\begin{aligned}
& \int _{\Omega} \mathscr{A}(\cdot , \grad  \ue ) \multg  \deriv  \zeta  \, \dwvi  = \int _{\Omega} \mathscr{A}(\cdot , \grad  \uu ) \multg  \deriv  \zeta  \, \dwvi    \\
& \quad \ + \int _{\Omega} \left[ \mathscr{A}(\cdot , \grad  \ue ) - \mathscr{A}(\cdot , \grad  \uu ) \right] \multg  \deriv  \zeta  \, \dwvi  .
\end{aligned}
$$
Hence,
$$
\begin{aligned}
& \int_{\Omega} \mathscr{A}(\cdot , \grad  \ue ) \multg  \left[  -\eta  ^{p_{\tr } ^{+}} \deriv (\ue -N+h)^{-}-p_{\tr } ^{+} \eta  ^{p_{\tr } ^{+}-1}(\ue -N+h)^{-} \deriv \eta  \right]   \, \dwvi  
  \geq \hat \funcion ( \ell , N , h) ,
\end{aligned}
$$
where
$$
\begin{aligned}
&\hat \funcion ( \ell , N , h) := \int _{\Omega} \left[ \mathscr{A}(\cdot , \grad  \ue ) - \mathscr{A}(\cdot , \grad  \uu ) \right] \multg  \left[   -\eta  ^{p_{\tr } ^{+}} \deriv (\ue -N+h)^{-} \right.\\
& \quad \quad \ \left. - p_{\tr } ^{+} \eta  ^{p_{\tr } ^{+}-1}(\ue -N+h)^{-} \deriv  \eta  \right] \, \dwvi .
\end{aligned}
$$

Let $\Omega_{\ell , N , h}^{\prime} :=\{\pvr \in \Omega\:|\: \ue (\pvr)<N-h\}$. By \ref{16} and \ref{21}, we have
$$
\begin{aligned}
& \consfi  \int_{\Omega _{\ell , N , h} ^{\prime}}|\grad  \ue | _{\fnf} ^{p} \eta  ^{p_{\tr } ^{+}} \, \dwvi  \leq  p^{+} _{\tr }\consfii  \int_{\Omega_{\ell , N , h}^{\prime}}|\grad  \ue | _{\fnf} ^{p-1}|\deriv  \eta  |_{\fnfs } |  (\ue -N+h)^{-}|   \eta  ^{p_{\tr } ^{+}-1} \, \dwvi  - \hat \funcion ( \ell , N , h)
\end{aligned}
$$

Proceeding similarly as we did for \eqref{65}, we will obtain the following inequality (see \eqref{31}):
\begin{equation} \label{90}
\frac{\consfi }{4} \int_{\Omega ^\prime _{\ell , N , h} }|\grad  \ue | _{\fnf}^{p} \eta  ^{p_{\tr } ^{+}} \, \dwvi  \leq \cones  \int_{\Omega ^\prime _{\ell , N , h}}  \left| (\ue -N-h)^{-} \right|  ^{p} |\deriv  \eta  |_{\fnfs }^{p} \, \dwvi     + \frac{\vinf  [(p^+ _{\tr})^2 -1]}{p^+ _{\tr}} \int _{\Omega ^\prime  _{\ell , N , h}} \eta ^{p^+_{\tr}} \, \dwvi .
\end{equation}
With which we obtain \eqref{22}.\end{proof}


\begin{proposition} \label{67} Suppose the conditions \eqref{78} - \eqref{26} are verified. Assume that   $r \leq \sigma<\rho \leq 2r \leq 2$ and $B(\pvr _0 , 2r) \subset \subset \psi (\tilde U)$.  Let  $|N| \leq N_0$ and  $1<\gamma<1_{\wghtz}$. Consider a family $\{\ue\} \subset \lwo  $ and an element $\uu \in \lwo \cap L^{\infty} _{\loc} (\Omega)$. Then, we have:
\begin{enumerate}[label=(\roman*)]
\item  \label{23} Assume that property \ref{94} is satisfied and $\uu, \ue \in \kob$ for all $\ell$. If $\uu$ is a solution of the obstacle problem \eqref{4}  with the obstacle $\obsi \leq N$ in $B(\pvr _0 , 2r)$, then
\begin{equation} \label{86}
\begin{aligned}
&\underset{  B(\pvr _0 , \sigma )}{\operatorname{ess} \sup } \ \left\{ (\ue -N)^{+}+r \right\}\\
&  \quad \ \leq (\cones _3 1_{\wghtz}/ \gamma)^{\cones _4}        \cones _6 ^{ \frac{ 1_{\wghtz}}{ p_{2r}^{-} (1_{\wghtz} - \gamma) } }  (p_{2r} ^-)^{ \frac{p_{2r} ^+ 1_{\wghtz}}{p_{2r} ^- (1_{\wghtz} - \gamma)} }   \left( \frac{\rho}{\rho-\sigma}\right)^{\frac{p_{2r} ^+ 1_{\wghtz}}{p_{2r} ^- (1_{\wghtz} - \gamma)} }   \\
 & \quad \quad \ \cdot \left\{ \fint _{B(\pvr _0 , \rho )} \left[ (\ue -N)^{+}+r \right] ^{\gamma p_{2r}^{-}} \, \dwvi  \right\} ^{\frac{1}{\gamma p_{2r}^{-}}} ,
\end{aligned}
\end{equation}
where {\footnotesize $\cones _3:= [2(\consman _0 +1) \conspoin _2 ^{-1} (\conspoin _{\wghtz} +1)(\conspoin _4 +1)]^{\cpoin _2} $,  $\cones _4 := \frac{5p_{2r} ^+ 1_{\wghtz}}{p_{2r} ^-( 1_{\wghtz} - \gamma)} + \frac{p_{2r} ^+ 1_{\wghtz} \gamma}{p_{2r} ^- (1_{\wghtz} -\gamma)^2} $, 
$$
\cones _5:=  \left\{ \fint _{B(\pvr _0, 2r  )} \left[ (\ue -N)^{+}+r  \right]^{(p-p_{2r  }^{-}  ) \gamma ^{\prime}} \, \dwvi   \right\}^{\frac{1}{\gamma ^{\prime}}},
$$
and $\cones _6:= \cones _1 \cones_5 \max \{\consman _1 ^{p_{2r}^{-}}, \consman _1 ^{p_{2r}^{+}}\} + \cones _2 + ( \consman _1 p_{2r}^{+} )^{p_{2r}^{-}} +1$.}

\item \label{24} Assume that \eqref{28} (of property \ref{94}) holds. If $\uu$ is a supersolution of \eqref{13} in $\Omega$, then
\begin{equation*}
\begin{aligned}
&\underset{  B(\pvr _0 , \sigma )}{\operatorname{ess} \sup } \ \left\{ |(\ue -N)^{-}|+r \right\}\\
&  \quad \ \leq (\cones _3 1_{\wghtz}/ \gamma)^{\cones _4}        \cones _8 ^{ \frac{ 1_{\wghtz}}{ p_{2r}^{-} ( 1_{\wghtz} - \gamma )} }  (p_{2r} ^-)^{ \frac{p_{2r} ^+ 1_{\wghtz}}{p_{2r} ^- (1_{\wghtz} - \gamma)} }   \left( \frac{\rho}{\rho-\sigma}\right)^{\frac{p_{2r} ^+ 1_{\wghtz}}{p_{2r} ^- (1_{\wghtz} - \gamma)} }   \\
 & \quad \quad \ \cdot \left\{ \fint _{B(\pvr _0 , \rho )} \left[ |(\ue -N)^{-} |+r \right] ^{\gamma p_{2r}^{-}} \, \dwvi  \right\} ^{\frac{1}{\gamma p_{2r}^{-}}} ,
\end{aligned}
\end{equation*}
where  {\footnotesize
$$
\cones _7:=  \left\{ \fint _{B(\pvr _0, 2r  )} \left[ |(\ue -N)^{-} |+r  \right]^{(p-p_{2r  }^{-}  ) \gamma ^{\prime}} \, \dwvi   \right\}^{\frac{1}{\gamma ^{\prime}}},
$$
and $\cones _8:= \cones _1 \cones_7 \max \{\consman _1 ^{p_{2r}^{-}} , \consman _1 ^{p_{2r}^{+}}\} + \cones _2 + ( \consman _1 p_{2r}^{+})^{p_{2r}^{-}} +1 $.}
\end{enumerate}

In \ref{23} and \ref{24},  $p_{2r } ^{-}:=\inf _{B( \pvr _0, 2r )  } p$ and $p_{2r } ^{+}:=\sup _{B( \pvr _0, 2r )  } p$.

\end{proposition}
\begin{proof}
\ref{23}  We take \( \tr = 2r \) in \eqref{25} and the function \( \eta \in C_0^\infty(B(\pvr_0, \rho)) \) from Lemma \ref{3}, which satisfies: \( 0 \leq \eta \leq 1 \), \( \eta = 1 \) in \( B(\pvr_0, \sigma) \), and $|\deriv  \eta|_{\fnfs } \leq \consman _0 \consman _1/(\rho-\sigma)$. We also consider \( q = \beta - p_{2r}^{-} \) in \eqref{25}, with \( \beta \geq p_{2r}^{-} \).
\begin{equation}  \label{32} 
\begin{aligned}
& \int_{\Omega} \left[ (\ue -N)^{+}+r \right]  ^{\beta-p_{2r  }^{-}}|  \grad  (\ue -N)^{+}| _{\fnf} ^{p_{2r } ^{-}} \eta  ^{p_{2r } ^{+}} \, \dwvi  \\
&\quad  \  \leq \cones _1   \int_{\Omega}  [(\ue -N)^{+} + r]  ^{p + \beta-p_{2r  }^{-}} |\deriv  \eta  |_{\fnfs }^{p} \, \dwvi    + \cones _2\int_{\Omega}  \left[ (\ue -N)^{+}+r \right]  ^{\beta-p_{2r  }^{-}} \eta  ^{p_{2r } ^{+}} \, \dwvi 
\end{aligned}
\end{equation}

Applying  inequality \eqref{27} to the function $[(\ue -N)^{+}+r  ]^{\beta/p_{2r  } ^-} \eta^{p_{2r  }^{+} / p_{2r  } ^-}$, we get
\begin{align*}
&\left( \fint _{B(\pvr _0, 2r ) } \left\{ [(\ue -N)^{+}+r  ]^{\beta / p_{2r  } ^-} \eta ^{p_{2r  }^{+} / p_{2r  }^{-}}    \right\}       ^{ p^- _{2r} 1_{\wghtz}} \, \dwvi    \right)^{\frac{1}{ 1_{\wghtz}} }\\
&\quad \ \leq (2 p^-_{2r} \conspoin  _{\wghtz}  r )^{p^- _{2r}}\fint_{B(\pvr _0, 2r   ) } \left| \grad  \left\{ [(\ue -N)^{+}+r  ]^{\beta / p_{2r  }^{-}} \eta^{p_{2r  }^{+} / p_{2r  } ^-}  \right\}  \right|_{\fnf} ^{p_{2r  }^{-}} \, \dwvi  \\
& \quad  \ \leq (2\conspoin  _{\wghtz}  r)^{p^- _{2r}} 2^{p^- _{2r} -1 } (1+\beta)^{p_{2r }^{-}} \left\{ \fint _{B(\pvr _0, 2r   ) } \left[ (\ue -N)^{+}+r  \right]^{\beta-p_{2r  }^{-}}|\grad (\ue -N)^{+}  | _{\fnf} ^{p_{2r  }^{-}} \eta^{p_{2r  }^{+}} \, \dwvi   \right.\\
&\quad \quad \ \left. + (p_{2r  }^{+})^{p_{2r  }^{-}}\fint _{B(\pvr _0, 2r   ) }  \left[ (\ue -N)^{+}+r   \right]^\beta \eta^{p_{2r  }^{+}-p_{2r  }^{-}}|\deriv  \eta|_{\fnfs }^{p_{2r  }^{-}} \, \dwvi   \right\} .
\end{align*}

By \eqref{32} and \eqref{34}, we can obtain
\begin{equation} \label{37}
\begin{aligned}
& \left\{ \conspoin _2 2^{-\cpoin _2}\fint _{B(\pvr  _0, \sigma  )}\left[(\ue -N)^{+}+r  \right]^{\beta 1_{\wghtz}}   \, \dwvi   \right\} ^{\frac{1}{1_{\wghtz}}} \\
& \quad \ \leq (2\conspoin  _{\wghtz}      r )^{p^- _{2r}} 2^{p^- _{2r} -1 } (1+\beta)^{p_{2r }^{-}} \left\{ \cones _1   \fint_{B(\pvr _0, 2r  )}  [(\ue -N)^{+} + r]  ^{p + \beta-p_{2r  }^{-}} |\deriv  \eta  |_{\fnfs }^{p} \, \dwvi  \right. \\
&\quad \quad \  + \cones _2  \fint_{B(\pvr _0, \rho  )}  \left[ (\ue -N)^{+}+r \right]  ^{\beta-p_{2r  }^{-}} \eta  ^{p_{2r } ^{+}} \, \dwvi  \\
&\quad \quad \ \left. + (p_{2r  }^{+})^{p_{2r  }^{-}}\fint _{B(\pvr _0, \rho   ) }  \left[ (\ue -N)^{+}+r   \right]^\beta |\deriv  \eta|_{\fnfs  }^{p_{2r  }^{-}} \, \dwvi   \right\}.
\end{aligned}
\end{equation}

$ $

{\it Steep 2.} Next, we estimate the right-hand side of \eqref{37}.

By Hölder's inequality, for $\gamma \in (1,1_{\wghtz})$,
\begin{equation} \label{35}
\fint _{B(\pvr _0, \rho  )}\left[(\ue -N)^{+}+r  \right]^\beta|\deriv  \eta| _{\fnfs}^{p_{2r  }^{-}} \, \dwvi  \leq   r^{-p_{2r  }^{-}}\left(\frac{ \consman _0 \consman _1 \rho}{\rho-\sigma}  \right)^{p_{2r  }^{-}}\left\{  \fint _{B(\pvr _0, \rho  )}\left[(\ue -N)^{+}+r  \right]^{\beta \gamma} \, \dwvi \right\} ^{\frac{1}{\gamma}}
\end{equation}
and
\begin{equation} \label{36}
\fint _{B(\pvr _0, \rho  )} \left[ (\ue -N)^{+}+r  \right]^{\beta-p_{2r  }^{-}} \, \dwvi  \leq  r^{-p_{2r  }^{-}} \left\{ \fint _{B(\pvr _0, \rho  )} \left[ (\ue -N)^{+}+r  \right]^{\beta \gamma} \, \dwvi  \right\}^{\frac{1}{\gamma}},
\end{equation}
since $r\leq (\ue -N)^{+}+r$.

By \eqref{38}, we have 
\begin{equation*} \label{39}
|\deriv  \eta|_{\fnfs }^{p} \leq r^{-p}\left( \frac{\consman _0 \consman _1 r}{\rho - \sigma}\right)^{p} \leq r^{-p}\left( \frac{\consman _0 \consman _1 \rho}{\rho - \sigma}\right)^{p} \leq \conspoin _4 r^{-p^- _{2r}}\left( \frac{\consman _0 \consman _1 \rho}{\rho - \sigma}\right)^{p}\leq c_2 r^{-p^- _{2r}}\left( \frac{ \rho}{\rho - \sigma}\right)^{p^+_{2r}},
\end{equation*}
where $c_2:= \conspoin _4 \max \{(\consman _0 \consman _1 )^{p^-_{2r}} , ( \consman _0 \consman _1) ^{p^+_{2r}}\}$.

Then,
\begin{align}
& \fint _{B(\pvr _0, 2r  )} \left[(\ue -N)^{+}+r  \right]^{\beta-p_{2r  }^{-}+p}|\deriv  \eta|_{\fnfs }^{p} \, \dwvi  \nonumber  \\
& \quad \ \leq \frac{c_2 r^{-p_{2r  }^{-}}}{\wghtvi  (B(\pvr _0, 2r  ))}  \left(\frac{\rho}{\rho-\sigma}  \right)^{p_{2r  }^{+}}   \int _{B(\pvr _0, \rho  )} \left[(\ue -N)^{+}+r  \right]^{\beta-p_{2r  }^{-}+p}  \, \dwvi   \nonumber  \\
& \quad \ \leq c_2 r^{-p_{2r  }^{-}}\left(\frac{\rho}{\rho-\sigma}  \right)^{p_{2r  }^{+}}  \left\{ \fint _{B(\pvr _0, 2r  )} \left[ (\ue -N)^{+}+r  \right]^{(p-p_{2r  }^{-}  ) \gamma ^{\prime}} \, \dwvi   \right\}^{\frac{1}{\gamma ^{\prime}}} \nonumber \\
& \quad \quad \ \cdot \left\{ \fint _{B(\pvr _0, \rho  )}\left[ (\ue -N)^{+}+r  \right]^{\beta \gamma} \, \dwvi  \right\} ^{\frac{1}{\gamma}} \nonumber\\
& \quad \ = c_2 c_3 r^{-p_{2r  }^{-}}\left(\frac{\rho}{\rho-\sigma}  \right)^{p_{2r  }^{+}}   \left\{ \fint _{B(\pvr _0, \rho  )}\left[ (\ue - N)^{+}+r  \right]^{\beta \gamma} \, \dwvi  \right\}^{\frac{1}{\gamma}}, \label{40}
\end{align}
where 
$$
c_3:=  \left\{ \fint _{B(\pvr _0, 2r  )} \left[ (\ue -N)^{+}+r  \right]^{(p-p_{2r  }^{-}  ) \gamma ^{\prime}} \, \dwvi   \right\}^{\frac{1}{\gamma ^{\prime}}}.
$$

From \eqref{37} - \eqref{40}, we can get
\begin{equation*} 
\begin{aligned}
& \left\{ \conspoin _2 2^{-\cpoin _2}\fint _{B(\pvr  _0, \sigma  )}\left[(\ue -N)^{+}+r  \right]^{\beta 1_{\wghtz}}   \, \dwvi   \right\} ^{\frac{1}{1_{\wghtz}}} \\
& \quad \ \leq (2 \conspoin  _{\wghtz}       )^{p^- _{2r}} 2^{p^- _{2r} -1 } \left\{ \cones _1   c _2 c_3   + \cones _2   + (\consman _0 \consman _1 p_{2r  }^{+})^{p_{2r  }^{-}} \right\}\\
& \quad \quad \ \cdot (1+\beta)^{p_{2r }^{-}}  \left( \frac{\rho}{\rho-\sigma}  \right)^{p_{2r  }^{+}}   \left\{ \fint _{B(\pvr _0, \rho  )}\left[ (\ue - N)^{+}+r  \right]^{\beta \gamma} \, \dwvi  \right\}^{\frac{1}{\gamma}}.
\end{aligned}
\end{equation*}

 Define 
 $$
 \Psi(f, q, D):=\left( \fint_D f^q \, \dwvi   \right)^{\frac{1}{q}}.
 $$

Consequently,
\begin{equation}\label{42}
\begin{aligned}
& \Psi( (\ue -N)^{+}+r, \beta 1_{\wghtz }, B(\pvr _0, \sigma  )  ) \\
& \quad \leq c_4^{\frac{1}{\beta}}\left(1+\beta \right)^{\frac{ p_{2r  }^{+}}{\beta}} \left(\frac{\rho }{\rho - \sigma }  \right)^{\frac{ p_{2r  }^{+}}{\beta}} \Psi((\ue -N)^{+}+r, \beta \gamma , B(\pvr _0, \rho  )  ) ,
\end{aligned}
\end{equation}
if $r\leq \sigma <\rho \leq 2r$, where 
$$
c_4:=  (\conspoin _2 ^{-1}2^{\cpoin _2})^{1/1_{\wghtz}} (2 \conspoin  _{\wghtz}       )^{p^- _{2r}} 2^{p^- _{2r} -1 } [ \cones _1   c _2 c_3   + \cones _2   + (\consman _0 \consman _1 p_{2r  }^{+})^{p_{2r  }^{-}} ] .
$$

Taking  $ r_j=\sigma+2^{-j} (\rho-\sigma )$, $ \xi _j=(1_{\wghtz}/\gamma  )^j \gamma p_{2r}^{-}$, and $\beta= (1_{\wghtz}/\gamma  )^j  p_{2r}^{-}$ in \eqref{42}, we have
\begin{equation*}
\begin{aligned}
& \Psi( (\ue -N)^{+}+r, \xi_{j+1}, B(\pvr _0, r_{j+1}  )  ) \\
& \quad \leq c_4^{\frac{\gamma}{\xi_j}}\left(1+\frac{\xi_j}{\gamma}  \right)^{\frac{\gamma p_{2r  }^{+}}{\xi_j}} \left(\frac{r_{j} }{r_{j} - r_{j+1} }  \right)^{\frac{\gamma p_{2r  }^{+}}{\xi_j}} \Psi((\ue -N)^{+}+r, \xi_j, B(\pvr _0, r_j  )  ) .
\end{aligned}
\end{equation*}

By iterating this inequality, we have
\begin{equation*}
\begin{aligned}
&\underset{  B(\pvr _0 , \sigma )}{\operatorname{ess} \sup } \ \left\{ (\ue -N)^{+}+r \right\}\\
&  \quad \ \leq \prod_{j=0}^{\infty}\left[c_4^{\gamma  /\xi_j}\left(1+\frac{\xi_j}{\gamma}\right)^{\gamma p_{2r}^{+} / \xi_j}\left( \frac{2^{j+1}\rho}{\rho-\sigma}\right)^{\gamma p_{2r}^{+} / \xi_j }\right] \Psi((\ue -N)^{+}+r, \gamma p_{2r}^{-}, B(\pvr _0 , \rho ))\\
&  \quad \ \leq c_4^{\sum _{j=0}^{\infty}\gamma  /\xi_j}   2^{\sum _{j=0}^{\infty} (j+1)\gamma p_{2r}^{+} / \xi_j }  \left( \frac{\rho}{\rho-\sigma}\right)^{\sum _{j=0}^{\infty} \gamma p_{2r}^{+} / \xi_j } \prod_{j=0}^{\infty }   \left(1+\frac{\xi_j}{\gamma}\right)^{\gamma p_{2r}^{+} / \xi_j}   \\
 & \quad \quad \ \cdot \Psi((\ue -N)^{+}+r, \gamma p_{2r}^{-}, B(\pvr _0 , \rho ))  .
\end{aligned}
\end{equation*}

We note that  
\begin{equation*}
\sum _{j=0}^{\infty} \frac{\gamma}{ \xi_j} = \frac{1_{\wghtz}}{ p_{2r} ^-(1_{\wghtz} - \gamma)}
\quad \text { and } \quad
\sum _{j=0}^{\infty} j\frac{\gamma}{ \xi_j} = \frac{1_{\wghtz} \gamma}{ p_{2r} ^-(1_{\wghtz} - \gamma)^2}.
\end{equation*}

Then,
\begin{equation*}
\begin{aligned}
& \prod_{j=0}^{\infty}\left(1+\frac{\xi_j}{\gamma}\right)^{\gamma p_{2r}^{+} / \xi_j} \leq 2^{\sum _{j=0}^{\infty}\gamma p_{2r}^{+} / \xi_j} \prod_{j=0}^{\infty} \left( \frac{\xi_j}{\gamma}\right)^{ \gamma p_{2r}^{+} / \xi_j}\\
& \quad \ = (2 p^- _{2r})^{\sum _{j=0}^{\infty}\gamma p_{2r}^{+} / \xi_j}\left( \frac{1_{\wghtz}}{\gamma}\right)^{\sum_{j=0}^{\infty} j\gamma p_{2r}^{+} / \xi_j}.
\end{aligned}
\end{equation*}

Which implies,
\begin{equation*}
\begin{aligned}
&\underset{  B(\pvr _0 , \sigma )}{\operatorname{ess} \sup } \ \left\{ (\ue -N)^{+}+r \right\}\\
&  \quad \ \leq (c_5 1_{\wghtz}/ \gamma)^{c_6}        c_7 ^{ \frac{ 1_{\wghtz}}{ 1_{\wghtz} - \gamma } }  (p_{2r} ^-)^{ \frac{p_{2r} ^+ 1_{\wghtz}}{p_{2r} ^- (1_{\wghtz} - \gamma)} }   \left( \frac{\rho}{\rho-\sigma}\right)^{\frac{p_{2r} ^+ 1_{\wghtz}}{p_{2r} ^- (1_{\wghtz} - \gamma)} }   \\
 & \quad \quad \ \cdot \Psi((\ue -N)^{+}+r, \gamma p_{2r}^{-}, B(\pvr _0 , \rho ))  ,
\end{aligned}
\end{equation*}
where $c _5:= [2(\consman _0 +1) \conspoin _2 ^{-1} (\conspoin _{\wghtz} +1)(\conspoin _4 +1)]^{\cpoin _2}$,  $c_6 = \frac{5p_{2r} ^+ 1_{\wghtz}}{p_{2r} ^-( 1_{\wghtz} - \gamma)} + \frac{p_{2r} ^+ 1_{\wghtz} \gamma}{p_{2r} ^- (1_{\wghtz} -\gamma)^2} $, and $c_7:= \cones _1 c_3  \max \{\consman _1 ^{p_{2r}^{-}} , \consman _1 ^{p_{2r}^{+}}\} + \cones _2 + ( \consman _1 p_{2r}^{+})^{p_{2r}^{-}} +1$. This completes the proof.\\

\ref{24} Performing the procedure of the proof of \ref{23}, we can easily obtain the result.
\end{proof}


\subsection{Lower bound of supersolutions}
In this subsection, we obtain a lower bound for a supersolution of \eqref{13} (Proposition \ref{84}). We also prove a reverse weak Hölder inequality (Lemma \ref{82}).
\begin{proposition} \label{84}
Let $\{\ue\} \subset \lwo  $ be a family  of functions, and let  $\uu \in \lwo $ be an element. Assume that \eqref{78} - \eqref{26} and \eqref{28} (of property \ref{94}) are verified. Let $r \leq \sigma<\rho \leq 2r \leq 2$,   $B(\pvr _0 , 2r) \subset \subset \psi (\tilde U)$, $ 1<\gamma<1_{\wghtz}$, and $\gamma \geq s_0 >0$.   Suppose that $\uu$ is a  supersolution of \eqref{13} with $\ue \geq N$ in $B(\pvr _0 , 2r)$. Then, we have 
\begin{equation*}
\begin{aligned}
& \left[\fint _{B(\pvr _0 , \rho )} (\ue -N+r)^{-s_0} \, \dwvi  \right] ^{\frac{1}{-s_0 }} \\
&  \quad \ \leq (\cones _3 1_{\wghtz}/ \gamma)^{\frac{\cones _{10}}{s_0}}        \cones _{14} ^{ \frac{ 1_{\wghtz} \gamma}{ s_0 (1_{\wghtz} - \gamma )} }     \left( \frac{\rho}{\rho-\sigma}\right)^{\frac{p_{2r} ^+ 1_{\wghtz} \gamma}{s_0 (1_{\wghtz} - \gamma)} }      \underset{  B(\pvr _0 , \sigma )}{\operatorname{ess} \inf } \ \left\{ \ue -N+r \right\} ,
\end{aligned}
\end{equation*}
where {\footnotesize $p_{2r } ^{-}:=\inf _{ B( \pvr _0, 2r )  } p$, $p_{2r } ^{+}:=\sup _{B( \pvr _0, 2r )  } p$,  $\cones _{3}= [2(\consman _0 +1) \conspoin _2 ^{-1} (\conspoin _{\wghtz} +1)(\conspoin _4 +1 )]^{\cpoin _2}$,  $ \cones _{10} = p^+_{2r}\left[\frac{ 5 1_{\wghtz} \gamma}{ 1_{\wghtz} - \gamma} + \frac{ 1_{\wghtz}  \gamma ^2}{(1_{\wghtz} -\gamma)^2}\right] $, 
\begin{equation*}
\footnotesize
\begin{gathered}
\cones _{11} :=  \frac{4 \consfii  p^{+} _{2r}}{ \consfi  (p^{-} _{2r} -1) p^{-} _{2r}}   \left[ \frac{2 \consfii  (p^{+}_{2r}-1) }{ \consfi (p^{-} _{2r} -1)} \right] ^{p^+_{2r}-1}  + \frac{ \vinf  4 p^+ _{2r}}{\consfi  (p^{-} _{2r} -1) p^- _{2r}} \\
  \cones _{12} := 4\vinf   \frac{ p^+_{2r}-1  }{ \consfi   p^+_{2r}}\left( \frac{p^+_{2r} }{p^{-} _{2r} -1}+1\right),\\
\cones _{13}:=  \left\{ \fint _{B(\pvr _0, 2r   )}  (\ue -N+r )^{(p-p_{2r   }^{-}  ) \gamma ^{\prime}} \, \dwvi   \right\}^{\frac{1}{\gamma ^{\prime}}},
\end{gathered}
\end{equation*}
and $\cones _{14}:= \cones _{11} \cones _{13} \max \{\consman _1 ^{p_{2r}^{-}} , \consman _1 ^{p_{2r}^{+}}\} + \cones _{12} + (\consman _1 p_{2r}^{+})^{p_{2r}^{-}} +1$.}
\end{proposition}
\begin{proof} Fix $r>0$. Let $\zeta :=(\ue-N+r)^q \eta^{p_{\tr }^{+}}$, with $q<0$, and let $\eta \in C_0^{\infty} (B ( \pvr _0, \tr  ) )$ such that $0 \leq \eta \leq 1$. Note that $\varphi \in W_0^{1, p(\pvr)} (B (\pvr _0, \tr  ) ;  \wghtvi )$. 

We have
$$
\begin{aligned}
& \int _{\Omega} \mathscr{A}(\cdot , \grad  \ue ) \multg  \deriv  \zeta  \, \dwvi  = \int _{\Omega} g( \mathscr{A}(\cdot , \grad  \uu ) ,\grad  \zeta ) \, \dwvi    \\
& \quad \ + \int _{\Omega} \left[ \mathscr{A}(\cdot , \grad  \ue ) - \mathscr{A}(\cdot , \grad  \uu )\right] \multg  \deriv  \zeta  \, \dwvi  .
\end{aligned}
$$


{\it Step 1.} Since $\uu$ is a supersolution and
$$
\deriv  \zeta =q(\ue-N+r)^{q-1} \eta^{p_{\tr }^{+}} \deriv  \ue+p_{\tr }^{+}(\ue-N+r)^q \eta^{p_{\tr }^{+}-1} \deriv  \eta
$$
we have
$$
 \int_{\Omega} \mathscr{A}(\cdot, \grad  \ue) \multg  \left[  q(\ue-N+r)^{q-1} \eta^{ p_{\tr }^{+} } \deriv  \ue+p_{\tr }^{+}(\ue-N+r)^q \eta^{p_{\tr }^{+}-1} \deriv  \eta \right] \, \dwvi  \geq E ( \ell , N ,q),
$$
where
$$
\begin{aligned}
&E ( \ell , N ,q ) := \int _{\Omega} \left[ \mathscr{A}(\cdot , \grad  \ue ) - \mathscr{A}(\cdot , \grad  \uu ) \right] \multg   \left[ q(\ue-N+r)^{q-1} \eta^{p_{\tr }^{+}} \deriv  \ue+p_{\tr }^{+}(\ue-N+r)^q \eta^{p_{\tr }^{+}-1} \deriv  \eta \right] \, \dwvi .
\end{aligned}
$$

By \ref{16} and \ref{21}, we have
$$
\begin{aligned}
& \consfi  |q | \int_{\Omega}|\grad  \ue| _{\fnf}^{p}(\ue-N+r)^{q-1} \eta^{p_{\tr }^{+}} \, \dwvi  \\
& \quad \  \leq \consfii  p^{+} _{\tr}\int_{\Omega}  |\grad  \ue| _{\fnf} ^{p-1}(\ue-N+r)^q \eta^{p_{\tr }^{+}-1}|\deriv  \eta| _{\fnfs } \, \dwvi  -  E ( \ell , N , q)  \\
& \quad \ \leq \consfii  p^{+} _{\tr} \int_{\Omega} (\ue-N+r)^{q-1} \eta ^{p^+_{\tr}} \left[\frac{p^{+} _{\tr} -1}{p^{+} _{\tr}}\epsilon |\grad  \ue| _{\fnf}^{p} \right.\\
& \quad \quad \ \left. + \frac{1}{p^{-} _{\tr}} \max \{ \epsilon ^{-p^-_{\tr}+1} , \epsilon ^{-p^+_{\tr}+1} \}(\ue-N+r)^{p} \eta^{-p}|\deriv  \eta| _{\fnfs }^{p} \right]\, \dwvi  - E ( \ell , N , q) .
\end{aligned}
$$

Taking $\epsilon=\frac{\consfi  |q |}{2 \consfii  (p^{+}_{\tr}-1)}$, we have
\begin{equation} \label{49}
\begin{aligned}
& \frac{\consfi  |q |}{2} \int_{\Omega}|\grad  \ue|_{\fnf} ^{p}(\ue-N+r)^{q-1} \eta^{p_{\tr }^{+}} \, \dwvi  \\
& \quad  \ \leq \frac{\consfii  p^{+} _{\tr}}{p^{-} _{\tr}} \max \left\{ [2 \consfii  (p^{+}_{\tr}-1)/ \consfi  |q |] ^{p^-_{\tr}-1} , [2 \consfii  (p^{+}_{\tr}-1)/ \consfi  |q |] ^{p^+_{\tr}-1} \right\}\\
&  \quad \quad \ \cdot \int_{\Omega}(\ue-N+r)^{q-1+p} \eta^{p_{\tr }^{+}-p}|\deriv  \eta| _{\fnfs } ^{p} \, \dwvi  - E(\ell , N ,q) .
\end{aligned}
\end{equation}

$ $

{\it Steep 2.} We estimate the term $E(\ell , N , q)$:
\begin{align}
& |E(\ell , N , q)| \nonumber\\
& \quad \ \leq     \left|\int _{\Omega} \left[ \mathscr{A}(\cdot , \grad  \ue ) - \mathscr{A}(\cdot , \grad  \uu ) \right]\multg  \left[ q(\ue-N+r)^{q-1} \eta^{p_{\tr }^{+}} \deriv  \ue+p_{\tr }^{+}(\ue-N+r)^q \eta^{p_{\tr }^{+}-1} \deriv  \eta \right] \, \dwvi  \right|   \nonumber \\
& \quad \ \leq   |q| \int _{\Omega} \left| \mathscr{A} ( \cdot , \grad  \ue ) - \mathscr{A} ( \cdot , \grad  \uu) \right|_{\fnf}  (\ue-N+r)^{q-1} \eta ^{p^+ _{\tr }} | \deriv  \ue |_{\fnfs } \nonumber\\
& \quad \quad \  + p^+ _{\tr } \left| \mathscr{A} ( \cdot , \grad  \ue ) - \mathscr{A} ( \cdot , \grad  \uu) \right|  _{\fnf}  (\ue - N + r)^q \eta ^{p^+ _{\tr } - 1}  | \deriv  \eta |_{\fnfs } \, \dwvi  .  \label{51}
\end{align}

From \eqref{28}, 
\begin{equation*}
\left| \mathscr{A} ( \cdot , \grad  \ue ) - \mathscr{A} ( \cdot , \grad  \uu) \right|_{\fnf} \leq \vinf < \frac{\consfi    p_{\tr} ^-}{4} \quad \text { in } \quad \Omega.
\end{equation*}

We also have, in $B(\pvr _0 , \tr)$,
\begin{equation*} 
|\deriv \ue |_{\fnfs}=|\grad  \ue |_{\fnf} \leq \frac{1}{p^- _{\tr}} |\grad  \ue| _{\fnf}^p + \frac{p^+ _{\tr} -1}{p^+ _{\tr}}
\end{equation*}
and 
\begin{align*}
& (\ue - N +r )^{q} \eta ^{p^+_{\tr} -1}  |\deriv  \eta|_{\fnfs }\\
& \quad \ \leq (\ue - N +r )^{q-1} \left[ \frac{1}{p}  (\ue - N + r)^p |\deriv  \eta|_{\fnfs }^p + \frac{p-1}{p} \eta ^{\frac{p^+ _{\tr} -1}{p-1} p} \right]\\
& \quad \ \leq \frac{1}{p^- _{\tr}}  (\ue - N +r)^{q-1+p} |\deriv  \eta| _{\fnf} ^p + \frac{p^+_{\tr}-1}{p^+ _{\tr}} (\ue - N +r)^{q-1}  \eta ^{p^+ _{\tr}}.
\end{align*}

Consequently, from \eqref{51},
\begin{equation} \label{50}
\begin{aligned}
&  |E(\ell , N , q)| \\
& \quad \ \leq  \frac{\consfi |q|}{4} \int _{\Omega } (\ue - N +r )^{q-1} |\grad  \ue|_{\fnf} ^p \eta ^{p^+_{\tr}}\, \dwvi  + \frac{\vinf   |q| (p^+ _{\tr} -1)}{p^+ _{\tr}} \int _{\Omega} (\ue - N +r )^{q-1} \eta ^{p^+_{\tr}} \, \dwvi  \\
& \quad \quad \ + \frac{\vinf  p^+ _{\tr}}{p^- _{\tr}} \int _{\Omega }   (\ue - N +r )^{q-1+p} |\deriv  \eta| _{\fnfs }^p \, \dwvi  + \vinf  (p^+_{\tr}-1 ) \int _{\Omega } (\ue - N +r )^{q-1}  \eta ^{p^+ _{\tr}} \, \dwvi  .
\end{aligned}
\end{equation}

$ $

{\it Steep 3.} From \eqref{49} and \eqref{50},
\begin{equation} \label{52}
\begin{aligned}
& \frac{\consfi  |q |}{4} \int_{\Omega}|\grad  \ue| _{\fnf} ^{p}(\ue-N+r)^{q-1} \eta^{p_{\tr }^{+}} \, \dwvi  \\
&  \quad  \ \leq \cones _0 \int_{\Omega}(\ue-N+r)^{q-1+p} |\deriv  \eta| _{\fnfs }^{p} \, \dwvi  + \vinf   \frac{(p^+_{\tr}-1 )(p^+_{\tr} +|q|)}{p^+_{\tr}} \int _{\Omega } (\ue - N +r )^{q-1}  \eta ^{p^+ _{\tr}} \, \dwvi  ,
\end{aligned}
\end{equation}
where
$$
\cones _0:= \frac{\consfii  p^{+} _{\tr}}{p^{-} _{\tr}} \max \left\{ [2 \consfii  (p^{+}_{\tr}-1)/ \consfi  |q_0 |] ^{p^-_{\tr}-1} , [2 \consfii  (p^{+}_{\tr}-1)/ \consfi  |q_0 |] ^{p^+_{\tr}-1} \right\} + \frac{\vinf  p^+ _{\tr}}{p^- _{\tr}} 
$$
and $0>q_0\geq q$.

Taking $q_0=1-p^{-} _{\tr}$ and $q=\beta-p_{\tr }^{-}+1$ with $\beta <0$ in \eqref{52}, we have
\begin{equation} \label{59} 
\begin{aligned}
&  \int_{\Omega}|\grad  \ue|_{\fnf} ^{p}(\ue-N+r)^{\beta-p_{\tr }^{-}} \eta^{p_{\tr }^{+}} \, \dwvi  \\
&  \quad  \ \leq \bfc _1 \int_{\Omega}(\ue-N+r)^{\beta-p_{\tr }^{-} +p} |\deriv  \eta|_{\fnfs }^{p} \, \dwvi  + \bfc _2 \int _{\Omega } (\ue - N +r )^{\beta-p_{\tr }^{-}}  \eta ^{p^+ _{\tr}} \, \dwvi  ,
\end{aligned}
\end{equation}
where 
\begin{gather*}
\bfc _1 :=  \frac{4 \consfii  p^{+} _{\tr}}{ \consfi  (p^{-} _{\tr} -1) p^{-} _{\tr}} \max \left\{ \left[\frac{2 \consfii  (p^{+}_{\tr}-1)}{ \consfi  (p^{-} _{\tr} -1)} \right] ^{p^-_{\tr}-1} , \left[ \frac{2 \consfii  (p^{+}_{\tr}-1)}{ \consfi  (p^{-} _{\tr} -1) } \right] ^{p^+_{\tr}-1} \right\} + \frac{ \vinf  4 p^+ _{\tr}}{\consfi  (p^{-} _{\tr} -1) p^- _{\tr}} \\
  \bfc _2 := 4 \vinf   \frac{  p^+_{\tr}-1  }{ \consfi   p^+_{\tr}} \left[ \frac{p^{+} _{\tr}}{p^{-} _{\tr} -1} +1\right].
\end{gather*}

$ $

{\it Steep 4.} Let $r \leq \sigma<\rho \leq 2 r$. Next,  take $\tr = 2r$ in \eqref{59} and  consider the function $\eta \in C^{\infty} _0 (B(\pvr _0 , \rho))$ from Lemma \ref{3}, where  $\eta$ satisfies:   $0 \leq \eta \leq 1$, $ \eta=1$ in  $B ( \pvr _0, \sigma )$, and $|\deriv  \eta| _{\fnfs } \leq \consman _0 \consman _1/(\rho-\sigma)$.

Applying  inequality \eqref{27} to the function $(\ue -N +r )^{\beta/p_{2r  } ^-} \eta^{p_{2r  }^{+} / p_{2r  } ^-}$, we have
\begin{align*}
&\left\{ \fint _{B(\pvr _0, 2r  ) } \left[ (\ue -N + r  )^{\beta / p_{2r   } ^-} \eta ^{p_{2r   }^{+} / p_{2r   }^{-}}    \right]       ^{ p^- _{2r } 1_{\wghtz}} \, \dwvi    \right\} ^{\frac{1}{ 1_{\wghtz}} }\\
&\quad \ \leq (2 p^-_{2r } \conspoin  _{\wghtz}  r )^{p^- _{2r }}\fint_{B(\pvr _0, 2r    ) } \left| \grad  \left[(\ue -N+r  )^{\beta / p_{2r   }^{-}} \eta^{p_{2r   }^{+} / p_{2r   } ^-}  \right]  \right|_{\fnf} ^{p_{2r   }^{-}} \, \dwvi  \\
& \quad  \ \leq (2\conspoin  _{\wghtz}  r)^{p^- _{2r }} 2^{p^- _{2r } -1 } (1+|\beta|)^{p_{2r  }^{-}} \left\{ \fint _{B(\pvr _0, 2r    ) } \left( \ue -N+r  \right)^{\beta-p_{2r   }^{-}}|\grad  \ue   | _{\fnf} ^{p_{2r   }^{-}} \eta^{p_{2r   }^{+}} \, \dwvi   \right.\\
&\quad \quad \ \left. + (p_{2r   }^{+})^{p_{2r   }^{-}}\fint _{B(\pvr _0, 2r    ) }  \left( \ue -N +r   \right)^\beta \eta^{p_{2r   }^{+}-p_{2r   }^{-}}|\deriv  \eta|_{\fnfs }^{p_{2r   }^{-}} \, \dwvi   \right\} .
\end{align*}

By \eqref{59} and \eqref{34}, we can obtain
\begin{equation} \label{54}
\begin{aligned}
& \left\{ \conspoin _2 2^{-\cpoin _2}\fint _{B(\pvr  _0, \sigma  )}\left(\ue -N +r  \right)^{\beta 1_{\wghtz}}   \, \dwvi   \right\} ^{\frac{1}{1_{\wghtz}}} \\
& \quad \ \leq (2\conspoin  _{\wghtz}      r )^{p^- _{2r }} 2^{p^- _{2r } -1 } (1+|\beta |)^{p_{2r  }^{-}} \left\{ \bfc _1   \fint_{B(\pvr _0, 2r   )}  (\ue -N + r)  ^{p + \beta-p_{2r   }^{-}} |\deriv  \eta  |_{\fnfs }^{p} \, \dwvi  \right. \\
&\quad \quad \  + \bfc _2  \fint_{B(\pvr _0, \rho  )}   (\ue -N +r )  ^{\beta-p_{2r   }^{-}} \eta  ^{p_{2r  } ^{+}} \, \dwvi  \\
&\quad \quad \ \left. + (p_{2r   }^{+})^{p_{2r   }^{-}}\fint _{B(\pvr _0, \rho   ) }  (\ue -N +r )^\beta |\deriv  \eta|_{\fnfs }^{p_{2r   }^{-}} \, \dwvi   \right\}.
\end{aligned}
\end{equation}

$ $

{\it Steep 2.} Next, we estimate the right-hand side of \eqref{54}.

By Hölder's inequality, for $\gamma \in (1,1_{\wghtz})$,
\begin{equation} \label{55}
\fint _{B(\pvr _0, \rho  )}  (\ue -N+r ) ^\beta|\deriv  \eta|_{\fnfs }^{p_{2r   }^{-}} \, \dwvi  \leq   r^{-p_{2r   }^{-}}\left(\frac{ \consman _0\consman _1 \rho}{\rho-\sigma}  \right)^{p_{2r   }^{-}}\left\{  \fint _{B(\pvr _0, \rho  )} (\ue -N +r  )^{\beta \gamma} \, \dwvi \right\} ^{\frac{1}{\gamma}}
\end{equation}
and
\begin{equation} \label{56}
\fint _{B(\pvr _0, \rho  )}  (\ue -N+r  )^{\beta-p_{2r   }^{-}} \, \dwvi  \leq  r^{-p_{2r   }^{-}} \left\{ \fint _{B(\pvr _0, \rho  )} (\ue -N +r  )^{\beta \gamma} \, \dwvi  \right\}^{\frac{1}{\gamma}},
\end{equation}
since $r\leq \ue -N+r$.

By \eqref{38}, we have 
\begin{equation*} \label{39}
|\deriv  \eta|_{\fnfs }^{p} \leq r^{-p}\left( \frac{\consman _0 \consman _1 r}{\rho - \sigma}\right)^{p} \leq r^{-p}\left( \frac{\consman _0 \consman _1 \rho}{\rho - \sigma}\right)^{p} \leq \conspoin _4 r^{-p^- _{2r }}\left( \frac{\consman _0 \consman _1 \rho}{\rho - \sigma}\right)^{p}\leq \bfc _3 r^{-p^- _{2r }}\left( \frac{ \rho}{\rho - \sigma}\right)^{p^+_{2r }},
\end{equation*}
where $\bfc _3:= \conspoin _4 \max \{(\consman _0 \consman _1) ^{p^-_{2r }} , (\consman _0 \consman _1) ^{p^+_{2r }}\}$.

Then,
\begin{align}
& \fint _{B(\pvr _0, 2r   )} (\ue -N+r  )^{\beta-p_{2r   }^{-}+p}|\deriv  \eta|_{\fnfs }^{p} \, \dwvi  \nonumber \\
& \quad \ \leq \frac{\bfc _3 r^{-p_{2r   }^{-}}}{\wghtvi  (B(\pvr _0, 2r   ))}  \left(\frac{\rho}{\rho-\sigma}  \right)^{p_{2r   }^{+}}   \int _{B(\pvr _0, \rho  )} (\ue -N+r  )^{\beta-p_{2r   }^{-}+p}  \, \dwvi   \nonumber  \\
& \quad \ \leq \bfc _3 r^{-p_{2r   }^{-}}\left(\frac{\rho}{\rho-\sigma}  \right)^{p_{2r   }^{+}}  \left\{ \fint _{B(\pvr _0, 2r   )} (\ue -N+r  )^{(p-p_{2r   }^{-}  ) \gamma ^{\prime}} \, \dwvi   \right\}^{\frac{1}{\gamma ^{\prime}}} \nonumber \\
& \quad \quad \ \cdot \left\{ \fint _{B(\pvr _0, \rho  )} (\ue -N+r  ) ^{\beta \gamma} \, \dwvi  \right\} ^{\frac{1}{\gamma}} \nonumber\\
& \quad \ \leq \bfc _3 \bfc _4 r^{-p_{2r   }^{-}}\left(\frac{\rho}{\rho-\sigma}  \right)^{p_{2r   }^{+}}   \left\{ \fint _{B(\pvr _0, \rho  )}  (\ue - N +r )^{\beta \gamma} \, \dwvi  \right\}^{\frac{1}{\gamma}}, \label{57}
\end{align}
where 
$$
\bfc _4:=  \left\{ \fint _{B(\pvr _0, 2r   )}  (\ue -N+r )^{(p-p_{2r   }^{-}  ) \gamma ^{\prime}} \, \dwvi   \right\}^{\frac{1}{\gamma ^{\prime}}}.
$$

From \eqref{54} - \eqref{57}, we can get
\begin{equation*} 
\begin{aligned}
& \left[ \conspoin _2 2^{-\cpoin _2}\fint _{B(\pvr  _0, \sigma  )} (\ue -N +r)  ^{\beta 1_{\wghtz}}   \, \dwvi   \right] ^{\frac{1}{1_{\wghtz}}} \\
& \quad \ \leq (2 \conspoin  _{\wghtz}       )^{p^- _{2r }} 2^{p^- _{2r } -1 } \left[ \bfc _1   \bfc _3 \bfc _4   + \bfc _2   + (\consman _0 \consman _1 p_{2r   }^{+})^{p_{2r   }^{-}} \right]\\
& \quad \quad \ \cdot (1+|\beta|)^{p_{2r  }^{-}}  \left( \frac{\rho}{\rho-\sigma}  \right)^{p_{2r   }^{+}}   \left[ \fint _{B(\pvr _0, \rho  )}  (\ue - N +r  )^{\beta \gamma} \, \dwvi  \right]^{\frac{1}{\gamma}}.
\end{aligned}
\end{equation*}

 
Then
\begin{equation} \label{58}
\begin{aligned}
&\left[ \fint _{B(\pvr _0, \rho  )}  (\ue - N +r  )^{\beta \gamma} \, \dwvi  \right]^{\frac{1}{\beta \gamma}} \\
& \quad \ \leq \bfc _5  ^{\frac{1}{|\beta|}} (1+|\beta|)^{\frac{p_{2r  }^{-} }{|\beta|}}  \left( \frac{\rho}{\rho-\sigma}  \right)^{\frac{p_{2r  }^{-} }{|\beta|}}    \left[ \fint _{B(\pvr  _0, \sigma  )} (\ue -N +r)  ^{\beta 1_{\wghtz}}   \, \dwvi   \right] ^{\frac{1}{\beta 1_{ \wghtz}}},
\end{aligned}
\end{equation}
where
$$
\bfc _5 := (\conspoin _2 ^{-1} 2^{\cpoin _2} )^{1/1_{\wghtz}} (2 \conspoin  _{\wghtz}       )^{p^- _{2r }} 2^{p^- _{2r } -1 } \left[ \bfc _1   \bfc _3 \bfc _4   + \bfc _2   + (\consman _0 \consman _1 p_{2r   }^{+})^{p_{2r   }^{-}} \right].
$$

Substituting $r_j=\sigma+ 2^{-j}(\rho-\sigma)$,  $ \xi_j=- (1_{\wghtz } / \gamma )^j s_0$, and $\beta = -(1_{\wghtz} / \gamma)^j (s_0 / \gamma)$ in \eqref{58}, we obtain the iterative inequality
$$
\begin{aligned}
& \Psi (\ue-N+r, \xi_j, B (\pvr _0, r_j ) ) \\
& \quad  \ \leq \bfc _5 ^{\frac{\gamma}{ |\xi_j |}} \left(1+\frac{ |\xi_j |}{\gamma} \right)^{\frac{\gamma p_{2r }^{+}}{ |\xi_j |}} \left(\frac{r_j}{r_j-r_{j+1}} \right)^{\frac{\gamma p_{2r }^{+}}{ |\xi_j |}} \Psi (\ue-N+r, \xi_{j+1}, B (\pvr _0, r_{j+1} ) ),
\end{aligned}
$$
where $\Psi (f,q,D)=( \fint _D f^q \, \dwvi  )^{1/q}$.

Hence,
\begin{equation*}
\begin{aligned}
& \Psi(\ue -N+r, -s_0, B(\pvr _0 , \rho ))   \\
&  \quad \ \leq \prod_{j=0}^{\infty}\left[\bfc_5 ^{\gamma  /|\xi_j | }\left(1+\frac{ |\xi_j| }{\gamma}\right)^{\gamma p_{2r}^{+} / | \xi_j |}\left( \frac{2^{j+1}\rho}{\rho-\sigma}\right)^{\gamma p_{2r}^{+} / |\xi_j |}\right] \underset{  B(\pvr _0 , \sigma )}{\operatorname{ess} \inf } \ \left\{ \ue -N+r \right\} \\
&  \quad \ \leq \bfc _5 ^{\sum _{i=0}^{\infty}\gamma  /|\xi_j|}   2^{\sum _{j=0}^{\infty} (j+1)\gamma p_{2r}^{+} / | \xi_j |}  \left( \frac{\rho}{\rho-\sigma}\right)^{\sum _{i=0}^{\infty} \gamma p_{2r}^{+} / | \xi_j | } \prod_{j=0}^{\infty }   \left(1+\frac{| \xi_j |}{\gamma}\right)^{\gamma p_{2r}^{+} / | \xi_j |}   \\
 & \quad \quad \ \cdot  \underset{  B(\pvr _0 , \sigma )}{\operatorname{ess} \inf } \ \left\{ \ue -N+r \right\}.
\end{aligned}
\end{equation*}

We have
\begin{equation*}
\sum _{j=0}^{\infty} \frac{\gamma}{ |\xi_j|} = \frac{ 1_{\wghtz} \gamma }{ s_0 (1_{\wghtz} - \gamma)}
\quad \text { and } \quad
\sum _{j=0}^{\infty} j\frac{\gamma}{ |\xi_j|} = \frac{1_{\wghtz} \gamma ^2}{ s_0 (1_{\wghtz} - \gamma)^2}.
\end{equation*}

Let $j_0 \in \mathbb{N}$ such that 
$$
\frac{|\xi _j|}{\gamma}\geq 1 \text { if } j\geq j_0 +1 \quad \text { and } \quad  \frac{|\xi _j|}{\gamma} < 1  \text { if } j\leq j_0 .
$$

Then,
\begin{equation*}
\begin{aligned}
& \prod_{j=0}^{\infty}\left(1+\frac{|\xi_j|}{\gamma}\right)^{\gamma p_{2r}^{+} / |\xi_j|} \leq  \prod_{j=0}^{j_0} 2^{ \gamma p_{2r}^{+} / |\xi_j|}     \prod_{j=j_0 +1}^{\infty} \left( 2 \frac{|\xi_j|}{\gamma}\right)^{ \gamma p_{2r}^{+} / |\xi_j|} \\
& \quad \ \leq  \prod_{j=0}^{j_0} 2^{ \gamma p_{2r}^{+} / |\xi_j|}     \prod_{j=j_0 +1}^{\infty} \left( 2 \frac{|\xi_j|}{\gamma}\right)^{ \gamma p_{2r}^{+} / |\xi_j|}\\ 
& \quad \  \leq 2^{\sum _{j=0}^{\infty}\gamma p_{2r}^{+} / |\xi_j|} \prod_{j=j_0 +1}^{\infty} \left[\left( \frac{1_{\wghtz}}{\gamma} \right) ^j \frac{s_0}{\gamma}\right] ^{  \gamma p_{2r}^{+} / |\xi_j|}\\
& \quad \ \leq 2^{\sum _{j=0}^{\infty}\gamma p_{2r}^{+} / |\xi_j|}\left( \frac{1_{\wghtz}}{\gamma}\right)^{\sum_{j=0}^{\infty} j\gamma p_{2r}^{+} / |\xi_j|},
\end{aligned}
\end{equation*}
since $\gamma \geq s_0$.

This implies,
\begin{equation*}
\begin{aligned}
& \Psi(\ue -N+r, -s_0 , B(\pvr _0 , \rho )) \\
&  \quad \ \leq (\bfc _6 1_{\wghtz}/ \gamma)^{\frac{\bfc _7}{s_0}}        \bfc _{8} ^{ \frac{ 1_{\wghtz} \gamma}{ s_0 (1_{\wghtz} - \gamma )} }    \left( \frac{\rho}{\rho-\sigma}\right)^{\frac{p_{2r} ^+ 1_{\wghtz} \gamma}{s_0 (1_{\wghtz} - \gamma)} }   \\
 & \quad \quad \ \cdot   \underset{  B(\pvr _0 , \sigma )}{\operatorname{ess} \inf } \ \left\{ \ue -N+r \right\} ,
\end{aligned}
\end{equation*}
where $\bfc _6:= [2(\consman _0 +1) \conspoin _2 ^{-1} (\conspoin _{\wghtz} +1)(\conspoin _4 +1)]^{\cpoin _2} $,  $\bfc_7 = p^+_{2r}\left[\frac{ 5 1_{\wghtz} \gamma}{ 1_{\wghtz} - \gamma} + \frac{ 1_{\wghtz}  \gamma ^2}{(1_{\wghtz} -\gamma)^2}\right]$, and $\bfc _{8}:= \bfc _1 \bfc _4 \max \{\consman _1 ^{p_{2r}^{-}} , \consman _1 ^{p_{2r}^{+}}\} + \bfc _2 + (\consman _1 p_{2r}^{+})^{p_{2r}^{-}}  +1$. This completes the proof. \end{proof}

Before proving the weak reverse Hölder inequality (Lemma \ref{82}), we will first prove the following 
\begin{lemma}   Let $\{\ue\} \subset \lwo  $ be a family  of functions, and let  $\uu \in \lwo $ be an element. Assume that \eqref{28} (of property \ref{94}) holds.      Suppose that $\uu$ is a  supersolution of \eqref{13} and $\ue \geq 0$.  Let $W$ be a measurable subset of $B(\pvr _0 , \tr)\subset \psi (\tilde U)$, and let $\eta \in C_0^{\infty}(B(\pvr _0 , \tr))$ such that $0 \leq \eta \leq 1$. Let $\gamma <0$. Then, we have
\begin{equation}\label{63}
\begin{aligned}
& \int_{W}|\grad  \ue|_{\fnf}^{p^-_W} \eta^{p^+_{\tr}} \uu _{\ell ,\alpha}^{\gamma-1} \, \dwvi  \\
& \quad \ \leq c_1 \int_{\Omega} \uu _{\ell , \alpha}^{\gamma+p-1}|\deriv  \eta| _{\fnfs }^{p}  \, \dwvi   + c_2 \int _{\Omega}   \eta^{p^+_{\tr}} \uu _{\ell , \alpha}^{\gamma-1} \, \dwvi ,
\end{aligned}
\end{equation}
where   $\uu _{\ell , \alpha}:= \ue + \alpha$,  $\alpha \geq 0$, \footnotesize{ $p_{W } ^{-}:=\inf _{ W  } p$,   $p_{\tr } ^{-}:=\inf _{ B( \pvr _0, \tr )  } p$, $p_{\tr } ^{+}:=\sup _{B( \pvr _0, \tr )  } p$,
\begin{equation*}
c_1 :=  \frac{4\consfii  p^+_{\tr}}{ |\gamma|\consfi  p^-_{\tr}}  \max \left\{ \left[2\consfii  (p^+_{\tr}-1) / (|\gamma|\consfi ) \right] ^{p^-_{\tr}-1} , \left[2 \consfii (p^+_{\tr} -1)/ (|\gamma|\consfi ) \right] ^{p^+_{\tr} - 1} \right\} + \vinf \frac{4\consfii  p^+_{\tr}}{|\gamma|\consfi p^- _{\tr}} 
\end{equation*}
and
\begin{equation*}
c_2 := \vinf\frac{4(p^+_{\tr}-1)(p^+_{\tr} +|\gamma|)}{\consfi p^+_{\tr}|\gamma |} +1.
\end{equation*}}
\end{lemma}

\begin{proof}  We have,
$$
\begin{aligned}
& \int _{\Omega} \mathscr{A}(\cdot , \grad  \ue ) \multg  \deriv  \zeta  \, \dwvi  = \int _{\Omega}  \mathscr{A}(\cdot , \grad  \uu ) \multg  \deriv  \zeta  \, \dwvi    \\
& \quad \ + \int _{\Omega} \left[ \mathscr{A}(\cdot , \grad  \ue ) - \mathscr{A}(\cdot , \grad  \uu ) \right] \multg  \deriv  \zeta  \, \dwvi  .
\end{aligned}
$$

Let $\zeta :=\uu _{\ell , \alpha}^\gamma \eta^{p^+_{\tr}}$. Since $\uu$ is a supersolution and
$$
\deriv  \zeta =\gamma \uu _{\ell ,\alpha}^{\gamma-1} \eta^{p^+_{\tr}} \deriv  \uu+ p^+_{\tr} \uu _{\ell , \alpha}^\gamma  \eta^{p^+_{\tr}-1} \deriv  \eta,
$$
we have
$$
 \int_{\Omega}  \mathscr{A}(\cdot , \grad  \ue) \multg  \left[ \gamma \uu _{\ell ,\alpha}^{\gamma-1} \eta^{p^+_{\tr}} \deriv  \uu+ p^+_{\tr} \uu _{\ell , \alpha}^\gamma  \eta^{p^+_{\tr}-1} \deriv  \eta\right] \, \dwvi  \geq E ( \ell  ),
$$
where
$$
\begin{aligned}
&E ( \ell  ) := \int _{\Omega} \left[ \mathscr{A}(\cdot , \grad  \ue ) - \mathscr{A}(\cdot , \grad  \uu ) \right]\multg  \left[  \gamma \uu _{\ell ,\alpha}^{\gamma-1} \eta^{p^+_{\tr}} \deriv  \ue + p^+_{\tr} \uu _{\ell , \alpha}^\gamma  \eta^{p^+_{\tr}-1} \deriv  \eta \right] \, \dwvi .
\end{aligned}
$$

From \ref{16} and \ref{21}, and since $\gamma$ is a negative,
\begin{equation}\label{60}
|\gamma|\consfi   \int_{\Omega}|\grad  \ue |_{\fnf}^{p} \eta^{p^+_{\tr}} \uu _{\ell ,\alpha}^{\gamma-1} \, \dwvi  \leq p^+_{\tr} \consfii \int_{\Omega}  \uu _{\ell , \alpha}^\gamma \eta^{p^+_{\tr}-1} |\grad  \ue |_{\fnf}^{p-1} |\deriv  \eta |_{\fnfs }  \, \dwvi  - E(\ell).
\end{equation}

Next, we estimate  the  first term of the right-hand side of \eqref{60}. Using the Young's inequality,
\begin{equation} \label{61}
\begin{aligned}
& p^+_{\tr} \int_{\Omega} |\grad  \ue| _{\fnf}^{p-1} \uu _{\ell ,\alpha}^\gamma \eta^{p^+_{\tr}-1}|\deriv  \eta| _{\fnfs } \, \dwvi  \\
& \quad \ \leq  p^+_{\tr} \int_{\Omega} \frac{\epsilon^{-p+1}}{p} \left[  \uu _{\ell ,\alpha}^{(\gamma+p-1) / p}|\deriv  \eta| _{\fnfs } \eta^{p^+_{\tr} -1 -p^+_{\tr}(p-1) / p} \right]^{p} \\
& \quad \quad \ +\frac{\epsilon(p-1)}{p} \left[  |\grad  \ue| _{\fnf} ^{p-1} \eta^{p^+_{\tr} (p-1) / p} \uu _{\ell , \alpha}^{\gamma-(\gamma+p-1) / p} \right] ^{p/(p-1)} \, \dwvi  \\
&\quad \ \leq  \frac{p^+_{\tr}}{p^-_{\tr}}  \max \{ \epsilon ^{-p^-_{\tr}+1} , \epsilon ^{-p^+_{\tr} +1}\} \int_{\Omega} \uu _{\ell , \alpha}^{\gamma+p-1}|\deriv  \eta|_{\fnfs }^{p} \eta^{p^+_{\tr}-p} \, \dwvi  \\
& \quad \quad \ + (p^+_{\tr}-1) \epsilon \int_{\Omega}|\grad  \uu|_{\fnf}^{p} \eta^{p^+_{\tr}} \uu _{\ell , \alpha}^{\gamma-1} \, \dwvi .
\end{aligned}
\end{equation}

Choosing $\epsilon = |\gamma|\consfi /[2\consfii  (p^+_{\tr}-1)]$, from \eqref{60} and \eqref{61},
\begin{equation}\label{62} 
\begin{aligned}
&\frac{|\gamma|\consfi }{2}  \int_{\Omega}|\grad  \ue|_{\fnf} ^{p} \eta^{p^+_{\tr}} \uu _{\ell ,\alpha}^{\gamma-1} \, \dwvi  \\
& \quad \ \leq \consfii  \frac{p^+_{\tr}}{p^-_{\tr}}  \max \left\{ \left[2\consfii  (p^+_{\tr} -1)/ (|\gamma|\consfi ) \right] ^{p^-_{\tr}-1} , \left[2\consfii ( p^+_{\tr} -1) / (|\gamma|\consfi ) \right] ^{p^+_{\tr} - 1} \right\}\\
& \quad \quad \ \cdot \int_{\Omega} \uu _{\ell , \alpha}^{\gamma+p-1}|\deriv  \eta|_{\fnfs }^{p} \eta^{p^+_{\tr}-p} \, \dwvi   - E(\ell).
\end{aligned}
\end{equation}

Now, we estimate the term $E(\ell)$. From \eqref{28}, 
$$
|\mathscr{A} (\cdot , \grad  \ue ) - \mathscr{A} (\cdot , \grad  \uu )|_{\fnf}\leq \vinf < \frac{\consfi  p^- _{\tr}}{4} \quad \text { in } \quad \Omega.
$$

We also have
$$
|\deriv \ue |_{\fnfs}=|\grad  \ue| _{\fnf} \leq \frac{1}{p^-_{\tr}} |\grad  \ue|_{\fnf}^p +\frac{p^+ _{\tr} -1}{p^+_{\tr}} \quad \text { in } \quad B(\pvr _0 , \tr),
$$ 
and
$$
\begin{aligned}
&\uu _{\ell , \alpha}^\gamma  \eta^{p^+_{\tr}-1} |\deriv  \eta | _{\fnfs }\\
&\quad  \ \leq \frac{1}{p^- _{\tr}} \left[  \uu _{\ell ,\alpha}^{(\gamma+p-1) / p}|\deriv  \eta|_{\fnfs } \eta^{p^+_{\tr} -1 -p^+_{\tr}(p-1) / p} \right]^{p}  +\frac{(p^+_{\tr}-1)}{p^+_{\tr}} \left[   \eta^{p^+_{\tr} (p-1) / p} \uu _{\ell , \alpha}^{\gamma-(\gamma+p-1) / p} \right] ^{p/(p-1)}\\
&\quad \ = \frac{1}{p^- _{\tr}}   \uu _{\ell ,\alpha}^{\gamma+p-1}|\deriv  \eta|_{\fnfs } ^p \eta^{ p^+_{\tr}-p}   +\frac{(p^+_{\tr}-1)}{p^+_{\tr}}    \eta^{p^+_{\tr}} \uu _{\ell , \alpha}^{\gamma-1} .
\end{aligned}
$$

Therefore,
$$
\begin{aligned}
&|E ( \ell  ) |\leq \vinf \int _{\Omega} \left|  \gamma \uu _{\ell ,\alpha}^{\gamma-1} \eta^{p^+_{\tr}} \deriv  \uu+ p^+_{\tr} \uu _{\ell , \alpha}^\gamma  \eta^{p^+_{\tr}-1} \deriv  \eta \right| _{\fnfs }\, \dwvi \\
& \quad \ \leq   \frac{|\gamma|\consfi  }{4}  \int _{\Omega} |\grad  \ue | _{\fnf} ^p\eta^{p^+_{\tr}} \uu _{\ell , \alpha}^{\gamma-1} \, \dwvi  \\
&\quad \quad \ \ + \vinf \frac{p^+_{\tr}}{p^- _{\tr}}  \int _{\Omega} \uu _{\ell ,\alpha}^{\gamma+p-1}|\deriv  \eta| _{\fnfs }^p \eta^{ p^+_{\tr}-p}  \, \dwvi  +\vinf\frac{(p^+_{\tr}-1)(p^+_{\tr} + |\gamma|)}{p^+_{\tr}} \int _{\Omega}   \eta^{p^+_{\tr}} \uu _{\ell , \alpha}^{\gamma-1} \, \dwvi .
\end{aligned}
$$

By combining this with \eqref{62} we arrive at
\begin{equation*}
\begin{aligned}
& \int_{\Omega}|\grad  \ue|_{\fnf} ^{p} \eta^{p^+_{\tr}} \uu _{\ell ,\alpha}^{\gamma-1} \, \dwvi  \\
& \quad \ \leq c_1 \int_{\Omega} \uu _{\ell , \alpha}^{\gamma+p-1}|\deriv  \eta| _{\fnfs }^{p}  \, \dwvi   + c_2 \int _{\Omega}   \eta^{p^+_{\tr}} \uu _{\ell , \alpha}^{\gamma-1} \, \dwvi ,
\end{aligned}
\end{equation*}
where
$$
c_1 :=  \frac{4\consfii  p^+_{\tr}}{ |\gamma|\consfi  p^-_{\tr}}  \max \left\{ \left[2\consfii  (p^+_{\tr} -1)/ (|\gamma|\consfi ) \right] ^{p^-_{\tr}-1} , \left[2 \consfii  (p^+_{\tr} -1)/ (|\gamma|\consfi ) \right] ^{p^+_{\tr} - 1} \right\} + \vinf \frac{4\consfii  p^+_{\tr}}{|\gamma|\consfi p^- _{\tr}} 
$$
and
$$
c_2 := \vinf\frac{4(p^+_{\tr}-1)(p^+_{\tr} +|\gamma|)}{\consfi p^+_{\tr}|\gamma |} .
$$

Using \( |a|^{p_W^-} \leq |a|^{p(\pvr)} + 1 \) for all \( \pvr \in W \) and \( a \in \mathbb{R} \), we obtain the desired estimate in \eqref{63}.\end{proof}

\begin{lemma} \label{82} Let $\{\ue\} \subset \lwo   $ be a family  of functions, and let  $\uu \in \lwo \cap L^\infty _{\loc}(\Omega) $ be an element. Assume that \eqref{78}, \eqref{74}, \eqref{64}, \eqref{70}, and \eqref{28} (of property \ref{94}) are satisfied. Let  $B(\pvr _0 , 20r )\subset \subset \psi (\tilde U)$, $0<20r<\rast$, and $ 1\geq \alpha \geq r$. Suppose that $\uu$ is a  supersolution of \eqref{13}, with $ \ue \geq N \geq 0$ and $N_0 \geq N$. Then, we have
$$
\left[ \fint _{B(\pvr _0 , 2r)} (\ue -N + \alpha)^{c_2 c_3 }\, \dwvi  \right]^{\frac{1}{c_2 c_3}} \leq c_1 ^{1/c_3} \left[ \fint _{B(\pvr _0 , 2r)} (\ue -N + \alpha)^{-c_2 c_3 }\, \dwvi  \right]^{\frac{1}{-c_2 c_3 }},
$$
where {\footnotesize $c _i:= c_i( \consfi  , \consfii  , p^+ , \tilde \conspoin _1 , \conspoin _1 , \conspoin _2 , \conspoin _3 , \conspoin _4 , \conspoin _{\wghtzi} , \conspoin _0 , N _0 , \vinf _0 ,  \cpoin _2 , n , s )>0$ for $i=1,2$, $p^-_{20r} := \inf _{B(\pvr _0 , 20r)} p$, and 
$$
c_3 := \left\{ \left[(p^- _{20r} -1)^{-1} +1\right] ( \| u  \|_{s , B(\pvr _0 , 20r) , \wghtvi } +1)^{\conspoin _3} (\consman _1+1)^{p^+}\right\} ^{-1}. 
$$}

\end{lemma}

\begin{proof} Choose a ball $B(\pvr _1 , 2\tr) \subset B(\pvr _0 , 20r)$ and a cutoff function $\eta \in C_0^{\infty}(B(\pvr _1 , 2\tr))$ given in Lemma \ref{3}: $\eta=1$ in $B(\pvr _1 , \tr)$ and $|\deriv  \eta|_{\fnfs } \leq \consman _0 \consman _1  \tr ^{-1}$. 

Let $p_{1,2 \tr}^{-} = \inf _{B(\pvr _1 , 2 \tr)} p$. Taking $W=B(\pvr _1 , 2 \tr)$ and $\gamma=1-p_{1,2\tr}^{-}$ in \eqref{63}, we have
\begin{align*}
&\fint _{B(\pvr _1 , \tr)}|\grad  \log w _{\ell , \alpha}|_{\fnf} ^{p_{1, 2 \tr}^-} \, \dwvi  \\
& \quad \ \leq (c_1 + c_2)\frac{\wghtvi(B(\pvr _1 , 2 \tr))}{\wghtvi(B(\pvr _1 ,  \tr))}\left(\fint _{B(\pvr _1 , 2 \tr)} w _{\ell , \alpha}^{-p_{1, 2 \tr}^-} \, \dwvi  +\fint _{B(\pvr _1 , 2\tr)} (\consman _0 \consman _1)^p w _{\ell , \alpha}^{p-p_{1,2 \tr} ^-} \tr ^{-p} \, \dwvi \right) ,
\end{align*}
where $w_{\ell , \alpha} := \uu _{\ell , \alpha} -N = \ue - N +\alpha$ and $r\leq \alpha \leq 1$.

 From \eqref{70}, \eqref{66}, \eqref{28}, and \eqref{34},
$$
c_1 + c_2 \leq (p_{1,2\tr}^{-} -1)^{-1} C_1 (\consfi  , \consfii  , p^+ , \tilde \conspoin _1 , \conspoin _2 , \conspoin _3 , \cpoin _2).
$$

Using \eqref{38} and the estimate $w _{\ell , \alpha}^{-p_{1, 2 \tr} ^-} \leq \alpha^{-p_{1 , 2 \tr}^-} \leq r^{-p_{1 , 2 \tr}^-} \leq (\tr / 10)^{-p_{ 1 , 2 \tr}^-}$, we get
$$
 \fint _{B(\pvr _1 , \tr)} |\grad  \log w _{\ell , \alpha} |_{\fnf} ^{p_{1,2 \tr} ^-} \, \dwvi  \leq (p_{1,2\tr}^{-} -1)^{-1} C_1 \left[ \left( \frac{\tr}{10} \right)^{-p_{1,2 \tr}^-}+ C_2 \tr ^{-p_{1, 2\tr} ^-} \fint_{B(\pvr _1 , 2\tr)} w _{\ell , \alpha}^{p-p_{1,2 \tr} ^-} \, \dwvi  \right] ,
$$
where $C_2 := \conspoin _4 \max \{(\consman _0 \consman _1)^{p^- _{1,2\tr}} , (\consman _0 \consman _1)^{p^+ _{1 , 2\tr}}\}$.

Let $f =\log w _{\ell , \alpha}$. By  \eqref{64},  \eqref{75}, \eqref{28}, and the above estimate, we obtain
\begin{equation} \label{72}
\begin{aligned}
& \fint _{B(\pvr _1, \tr )}\left|f -f _{B(\pvr _1, \tr )}\right| \, \dwvi  \leq \conspoin _{\wghtzi}\left( \tr ^{p_{1, 2 \tr} ^-} \fint _{ B(\pvr _1, \tr )}|\grad  f | _{\fnf} ^{p_{1 , 2 \tr}^-}\, \dwvi \right)^{\frac{1 } {p_{1 , 2 \tr}^-}}\\
& \quad \ \leq  (p_{1,2 \tr}^{-} -1)^{-\frac{1 } {p_{1 , 2 \tr}^-} } C_3 C_4 \left[ 1+  \fint_{B(\pvr _1 , 2 \tr)} (\uu -N + \vinf _0 +1 )^{p-p_{1,2\tr} ^-} \, \dwvi  \right] ^{\frac{1 } {p_{1,2\tr}^-}} \\
& \quad \ \leq \left[ (p_{20r}^{-} -1)^{-1} + 1 \right] C_3 C_4 \\
& \quad \quad \ \cdot \left\{ 1+ 2^{p_{ 20r} ^+ -p_{ 20r} ^-}  \left[ (\conspoin _0 ^{1/s} +1)\left( \| u  \|_{s , B(\pvr _0 , 20r) , \wghtvi } +1 \right)^{p_{ 20r} ^+ -p_{ 20r} ^-}  + (N+ \vinf _0 + 1)^{p_{ 20r} ^+ -p_{ 20r} ^-} \right]\right\},
\end{aligned}
\end{equation}
if $B ( \pvr _1, \tr ) \subset B( \pvr _0 , 10 r)$, where  $C_3 := C_3 ( \consfi  , \consfii  , p^+ , \tilde \conspoin _1 , \conspoin _2 ,  \conspoin _3 , \conspoin _4 , \conspoin _{\wghtzi} , N _0 ,\cpoin _2 , n) >0$ and $C_4 :=(\consman _1 + 1)^{p^+}$.

Since \eqref{72} holds for all balls $B ( \pvr _1, \tr ) \subset B( \pvr _0 , 10 r)$, by Proposition \ref{73}:
\begin{equation}\label{79}
\fint _{B(\pvr _0 , 2 r) } e^{C_5 C_6 |f -f _{B(\pvr _0 , 10r) }|} \, \dwvi  \leq C_7 ,
\end{equation}
where $C_5:=C_5 ( \consfi  , \consfii  , p^+ , \tilde \conspoin _1 , \conspoin _1 , \conspoin _2  , \conspoin _3 , \conspoin _4 , \conspoin _{\wghtzi} , \conspoin _0 , N _0 , \vinf _0 ,  \cpoin _2, n ,  s )>0$, $C_7:=C_7( \conspoin _1)>0$ and 
$$
C_6 := \left\{\left[ (p^- _{20r} -1 )^{-1} +1 \right]( \| u  \|_{s , B(\pvr _0 , 20r) , \wghtvi } +1)^{\conspoin _3} (\consman _1 + 1)^{p^+} \right\}^{-1}. 
$$

Using \eqref{79}, we can conclude that
$$
\begin{aligned}
& \left(\fint _{B(\pvr _0 , 2r) } e^{C_5 C_6  f } \, \dwvi  \right) \left(\fint _{B(\pvr _0 , 2r) } e^{-C_5 C_6  f } \, \dwvi  \right) \\
& \quad \ = \left(\fint _{B(\pvr _0 , 2r) } e^{C_5 C_6 [ f - f _{B(\pvr _0 , 10r)  } ] } \, \dwvi  \right) \left(\fint _{B(\pvr _0 , 2r) } e^{-C_5 C_6 [  f - f _{B(\pvr _0 , 10r) } ]} \, \dwvi   \right)  \\
& \quad \ \leq  C_7 ^2
\end{aligned}
$$
which implies 
$$
\begin{aligned}
&\left(\fint _{B(\pvr _0 , 2r) } w _{\ell , \alpha}^{C_5 C_6} \, \dwvi  \right)^{\frac{1}{C_5 C_6}}  =\left(\fint _{B(\pvr _ 0 , 2r) } e^{C_5 C_6 f } \, \dwvi  \right)^{\frac{1}{C_5 C_6}} \\
& \quad \ \leq C_7 ^{\frac{2}{C_5 C_6}}\left(\fint _{B(\pvr _ 0 , 2r) } e^{-C_5  C_6 f } \, \dwvi  \right)^{\frac{-1}{C_5 C_6}}  \\
& \quad \ =C_7 ^{\frac{2}{C_5 C_6}} \left(\fint _{B(\pvr _0 , 2r) } w _{\ell , \alpha}^{-C_5 C_6} \, \dwvi  \right)^{\frac{-1}{C_5 C_6}}.
\end{aligned}
$$
Thus we conclude the proof of the lemma.\end{proof}


\section{Removable sets for Hölder continuous solutions} \label{123}

In this section, we prove the two main results of this paper.
\begin{definition} Let \((X,d)\) be a metric space, and fix \(s > 0\). For each \(\delta > 0\) and \(E \subset X\), define
\begin{gather*}
H_\delta^s(E):=\inf \left\{\sum_{j=1}^{\infty} \beta(s) \left[ \frac{\operatorname{diam}(C_j)}{2} \right]^s \:|\: E \subset \cup_{j=1}^{\infty} C_j, \operatorname{diam}(C_j) \leq \delta \right\},
\end{gather*}
where $\beta(s)=\frac{\pi^{\frac{s}{2}}}{\Gamma(\frac{s}{2}+1)}$ and $\Gamma(s)=\int_0^{\infty} e^{-t} t^{s-1} \dt$ is the usual Gamma function.

The \(s\)-Hausdorff measure of \(E\) is the number
\begin{gather*}
H^s(E):=\lim _{\delta \rightarrow 0} H_\delta^s(E)=\sup _{\delta>0} H_\delta^s(E).
\end{gather*}
\end{definition}

\subsection{Proof of Theorem \ref{81}}
\begin{lemma} \label{100}
Let $\{\ue\} \subset \lwo  $ be a family  of functions, and let  $\uu \in \lwo \cap  L_{\loc}^{\infty} (\Omega ) $ be an element. Suppose that \eqref{78} - \eqref{26} and property \ref{103} are satisfied.. Assume that   $r \leq \sigma<\rho \leq 2r < \rast / 10$ and $B(\pvr _0 , 20r) \subset \subset \psi (\tilde U)$.  Let  $|N| \leq N_0$ and  $1<\gamma<1_{\wghtz}$. There is $R_1>0$ such that $(c_2+1) (p-1)\leq 1$ in $\Omega \backslash B(0_{\Omega} , R_1)$, where $c_2=c_2 ( \consfi  , \consfii  , p^+ , \tilde \conspoin _1 , \conspoin _1 , \conspoin _2 , \conspoin _3 , \conspoin _4 , \conspoin _{\wghtzi} , \conspoin _0 , N_0 , $ $ \vinf _0 , \cpoin _2 , n , s )$   is as defined in Lemma \ref{82}.  Furthermore, we have
\begin{enumerate}[label=(\roman*)]
\item  \label{68}  Assume that property \ref{94} is satisfied and $\uu, \ue \in \kob$ for all $\ell$. If $u$ is a solution of the obstacle problem \eqref{4}  with the obstacle $\obsi \leq N$ in $B(\pvr _0 , 2r)$ and $\pvr _0 \in \Omega \backslash B(0_{\Omega} , R_1)$, then 
\begin{equation} \label{83}
\begin{aligned}
&\underset{  B(\pvr _0 , r )}{\operatorname{ess} \sup } \ \left\{ (\ue -N)^{+} \right\}\\
 & \quad \ \leq C_1^{\frac{1}{(p_{20r} ^- -1)^2} }\left\{ \fint _{B(\pvr _0 , 2r )} \left[ (\ue -N)^{+}+r \right] ^{c_2 c_3} \, \dwvi  \right\} ^{\frac{1}{c_2 c_3}} ,
\end{aligned}
\end{equation}
where {\footnotesize $C_1:= C_1( \consfi  , \consfii  , p^+ , \tilde \conspoin _1 , \conspoin _1 , \conspoin _2 , \conspoin _3 , \conspoin _4 , \conspoin _{\wghtz} ,\conspoin _{\wghtzi} , \conspoin _0 , \cpoin _2 , n , s , \gamma , 1_{\wghtz}, N_0 , \vinf _0 , \consman _1 , \| u  \|_{s , B(\pvr _0 , 20r) , \wghtvi } , \| u  \|_{s \gamma ^\prime , B(\pvr _0 , 20 r) , \wghtvi } )>0$,  $s>p^+_{20r} - p^-_{20r}$, $1<\gamma<1_{\wghtz}$,  and 
$$
c_3 := \left\{ \left[(p^- _{20r} -1 )^{-1}+1\right]( \| u  \|_{s , B(\pvr _0 , 20r) , \wghtvi } +1)^{\conspoin _3} (\consman _1+1)^{p^+} \right\}^{-1} . 
$$}

\item \label{69} Assume that \eqref{28} (of property \ref{94}) holds. If $\uu$ is a supersolution of \eqref{13}, then, for $\pvr _0 \in \Omega \backslash B(0_{\Omega} , R_1)$: 
\begin{equation*}
\begin{aligned}
&\underset{  B(\pvr _0 , r )}{\operatorname{ess} \sup } \ \left\{ |(\ue -N)^{-}| \right\}\\
&  \quad \ \leq C_ 2^{\frac{1}{(p_{20r} ^- -1)^2} }  \left\{ \fint _{B(\pvr _0 , 2r )} \left[ |(\ue -N)^{-} |+r \right] ^{c_2 c_3} \, \dwvi  \right\} ^{\frac{1}{c_2 c_3}} ,
\end{aligned}
\end{equation*}
where {\footnotesize  $C_2 := C_2 ( \consfi  , \consfii  , p^+ , \tilde \conspoin _1 , \conspoin _1 , \conspoin _2 , \conspoin _3 , \conspoin _4 , \conspoin _{\wghtz} , \conspoin _{\wghtzi} , \conspoin _0 , \cpoin _2 , n , s , \gamma , 1_{\wghtz}, N_0 , \vinf _0  , \consman _1 , \| u  \|_{s , B(\pvr _0 , 20r) , \wghtvi } , \| u  \|_{s \gamma ^\prime , B(\pvr _0 , 2 0r) , \wghtvi } )>0$.}
\end{enumerate}

In \ref{68} and \ref{69}, $p^{-}_{20r} := \inf _{B(\pvr _0 , 20r)} p$ and $p^{+}_{20r} := \sup _{B(\pvr _0 , 20r)} p$.
\end{lemma}

\begin{proof}
\ref{68} Choose $R_0 >0$ such that 
$$
(c_2 +1)( p -1 )< 1 \quad \text { in } \Omega \backslash B(0_{\Omega} , R_0), 
$$
 where $c_2=c_2( \consfi  , \consfii  , p^+ , \tilde \conspoin _1 , \conspoin _1 ,  \conspoin _2, \conspoin _3 , \conspoin _4 , \conspoin _{\wghtzi} , N_0 , \vinf _0 , \conspoin _0 , \cpoin _2 ,  n , s )>0$. Next, from Corollary \ref{127} and Proposition \ref{67} \ref{23}, we obtain \eqref{83}.
 
 Similarly to \ref{68}, we conclude \ref{69}.\end{proof}

\begin{lemma} \label{85}
Let $\{\ue\} \subset \lwo  $ be a family  of functions, and let  $\uu \in \lwo \cap L^{\infty} _{\loc}(\Omega)$ be an element. Assume that \eqref{78}–\eqref{64} and \eqref{28} (from property \ref{94}) are verified, with \( p \in C(\Omega) \) satisfying property \ref{103}. Let  $r \leq \sigma<\rho \leq 2r \leq \rast / 10$, $B(\pvr _0 , 20r) \subset \subset \psi (\tilde U)$, and $ 1<\gamma<1_{\wghtz}$.   Suppose that $\uu$ is a  supersolution of \eqref{13} with $\ue \geq N$ in $B(\pvr _0 , 2r)$ and  $0 \leq N \leq N_0$. There is $R_1>0$ such that $(c_2+1) (p-1)\leq 1$ in $\Omega \backslash B(0_{\Omega} , R_1)$, where $c_2=c_2 ( \consfi  , \consfii  , p^+ , \tilde \conspoin _1 , \conspoin _1 , \conspoin _2 , \conspoin _3 , \conspoin _4 , \conspoin _{\wghtzi} , \conspoin _0 , N_0 , \vinf _0 , \cpoin _2 , n , s )$ is as defined in Lemma \ref{82}. Furthermore, we have
\begin{equation} \label{87}
\begin{aligned}
& \left[\fint _{B(\pvr _0 , 2r )} (\ue -N+r)^{c_2 c_3} \, \dwvi  \right] ^{\frac{1}{c_2c_3}} \\
 & \quad  \ \leq   \left[ C_3(p^-_{20r}-1)\right]^{-\frac{1}{p^-_{20r}-1}} \underset{  B(\pvr _0 , r )}{\operatorname{ess} \inf } \ \left\{ \ue -N+r \right\} ,
\end{aligned}
\end{equation}
where {\footnotesize $C_3:= C_3( \consfi  , \consfii  , p^+ , \tilde \conspoin _1 , \conspoin _1 , \conspoin _2 , \conspoin _3 , \conspoin _4 , \conspoin _{\wghtz} , \conspoin _{\wghtzi} , \conspoin _0 , \cpoin _2 , n , s , \gamma , 1_{\wghtz}, N_0 , \vinf _0 , \consman _1 , \| u  \|_{s , B(\pvr _0 , 20r) , \wghtvi } , \| u  \|_{s \gamma ^\prime , B(\pvr _0 , 2 r) , \wghtvi } )\in (0,1)$, $s>p^+_{20r} - p^-_{20r}$, $1<\gamma<1_{\wghtz}$, and
$$
c_3 := \left\{ \left[ (p^- _{20r} -1)^{-1} +1 \right] ( \| u  \|_{s , B(\pvr _0 , 20r) , \wghtvi } +1)^{\conspoin _3} (\consman _1+1)^{p^+} \right\}^{-1}. 
$$
}
\end{lemma}
\begin{proof}
The proof follow from of  Lemma \ref{82} and Proposition \ref{84}, taking $s_0=c_2c_3$.\end{proof}

Remember that:

\medskip

\noindent {\bf Theorem \ref{81}.} {\it Let \( (\Omega, \fnf, \wghtvi) \) be a reversible Finsler manifold, where \( \wghtvi \) and \( p \) satisfy conditions \eqref{118}–\eqref{64}. Let \( K \) be a compact subset of \( \Omega \), and define \( K_\delta := \{\pvr \in \Omega \:|\: d(\pvr, K) < \delta\} \), where \( \overline{K_\delta} \subset \Omega \). Suppose that \( \vartheta \in \mathscr{K}_{\obsi}(\Omega) \) is a solution to the obstacle problem \eqref{4}, with the obstacle \( \obsi \in C(\Omega) \) satisfying
\begin{equation}\label{107}
|\obsi(\pvr) - \obsi(\pvry)| \leq \conrem \, d^\alpha(\pvr, \pvry), \quad \text { for all } \pvr \in K, \pvry \in \Omega,
\end{equation}
where \( 0 < \alpha \leq 1 \) and \( \conrem > 0 \).

Additionally, let \( \{ \vartheta_\ell \} \subset \mathscr{K}_{\obsi}(\Omega) \) be a family such that \( \{\vuu = \obsi\} \subset \{\vue = \obsi\} \) for all \( \ell \), and suppose that this family satisfies properties \ref{94} and \ref{93}, with \( p \in C(\Omega) \) satisfying property \ref{103}.

There exists $ R_1= R_1 ( \consfi  , \consfii  , p^+ , \tilde \conspoin _1 , \conspoin _1 , \conspoin _2  ,  \conspoin _3 , \conspoin _4 , \conspoin _{\wghtzi} , \conspoin _0 , \vinf _0 ,  n , s , \cpoin _1 , \inf _{B(\pvr , 5r)} \obsi , \sup _{B(\pvr , 5r)} \obsi)>0$, where $s>\sup _{K_\delta} p - \inf _{K_\delta} p$, such that $p-1\leq 1$ in $\Omega \backslash B(0_{\Omega} , R_1)$.  Furthermore, for every $\pvr \in K \cap( \Omega \backslash B(0_{\Omega} , R_1) )$ and $0<r< \frac{1}{41}\min \{\rast , \delta\}$, the following  holds:
\begin{align*}
&\mu_\ell(B(\pvr, r)) \\
& \quad \ \leq C _1 C_2^{\frac{\tilde{\conspoin} _1}{p^-_{ 41r} -1}}   (p^-_{41r}-1)^{-\tilde{\conspoin} _1} r^{\cpoin _1 -p(\pvr) + \alpha (p(\pvr) -1 )  
} + \Upsilon _1 ( \pvr , 5r) + \Upsilon _2 (\pvr , r), 
\end{align*}
if $B(\pvr , 41 r )\subset \subset  \psi (\tilde U)$, where $\psi : \tilde U \rightarrow \psi (\tilde U )\subset \Omega$ is a chart,  {\footnotesize $p^-_{41r} := \inf _{B(\pvr , 41r)}p$,    $C_i:= C_i( \consfi  , \consfii  , p^+ , \tilde \conspoin _1 , \conspoin _1 , \conspoin _2 , \conspoin _3 ,$ $ \conspoin _4 , \conspoin _{\wghtz} , \conspoin _{\wghtzi} , $ $\conspoin _0 , \conspoin , n , s , \gamma , 1_{\wghtz}, \cpoin _1 , \cpoin _2 , \alpha , \conrem , $ $\inf _{B(\pvr , 5r)} \obsi , \sup _{B(\pvr , 5r)} \obsi , \consman _1 ,  \vinf _0 , \| \vuu  \|_{s , B(\pvr  , 41r) , \wghtvi } , $ $ \| \vuu  \|_{s \gamma ^\prime , B(\pvr  , 41r) , \wghtvi } )>0$,  for $i=1,2$, and $1<\gamma<1_{\wghtz}$.  Also, \(\consman_1\) is a constant (depending on \(\psi\) and \(\fnf\)) given in \eqref{114}, and $\rast$ is defined in Lemma \ref{128} \ref{76}.}}

\medskip

\begin{proof} {\it Steep 1.} \ Since \(\vuu \) is a solution to the obstacle problem \eqref{4} in \(\mathscr{K}_{\obsi}(\Omega)\), it follows that \(\vuu\) is a supersolution of \eqref{13}. Moreover, since the obstacle \(\obsi \in C(\Omega)\), Proposition \ref{95} ensures that \(\vuu \) is continuous.

If $B(\pvr , r) \cap\{\pvry \in \Omega \:|\: \vartheta(\pvry)=\obsi(\pvry)\}=\emptyset$, by Proposition \ref{95} and property \ref{93} we have
\begin{align*}
& \int _{\Omega} \mathscr{A}( \cdot , \grad  \vue ) \multg   \deriv  \zeta \, \dwvi  = \int _{\{ \vuu > \obsi\}} \left[\mathscr{A}( \cdot , \grad  \vue )  - \mathscr{A}( \cdot , \grad  \vuu )  \right]\multg  \deriv  \zeta  \, \dwvi  \leq \Upsilon  _2 (\pvr , r),
\end{align*} 
where $\zeta \in C^{\infty} _0 ( B(\pvr , r) )$ and $0\leq \zeta \leq 1$.

If \( B(\pvr, r) \cap \{ \pvry \in \Omega \mid \vuu(\pvry) = \obsi(\pvry) \} \neq \emptyset \), then let \( \pvr_0 \in B(\pvr, r) \cap \{ \pvry \in \Omega \mid \vuu(\pvry) = \obsi(\pvry) \} \). As a consequence of the hypotheses, \( \pvr_0 \in B(\pvr, r) \cap \{ \pvry \in \Omega \mid \vue(\pvry) = \obsi(\pvry) \} \). Then, 
\begin{equation} \label{102}
 \mu_\ell(B(\pvr, r)) \leq \mu_\ell(B(\pvr_0, 2r)).
 \end{equation}
Next, we proceed to estimate the right hand side of \eqref{102}.

There exist points \( \pvr_1, \pvr_2 \in B(\pvr_0, 8r) \) such that \( \sup_{B(\pvr_0, 8r)} \obsi = \obsi(\pvr_1) \) and \( \inf_{B(\pvr_0, 8r)} \obsi = \obsi(\pvr_2) \). Then, \eqref{107} implies that
\begin{equation} \label{96}
\begin{aligned}
&\sup _{ B(\pvr _0, 8 r)} \obsi -\inf _{ B(\pvr _0, 8 r)} \obsi   =\obsi (\pvr _1)-\obsi (\pvr )+\obsi (\pvr )-\obsi (\pvr _2) \\
&\quad \ \leq \conrem d ^{\alpha}(\pvr _1, \pvr )+\conrem  d ^{\alpha}(\pvr , \pvr _2) \\
&\quad \  \leq \conrem 2^\alpha \left( d ^\alpha (\pvr _1,\pvr _0) + d ^\alpha (\pvr _0 , \pvr ) +d^\alpha (\pvr , \pvr _0)+ d^\alpha (\pvr _0 , \pvr _2)\right) \\
&\quad \ \leq \conrem 2^\alpha \left[ (8 r)^\alpha+r^\alpha+r^\alpha+(8 r)^\alpha \right]  \leq C_4 r^\alpha ,
\end{aligned}
\end{equation}
where $C_4 := 2^{\alpha +1}\conrem (8^{\alpha} + 2)$.

We define,
$$
\begin{aligned}
& \alpha_1=\sup _{B(\pvr _0, 4 r)} \obsi -\inf _{ B(\pvr _0, 4 r)} \obsi +\inf _{ B(\pvr _0, 4 r)} \vue , \\
& \alpha_2=\sup _{ B(\pvr _0, 4 r)} \obsi -\inf _{ B(\pvr _0, 4 r)} \obsi -\inf _{ B(\pvr _0, 4 r)} \vue  .
\end{aligned}
$$


We have
\begin{equation}\label{89}
\inf_{B(\pvr_0, 4r)} \obsi \leq \inf_{B(\pvr_0, 4r)} \vue \leq \vue(\pvr_0) = \obsi(\pvr_0) \leq \sup_{B(\pvr_0, 4r)} \obsi.
\end{equation}

Thus, we obtain
$$
|\alpha_1| \leq 2 \left(\sup_{B(\pvr  , 5r)} \obsi \right)^+ - \inf_{B(\pvr  , 5r)} \obsi.
$$
Furthermore, since \(\obsi \leq \alpha_1\) in \(B(\pvr_0, 4r)\), we take \(N = \alpha_1\) and \(N_0 = 2 (\sup_{B(\pvr  , 5r)} \obsi)^+ - \inf_{B(\pvr  , 5r)} \obsi\) in  \eqref{83}. Then,
$$
\sup _{ B(\pvr _0, 2 r)}\left\{ (\vue -\alpha_1)^{+} \right\} \leq C_1^{\frac{1}{(p^-_{40r} -1)^2}} \left\{ \fint_{B(\pvr _0, 4 r)} \left[ (\vue  -\alpha_1)^{+} + 2r \right]^{c_2 c_3} \, \dwvi  \right\}^{\frac{1}{c_2 c_3}},
$$
where $p^-_{40r} := \inf _{B(\pvr _0, 40r)}p$.

Since
$$
\begin{aligned}
&\vue -\alpha_1  =\vue -\left(\sup _{ B(\pvr _0, 4 r)} \obsi -\inf _{ B(\pvr _0, 4 r)} \obsi \right)-\inf _{ B(\pvr _0, 4 r)} \vue  \\
& \quad \ \leq \vue + \left( \sup _{ B(\pvr _0, 4 r)} \obsi -\inf _{ B(\pvr _0, 4 r)} \obsi \right)-\inf _{ B(\pvr _0, 4 r)} \vue  =\vue +\alpha_2
\end{aligned}
$$
and $\vue  + \alpha_2 \geq 0 $ in $B(\pvr _0, 4 r)$, we have
\begin{equation} \label{88}
\sup _{ B(\pvr _0, 2 r)}\left\{ (\vue  -\alpha_1)^{+}\right\} \leq C_1^{\frac{1}{(p^-_{40r} -1)^2}} \left[ \fint_{B(\pvr _0, 4 r)} \left( \vue  + \alpha _2 + 2r \right)^{c_2 c_3} \, \dwvi  \right]^{\frac{1}{c_2 c_3}}.
\end{equation}

If $\alpha_2 \geq 0$, then $\vue +\alpha_2$ is a nonnegative supersolution of \eqref{13} in $B(\pvr _0, 4 r)$, then we take $N=N_0=0$ in \eqref{87}. If $\alpha_2<0$, we take $N=-\alpha_2$ and $N_0=\sup _{ B(\pvr  , 5r) } \obsi $ in \eqref{87}. Then, from \eqref{88} and \eqref{89},
$$
\begin{aligned}
&\sup _{ B(\pvr _0, 2 r)}\left\{ (\vue  -\alpha_1)^{+}\right\}  \leq  C_5 \inf _{B(\pvr _0, 2 r)} \{ \vue  +\alpha_2 +2r\} \\
&\quad \ \leq C_5\left( \vue (\pvr _0)+\sup _{ B(\pvr _0, 4 r)} \obsi  - \inf _{ B(\pvr _0, 4 r)} \obsi -\inf _{ B(\pvr _0, 4 r)} \vue  + 2r \right) \\
& \quad \ \leq 2C_5\left(\sup _{ B(\pvr _0, 4 r)} \obsi -\inf _{ B(\pvr _0, 4 r)} \obsi + r\right),
\end{aligned}
$$
where $C_5:= C_1^{\frac{1}{(p^-_{40r} -1)^2}}   \left[ C_3(p^-_{40r}-1)\right]^{-\frac{1}{p^-_{40r}-1}} $,  and  $C_i:= C_i( \consfi  , \consfii  , p^+ , \tilde \conspoin _1 , \conspoin _1 , \conspoin _2 , \conspoin _3 , \conspoin _4 , \conspoin _{\wghtzi} , \conspoin _0 , \cpoin _2 , n , s , \gamma , 1_{\wghtz}, $ $ \inf _{B(\pvr , 5r) } \obsi , \sup _{B(\pvr , 5r)} \obsi , \vinf _0 , \consman _1 , \| \vuu  \|_{s , B(\pvr _0 , 40r) , \wghtvi } , \| \vuu  \|_{s \gamma ^\prime , B(\pvr _0 , 40r) , \wghtvi } )>0$,  $i=1,3$, $C_1 , C_3^{-1} >1$,  $s>p^+_{40r} - p^-_{40r}$, $1<\gamma<1_{\wghtz}$. Hence,
$$
\begin{aligned}
& \sup _{\pvr \in  B(\pvr _0, 2 r)} \left\{ \left(\vue (\pvr )-\sup _{ B(\pvr _0, 4 r)} \obsi  +\inf _{ B(\pvr _0, 4 r)} \obsi -\inf _{ B(\pvr _0, 4 r)} \vue  \right)^{+} \right\} \\
& \quad \ \leq 2 C_5\left(\sup _{ B(\pvr _0, 4 r)} \obsi -\inf _{ B(\pvr _0, 4 r)} \obsi + r\right) .
\end{aligned}
$$

So,
\begin{equation}\label{97}
\sup _{ B(\pvr _0, 2 r)} \vue  - \inf _{ B(\pvr _0, 2 r)} \vue  \leq 3 C_5\left(\sup _{ B(\pvr _0, 4 r)} \obsi -\inf _{ B(\pvr _0, 4 r)} \obsi + r\right) .
\end{equation}
 
 $ $
 
{\it Steep 2.} \ Let $\eta \in C_0^{\infty}(B(\pvr _0, 4 r))$ with $0<\eta \leq 1$, $\eta=1$ in  $B(\pvr _0, 2 r)$, and $|\deriv  \eta|_{\fnfs } \leq \consman _0 \consman _1 (2r)^{-1}$, as stated in Lemma \ref{3}.

Define $p^+_{4r} := \sup _{B(\pvr _0, 4r)}p$. If \( \zeta \in C_0^\infty(B(\pvr_0, 2r)) \) with \( 0 \leq \zeta \leq 1 \), we have 
\begin{align*}
&\int _{B(\pvr , 5 r)}  \mathscr{A}(\cdot , \grad  \vue ) \multg  \deriv  \left(\eta ^{p^+_{4r}}- \zeta\right) \, \dwvi\\
&\quad \ = \int _{B(\pvr , 5r)} \left[ \mathscr{A}(\cdot , \grad  \vue ) - \mathscr{A}(\cdot , \grad  \vuu ) \right]\multg  \deriv  \left(\eta ^{p^+_{4r}}- \zeta \right) \, \dwvi + \int _{B(\pvr , 5r)}  \mathscr{A}(\cdot , \grad  \vuu ) \multg  \deriv  \left(\eta ^{p^+_{4r}}- \zeta \right) \, \dwvi\\
&\quad \ \geq - \Upsilon _1 ( \pvr , 5r), 
\end{align*}
due to property \ref{93}, the fact that \( \vuu \) is a supersolution of \eqref{13}, and  \( \eta ^{p^+_{4r}} - \zeta \geq 0 \) in \( B(\pvr_0, 4r) \). 

Using property \ref{21}, we get:
\begin{align}
& \mu _\ell (B(\pvr _0, 2 r))   \leq p^+_{4r} \int_{B(\pvr _0, 4 r)} \eta^{p_{4 r}^{+}-1} \mathscr{A}(\cdot , \grad  \vue ) \multg  \deriv  \eta \, \dwvi  + \Upsilon _1 ( \pvr , 5r)\label{91}\\
&\quad \  \leq p^+_{4r} \consfii  \int_{B(\pvr _0, 4 r)} \eta^{p_{4 r}^{+}-1}|\grad  \vue |_{\fnf} ^{p-1} |\deriv  \eta|_{\fnfs } \, \dwvi   + \Upsilon _1 (\pvr , 5r) \nonumber  \\
&\quad \   \leq 2 p^+_{4r} \consfii \||\grad  \vue |^{p-1} _{\fnf} \eta^{p_{4 r}^{+}-1}\|_{p/(p-1), B(\pvr _0, 4 r) , \wghtv } \| |\deriv  \eta| _{\fnfs }\|_{p, B(\pvr _0, 4 r) , \wghtv } + \Upsilon _1 (\pvr , 5r) \nonumber \\
& \quad \  \leq  2 p^+_{4r} \consfii  \max \left\{\left(\int_{B(\pvr _0, 4 r)}|\grad  \vue |_{\fnf}^{p} \eta^{p_{4 r}^{+}} \, \dwvi  \right)^{\frac{p^{+} _{4r} -1}{p^{+} _{4r} }} , \left(\int_{B(\pvr _0, 4 r)}|\grad  \vue |_{\fnf} ^{p} \eta^{p_{4r}^{+}} \, \dwvi  \right)^{\frac{p^{-} _{4r} -1}{p^{-}  _{4r} }} \right\}  \nonumber\\
& \quad \quad \ \cdot \max \left\{ \left(\int_{B(\pvr _0, 4 r)}|\deriv  \eta|_{\fnfs }^{p} \, \dwvi  \right)^{\frac{1}{p^{+} _{4r} }} , \left(\int_{B(\pvr _0, 4 r)}|\deriv  \eta|_{\fnfs }^{p} \, \dwvi  \right)^{\frac{1}{p^{-} _{4r}} } \right\} + \Upsilon _1 (\pvr , 5r) . \label{92}
\end{align}


We take  $N=\sup _{ B(\pvr _0, 4 r)} \vue $, $h=0$, and $\tr = 4r$ in \eqref{90}. By \eqref{78}, \eqref{97}, \eqref{96}, and \eqref{38}, we obtain
\begin{align}
& \int_{B(\pvr _0, 4r) }|\grad  \vue | _{\fnf}^{p} \eta  ^{p_{4r } ^{+}} \, \dwvi  \leq C_6  \int_{B(\pvr _0, 4r)}  \left( \sup _{ B(\pvr _0, 4r)} \vue  - \vue  \right)  ^{p} |\deriv \eta  |_{\fnfs }^{p} \, \dwvi     + C_6 \int _{ B(\pvr _0, 4r) } \eta ^{p^+_{4r}} \, \dwvi  \nonumber \\
& \quad \ \leq C _6 \int_{B(\pvr _0, 4 r)}\left( \sup _{ B(\pvr _0, 4 r)} \vue - \inf _{ B(\pvr _0, 4r)} \vue \right)^{p} \left( \frac{\consman _0 \consman _1}{2r} \right)^{p} \, \dwvi  +C_6\conspoin (4r)^{\cpoin _1} \nonumber\\
& \quad \ \leq C _6\int_{B(\pvr _0, 4 r)}  (3 C_5)^p\left( \sup _{ B(\pvr _0, 8 r)} \obsi -\inf _{ B(\pvr _0, 8 r)} \obsi + r \right) ^{p} \left( \frac{\consman _0 \consman _1}{2r} \right)^{p} \, \dwvi  + C_6\conspoin (4r)^{\cpoin _1} \nonumber \\
& \quad \ \leq C _6 \int_{B(\pvr _0, 4 r)}  (3 C_5)^p\left( C_4 r^{\alpha} + r \right) ^{p} \left( \frac{\consman _0 \consman _1}{2r} \right)^{p} \, \dwvi  + C_6\conspoin (4r)^{\cpoin _1} \nonumber \\
& \quad \ \leq C _7 C_5 ^{p^+_{4r}} r^{\cpoin _1-(1-\alpha) p^{-} _{4r}} , \label{98}
\end{align}
where $C_6:= C_6(\consfi  , \consfii  , p^+)$ and $C_7:=C_7 (\consfi  , \consfii  , p^+, n , \consman _1 , \conspoin , \conspoin _4 , \cpoin _1 , \alpha , \conrem)>0$.

From \eqref{92}, \eqref{98},  and \eqref{78}, we conclude
$$
\begin{aligned}
&\mu _\ell (B(\pvr _0, 2r))  \\
&\quad \ \leq 2 p^+ \consfii  (2^{-1}\consman _0 \consman _1+1)^{p^+ _{4r}/ p^- _{4r}} \max \left\{ \left[ C _7 C_5 ^{p^+ _{4r}} r^{\cpoin _1-(1-\alpha) p^- _{4r} } \right]^{\frac{p^{+} _{4r} -1}{p^{+} _{4r}}} , \left[  C _7 C_5 ^{p^+ _{4r}} r^{\cpoin _1-(1-\alpha) p^- _{4r} }    \right]^{\frac{p^{-} _{4r} -1}{p^{-} _{4r}}} \right\} \\
& \quad \quad \ \cdot \max \left\{ \left[ \conspoin (4r)^{\cpoin _1} \conspoin _4 r^{-p^- _{4r}} \right]^{\frac{1}{p^{+} _{4r} }} , \left[ \conspoin (4r)^{\cpoin _1} \conspoin _4 r^{-p^- _{4r}}\right]^{\frac{1}{p^{-} _{4r}} } \right\} + \Upsilon _1 (\pvr , 5r),
\end{aligned}
$$
since \eqref{38} holds. Additionally, from \eqref{74}, we have
$$
\begin{aligned}
&\quad \ \leq C _8 C_1^{\frac{\tilde{\conspoin} _1}{p^-_{40r} -1}}   \left[ C_3(p^-_{40r}-1)\right]^{-\tilde{\conspoin} _1} r^{\cpoin _1 -p^-_{4r} + \alpha (p^-_{4r} -1 )- \frac{p^+_{4r} - p^- _{4r}}{p^+_{4r}} \max \left\{\frac{\cpoin _1}{p^-_{4r}} , 1 \right\}  } + \Upsilon _1 (\pvr , 5r)\\
&\quad \ \leq C _9 C_1^{\frac{\tilde{\conspoin} _1}{p^-_{\pvr , 41r} -1}}   \left[ C_3(p^-_{\pvr ,41r}-1)\right]^{-\tilde{\conspoin} _1} r^{\cpoin _1 -p(\pvr) + \alpha (p(\pvr) -1 ) } + \Upsilon _1 (\pvr , 5r),
\end{aligned}
$$
where  $p^-_{\pvr , 41r} := \inf _{B(\pvr , 41r)}p$ and  $C_i:=C_i (\consfi  , \consfii  , p^+, n , \consman _1 , \conspoin , \conspoin  _0 , \conspoin _3 , \conspoin _4 , \cpoin _1 , \cpoin _2 , \alpha , \conrem)>0$ for $i=8,9$.

Therefore, from \eqref{102}, the proof of the theorem is complete.  \end{proof}

For every open set $V\subset \subset \Omega$, we write
$$
\mu (V):=\sup \left\{\int_V  \mathscr{A}(\cdot , \grad  \vuu) \multg  \deriv  \zeta  \, \dwvi  \:|\: 0 \leq \zeta \leq 1 \text { and } \zeta \in C_0^{\infty}(V)\right\} .
$$

\begin{corollary} \label{106} Let \( K \) be a compact subset of \( \Omega \), and define \( K_\delta := \{\pvr \in \Omega \:|\: d(\pvr, K) < \delta\} \) such that \( \overline{K_\delta} \subset \Omega \). Suppose that \( \vartheta \in \mathscr{K}_{\obsi}(\Omega) \) is a solution to the obstacle problem \eqref{4}. Additionally, assume that the obstacle \( \obsi \in C(\Omega) \)  satisfies \eqref{107}.

For every $\pvr \in K$ and $0<r< \frac{1}{41}\min \{\rast , \delta\}$, the following inequality holds:
$$
\mu (B(\pvr, r)) \leq C   r^{\cpoin _1 -p(\pvr) + \alpha (p(\pvr) -1 )},
$$
where $B(\pvr , 41r)  \subset \subset \psi (\tilde U)$ and {\footnotesize  $C:= C( \consfi  , \consfii  , p^- _{K_\delta} , p^+ _{K_\delta} , \tilde \conspoin _1 , \conspoin _1 , \conspoin _2 ,  \conspoin _3 , \conspoin _4 , \conspoin _{\wghtzi} , \conspoin _0 , \conspoin , n , s , \gamma , 1_{\wghtz}, \cpoin _1 , \cpoin _2 , \alpha ,   \conrem  , \consman _1 , \| \vuu  \|_{s , B(\pvr  , 41r) ,  \wghtvi } , $  $ \| \vuu  \|_{s \gamma ^\prime , B(\pvr  , 41r) , \wghtvi },  \inf _{B(\pvr  , 5r) } \obsi ,  \sup _{B(\pvr  , 5r) } \obsi ,   )>0$. Here, $s> p^{+}_{41r} - p^-_{41r}$ $1<\gamma<1_{\wghtz}$, $p^-_{K_\delta}=\inf _{K_\delta} p$ and $p^+_{K_\delta}=\sup _{K_\delta} p$.}
\end{corollary}

\subsection{Proof of Theorem \ref{104}}

Recall that:

\medskip

\noindent {\bf Theorem \ref{104}.} {\it  Let \( (\Omega, \fnf, \wghtvi) \) be a reversible Finsler manifold, where \( \wghtvi \) and \( p \) satisfy conditions \eqref{118}–\eqref{120}. Let $\singset \subset \Omega$ be a closed set. Assume that $u \in C(\Omega)$ is a solution of \eqref{13} in $\Omega \backslash \singset$, and there exists $\alpha \in (0,1]$ such that for any $\pvr \in \singset$ and $\pvry \in \Omega$:
$$
|u(\pvr)-u(\pvry)| \leq Cd^\alpha (\pvr , \pvry).
$$

If for each compact subset $K$ of $\singset$, the $\cpoin _1 -p^+_{K} + \alpha (p^+_{K} -1 )$-Hausdorff measure of $K$ is zero, then $u$ is a solution of \eqref{13} in $\Omega$.}

\medskip

\begin{proof}
{\it Steep 1.} Fix open sets $\Omega_0, \hat{\Omega} \subset \Omega$ such that $\Omega_0 \subset\subset \hat{\Omega} \subset\subset \Omega$. From Propositions \ref{105} and \ref{95}, there exists a unique continuous solution $u _1 \in \mathscr{K}_{u}(\hat \Omega )$ to the obstacle problem \eqref{4}. 
 
For an open set \( V \subset \Omega_0 \), define 
$$
\mu  _1 (V):=\sup \left\{\int_V  \mathscr{A}(\cdot , \grad  u _1) \multg  \deriv  \zeta  \, \dwvi  \:|\: 0 \leq \zeta \leq 1 \text { and } \zeta \in C_0^{\infty}(V)\right\} .
$$ 
Also, for an arbitrary set \( E \subset \Omega _0 \), define
\[
\mu_1(E) := \inf \left\{ \mu_1(V) \:|\: V \subset \Omega_0 \text{ is an open set such that } E \subset V \right\}.
\]

Let $K$ be a compact set in $\singset \cap \Omega _0$. From Corollary \ref{106}, for any $\pvr  \in K$ and $0<r< \frac{1}{41}\min \{\rast , \delta\}$, we have
$$
\mu(B(\pvr , r)) \leq C r^{\cpoin _1 -p^+_K + \alpha (p^+_K -1 )} .
$$

Since the $\cpoin _1 -p^+_K + \alpha (p^+_K -1 )$-Hausdorff measure of $K$ is zero,  for any $\epsilon>0$, there exists $\delta _0>0$ such that for any $0 < \delta _1<\delta_0$,
$$
H_{\delta_1}^s(K)<\epsilon ,
$$
where $s=\cpoin _1 -p^+_K + \alpha (p^+_K -1 )$.

Taking $\delta _1$ small enough, there exists a family of sets $C_j^{\delta _1}$ such that $K \subset \cup_{j=1}^{\infty} C_j^{\delta _1}$, $\operatorname{diam}(C_j^{\delta _1})<\delta _1$, and
$$
H_{\delta_1} ^s(K) \leq \sum_{j=1}^{\infty} \beta(s) \left[ \frac{\operatorname{diam}(C_j^{\delta _1})}{2} \right]^s<\epsilon .
$$
Hence,
$$
\sum_{j=1}^{\infty} \left[\operatorname{diam} (C_j^{\delta _1}) \right]^s<\left(\frac{2^s}{\beta(s)} \right) \epsilon .
$$

Take a family of balls $\{B(\pvr _j, r_j)\}$ such that $C_j^{\delta _1} \subset B(\pvr _j, r_j)$, where $\pvr _j \in K$ and $ r_j=\operatorname{diam}(C_j^{\delta _1})$. Then
$$
\mu _1 (K) \leq \sum_{i=1}^{\infty} \mu _1 (C_j^{\delta _1}) \leq \sum_{j =1}^{\infty} \mu _1 (B(\pvr _j, r_j)) \leq C \sum_{j=1}^{\infty} r_j^{\cpoin _1 -p^+_K + \alpha (p^+_K -1 )} \leq C\left(\frac{2^s}{\beta(s)}\right) \epsilon,
$$
where the constant $C$ is given in Corollary \ref{106}.

Since $\epsilon$ is arbitrary, we conclude that $\mu _1 (K)=0$. Hence, $\mu _1(\singset \cap \Omega _0)=0$.\\

{\it Steep 2.} Next, we prove that $\mu _1(\Omega _0 \backslash \singset)=0$.

Let $\epsilon>0$ and $\zeta  \in C_0^{\infty}(\Omega _0 \backslash \singset)$ with $\zeta \geq 0$. Define $\zeta _{\epsilon}=\min \{\zeta , \frac{u_1 - u }{\epsilon}\}$, so that  $\zeta _{\epsilon} \in W_0^{1, p(\pvr )}(\Omega _0 \backslash \singset ; \wghtv)$.

Define $D_1:=\{\pvr  \in \Omega _0 \:|\: u_1(\pvr )>u(\pvr )\}$ and  $D_2:=\{\pvr  \in \Omega _0 \:|\: u_1(\pvr )=u(\pvr )\}$. Since the obstacle $u$ is continuous, by Proposition \ref{95}, we know that $u_1$ is a solution of \eqref{13} in $D_1$. Then,
\begin{equation}\label{108}
\begin{aligned}
& \int_{\Omega _0 \backslash \singset}  \mathscr{A}(\cdot , \grad  u_1) \multg   \deriv  \zeta _{\epsilon }  \, \dwvi   \\
& \quad \ =\int_{D_1}  \mathscr{A}(\cdot , \grad  u_1) \multg  \deriv  \zeta _{\epsilon }  \, \dwvi  +\int_{D_2} \mathscr{A}(\cdot , \grad  u_1) \multg  \deriv  \zeta _{\epsilon } \, \dwvi  =0 .
\end{aligned}
\end{equation}


Additionally, since $u$ is a solution of \eqref{13} in $\Omega _0 \backslash \singset$, we have
\begin{equation}\label{109}
\int_{\Omega _0 \backslash \singset}  \mathscr{A}(\cdot , \grad  u) \multg  \deriv \zeta _{\epsilon }  \, \dwvi  =0 .
\end{equation}

By \eqref{108} and \eqref{109}, 
\begin{equation} \label{110}
\int_{\Omega _0 \backslash \singset} \left[ \mathscr{A}(\cdot , \grad  u_1)-\mathscr{A}(\cdot , \grad  u) \right] \multg  \deriv  \zeta _{\epsilon }  \, \dwvi  =0 .
\end{equation}

Define $(\Omega _0 \backslash \singset)_1=\{\pvr  \in \Omega _0 \backslash \singset \:|\:   \zeta (\pvr ) \leq (u_1(\pvr )-u(\pvr ))/\epsilon \}$ and $(\Omega _0 \backslash \singset)_2=\{\pvr  \in \Omega _0 \backslash \singset \:|\:   \zeta (\pvr )> (u_1(\pvr )-u(\pvr ) ) / \epsilon\}$. From \eqref{110} and \ref{17}, we have
$$
 \int_{(\Omega _0 \backslash \singset)_1}  \left[ \mathscr{A}(\cdot , \grad  u_1)-\mathscr{A}(\cdot , \grad  u) \right] \multg  \deriv \zeta  \, \dwvi  =-\frac{1}{\epsilon } \int_{(\Omega _0 \backslash \singset)_2} \left[  \mathscr{A}(\cdot , \grad  u_1)-\mathscr{A}(\cdot , \grad  u) \right] \multg   \deriv (u_1-u)  \, \dwvi  \leq 0.
$$
Taking the limit as $\epsilon  \rightarrow 0$, we obtain
$$
\int_{(\Omega _0 \backslash \singset)_1}  \left[ \mathscr{A}(\cdot , \grad  u_1)-\mathscr{A}(\cdot , \grad  u) \right] \multg  \deriv  \zeta  \, \dwvi   \leq 0 .
$$
Further, since
$$
\int_{\Omega _0 \backslash \singset}  \mathscr{A}(\cdot , \grad  u) \multg  \deriv  \zeta  \, \dwvi  =0,
$$
we conclude
$$
\int_{\Omega _0 \backslash \singset}  \mathscr{A}(\cdot , \grad  u_1) \multg  \deriv  \zeta  \, \dwvi   \leq 0,
$$
it implies that $\mu _1 (\Omega _0 \backslash \singset) \leq 0$, so $\mu _1 (\Omega _0 \backslash \singset)=0$.

From the  above, we conclude that $\mu _1 (\Omega _0)=0$. That is, for any $\zeta \in W_0^{1, p(\pvr )}(\Omega _0 ; \wghtv)$, we have
$$
\int_{\Omega _0} \mathscr{A}(\cdot , \grad  u_1) \multg  \deriv  \zeta \, \dwvi =0 .
$$

$ $

{\it Steep 3.} Let $u_2$ be a solution to the obstacle problem \eqref{111} in $\mathscr{K}_{-u}(\hat \Omega)$. Similarly to \(u_1\), we find that for all $\zeta \in C^{\infty} _0(\Omega _0)$,
$$
\int_{\Omega _0} \left[- \mathscr{A}(\cdot ,-\grad  u_2)\right] \multg  \deriv   \zeta \, \dwvi   =0 .
$$

 Next, we will prove that $u_1=-u_2=u$ a.e. in $\Omega _0$. Indeed, since $-u_2 \leq u \leq u_1$ a.e. in $\Omega _0$, $u_1-u \in$ $W_0^{1, p(\pvr )}(\Omega _0 ; \wghtv)$, and $u_2+u \in W_0^{1, p(\pvr )}(\Omega _0 ; \wghtv )$, it follows that  $u_1+u_2 \in W_0^{1, p(\pvr )}(\Omega _0 ; \wghtv  )$. Therefore,
\begin{gather*}
 \int_{\Omega _0}  \mathscr{A}(\cdot , \grad  u_1) \multg  \left( \deriv u_1+\deriv  u_2 \right) \, \dwvi  =0 \\
 \int_{\Omega _0} \mathscr{A}(\cdot ,-\grad  u_2) \multg  \left(  \deriv  u_1+\deriv  u_2 \right) \, \dwvi  =0 .
\end{gather*}

Thus,
$$
\int_{\Omega _0}  \left[ \mathscr{A}(\cdot , \grad  u_1)-\mathscr{A}(\cdot ,-\grad  u_2)\right] \multg  \left(  \deriv  u_1+\deriv  u_2 \right) \, \dwvi  =0 .
$$

Therefore, by \ref{17}, $\deriv (u_1+u_2)=0$ a.e. in $\Omega _0$.  This implies $u_1+u_2=0$ a.e. in $\Omega _0$, so $u_1=-u_2=u$ a.e. in $\Omega _0$. Therefore, $u$ is the solution to \eqref{13} in $\Omega _0$, meaning that for any $\zeta \in W_0^{1, p(\pvr )}(\Omega _0 ; \wghtv )$, we have
$$
\int_{\Omega _0} \mathscr{A}(\cdot , \grad  u) \multg  \deriv \zeta \, \dwvi  =0 .
$$

Thus, $u$ is a solution of \eqref{13} in $\Omega$, and the closed set $\singset$ is removable.\end{proof}

\vspace{1.5cm}

\noindent {\bf Author contributions:} All authors have contributed equally to this work for writing, review and editing. All authors have read and agreed to the published version of the manuscript.

\noindent {\bf Funding:} This work was supported by CNPq - Conselho Nacional de Desenvolvimento Científico e Tecnológico,  Grant 153232/2024-2.

\noindent {\bf Data Availibility:} No data was used for the research described in the article.

\noindent {\bf Declarations}

\noindent {\bf Conflict of interest:} The authors declare no conflict of interest.


\bibliographystyle{abbrv}
    
\bibliography{ref}

\end{document}